\documentclass{article}
\usepackage[utf8]{inputenc}
\usepackage{amsmath,amssymb,amsthm, bbold, xfrac, graphicx, tabu, tikz, framed, caption, subcaption, verbatim, url, upgreek, wasysym, bbm, dictsym, staves, physics, algorithmic, calligra, mathrsfs, breqn, placeins, hyperref}
\hypersetup{
    colorlinks=true,
    citecolor=blue,
    urlcolor=violet,
    linkcolor=red
}
\usepackage[margin=1 in]{geometry}

\newif\iflong
\longtrue
\newcommand{\LongVersionShortVersion}[2]{\iflong #1 \else #2 \fi}

\numberwithin{equation}{section}
\numberwithin{figure}{section}

\newtheorem{theorem}{Theorem}[section]
\newtheorem*{theorem*}{Theorem}
\newtheorem*{thm113}{\cite[Theorem 1.13]{ckw}}
\newtheorem*{thm117}{\cite[Theorem 1.17]{ckw2}}
\newtheorem{lemma}[theorem]{Lemma}
\newtheorem{prop}[theorem]{Proposition}
\newtheorem{corollary}[theorem]{Corollary}

\theoremstyle{definition}
\newtheorem{definition}[theorem]{Definition}
\newtheorem{example}[theorem]{Example}

\newtheorem{remark}[theorem]{Remark}
\newtheorem{assumption}[theorem]{Assumption}

\newcommand{\Var}{\mathrm{Var}}

\newcommand{\floor}[1]{\left\lfloor #1 \right\rfloor}

\newcommand{\supp}{\mathrm{supp}}
\newcommand{\diam}{\mathrm{diam}}

\DeclareMathOperator*{\esssup}{\,ess\, sup}
\DeclareMathOperator*{\essinf}{\,ess\, inf}

\newcommand{\diag}{\mathrm{diag}}
\newcommand{\cutoff}{\mathrm{cutoff}}
\newcommand{\death}{\partial}
\newcommand{\borel}{\mathscr{B}}
\newcommand{\mean}{\mathrm{mean}}
\newcommand{\rms}{\mathrm{rms}}

\newcommand{\VD}{\hyperref[VdDefinition]{\mathrm{VD}}}
\newcommand{\RVD}{\hyperref[RvdDefinition]{\mathrm{RVD}}}
\newcommand{\QRVD}{\hyperref[QrvdDefinition]{\mathrm{QRVD}}}
\newcommand{\HK}{\hyperref[HKphiDefinition]{\mathrm{HK}_\phi}}
\newcommand{\UHK}{\hyperref[HKphiDefinition]{\mathrm{UHK}_\phi}}
\newcommand{\LHK}{\hyperref[HKphiDefinition]{\mathrm{LHK}_\phi}}
\newcommand{\PHI}{\hyperref[PhiDefinition]{\mathrm{PHI}_\phi}}
\newcommand{\PHIplus}{\hyperref[PhiplusDefinition]{\mathrm{PHI}^+_\phi}}
\newcommand{\J}{\hyperref[JphiDefinition]{\mathrm{J}_\phi}}
\newcommand{\Jleq}{\hyperref[JphiDefinition]{\mathrm{J}_{\phi, \leq}}}
\newcommand{\Jgeq}{\hyperref[JphiDefinition]{\mathrm{J}_{\phi, \geq}}}
\newcommand{\E}{\hyperref[EphiDefinition]{\mathrm{E}_\phi}}
\newcommand{\Eleq}{\hyperref[EphiDefinition]{\mathrm{E}_{\phi, \leq}}}
\newcommand{\Egeq}{\hyperref[EphiDefinition]{\mathrm{E}_{\phi, \geq}}}
\newcommand{\QE}{\hyperref[QEphiDefinition]{\mathrm{QE}_\phi}}
\newcommand{\QEleq}{\hyperref[QEphiDefinition]{\mathrm{QE}_{\phi, \leq}}}
\newcommand{\QEgeq}{\hyperref[QEphiDefinition]{\mathrm{QE}_{\phi, \geq}}}
\newcommand{\CSJ}{\hyperref[CsjDefinition]{\mathrm{CSJ}_\phi}}
\newcommand{\SCSJ}{\hyperref[ScsjDefinition]{\mathrm{SCSJ}_\phi}}
\newcommand{\UHKD}{\hyperref[UhkdDefinition]{\mathrm{UHKD}_\phi}}
\newcommand{\NDL}{\hyperref[NdlDefinition]{\mathrm{NDL}_\phi}}
\newcommand{\UJS}{\hyperref[UjsDefinition]{\mathrm{UJS}}}
\newcommand{\PI}{\hyperref[PoincareDefinition]{\mathrm{PI}_\phi}}
\newcommand{\PHR}{\hyperref[PhrDefinition]{\mathrm{PHR}_\phi}}
\newcommand{\EHR}{\hyperref[EhrDefinition]{\mathrm{EHR}}}
\newcommand{\COND}{\mathrm{COND}}
\newcommand{\CHAR}{\mathrm{CHAR}}

\newcommand{\Fmax}{\mathcal{F}_{\max}}

\newcommand{\Fscrmax}{\mathscr{F}_{\max}}

\newcommand{\Dtilde}{\tilde{\mathcal{D}}}
\newcommand{\Etilde}{\tilde{\mathcal{E}}}

\newcommand{\Btilde}{\tilde{B}}
\newcommand{\Vtilde}{\tilde{V}}
\newcommand{\Jtilde}{\tilde{J}}

\newcommand{\Mhat}{\widehat{M}}
\newcommand{\Nhat}{\widehat{\mathcal{N}}}
\newcommand{\Mohat}{\widehat{M_0}}
\newcommand{\Jhat}{\widehat{J}}
\newcommand{\dhat}{\widehat{d}}
\newcommand{\Dhat}{\widehat{\mathcal{D}}}
\newcommand{\Ehat}{\widehat{\mathcal{E}}}
\newcommand{\Fhat}{\widehat{\mathcal{F}}}
\newcommand{\Fhatmax}{\widehat{\mathcal{F}}_{\max}}
\newcommand{\muhat}{\widehat{\mu}}
\newcommand{\Xhat}{\widehat{X}}
\newcommand{\phat}{\widehat{p}}
\newcommand{\qhat}{\widehat{q}}

\newcommand{\Vhat}{\widehat{V}}
\newcommand{\Bhat}{\widehat{B}}

\title{\textbf{Stability results for symmetric jump processes on metric measure spaces with atoms}}
\author{Jens Malmquist\\
\textit{Department of Mathematics, The University of British Columbia}\\
\textit{Vancouver, BC, Canada}\\
\textit{ORCID: 0000-0002-5439-2367}\\
\texttt{jens@math.ubc.ca}}
\date{}

\begin{document}
\maketitle
\begin{abstract}
Consider a symmetric Markovian jump process $\{X_t\}$ on a metric measure space $(M, d, \mu)$. Chen, Kumagai, and Wang recently showed that two-sided heat kernel estimates and the parabolic Harnack inequality are both stable under bounded perturbations of the jumping measure, assuming $(M, d, \mu)$ satisfies the volume-doubling and reverse-volume-doubling conditions. These results do not apply if $(M, d, \mu)$ is a graph (or more generally, if $M$ contains any atoms $x$ such that $\mu(x)>0$) because it is impossible for reverse-volume-doubling to hold on a space with atoms. We generalize the results of Chen, Kumagai, and Wang to a larger class of metric measure spaces, including all infinite graphs with volume-doubling. Our main tool is the construction of an ``auxiliary space" that smooths out the atoms. We show that many properties transfer from $(M, d, \mu)$ to the auxiliary space, and vice versa, including heat kernel estimates, the parabolic Harnack inequality, and their stable characterizations.
\end{abstract}

\providecommand{\classification}[1]
{
  \textbf{\textit{MSC 2020 Mathematics Subject Classification---}} #1
}

\providecommand{\keywords}[1]
{
  \textbf{\textit{Keywords---}} #1
}

\providecommand{\keywords}[1]
{
  \textbf{\textit{Keywords---}} #1
}

\providecommand{\funding}[1]
{
  \textbf{\textit{Funding---}} #1
}

\providecommand{\conflicts}[1]
{
  \textbf{\textit{Conflicts of interest/Competing interests---}} #1
}

\providecommand{\availabilityOfData}[1]
{
    \textbf{\textit{Availability of data and material---}} #1
}

\providecommand{\codeAvailability}[1]
{
  \textbf{\textit{Code availability---}} #1
}

\providecommand{\authorsContributions}[1]
{
  \textbf{\textit{Authors' contributions---}} #1
}

\classification{Primary 60J76, 60J25, 35K08; Secondary 31C25, 60J27, 60J74, 60J45}



\section*{Declarations}

\,

\funding{Not applicable}

\conflicts{Not applicable}

\availabilityOfData{Not applicable}

\codeAvailability{Not applicable}

\authorsContributions{Not applicable}

\pagebreak

\tableofcontents

\section{Introduction}

Let $(M, d)$ be a locally compact, separable metric space. Let $\mu$ be a positive Radon measure on $M$ with full support, such that $\mu(M) = \infty$. We refer to the triple $(M, d, \mu)$ as the \textit{metric measure space}.

Let $(\mathcal{E}, \mathcal{F})$ be a regular symmetric Dirichlet form on $L^2(M, \mu)$. By the Beurling-Deny formula \cite[Theorem 4.5.2]{fukush}, $\mathcal{E}$ can be decomposed into a strongly local part, a jump part, and a killing part. We assume that all but the jump part are identically $0$. This means that there exists a positive Radon measure $J$ on $M \times M \setminus \diag_M$ such that
\begin{equation} \label{energyDef}
    \mathcal{E}(f, g) = \int_{M \times M \setminus \diag_M} (f(x)-f(y))(g(x)-g(y)) \, J(dx, dy) \qquad\mbox{for $f, g \in \mathcal{F}$}.
\end{equation}
Here, $\diag_M := \{(x, x): x\in M\}$ denotes the diagonal of $M$. Throughout the paper, we will use the notation $\diag_E$ to refer to the set $\{(x, x) : x \in E\}$, for any set $E$. We refer to the symmetric Radon measure $J$ as the \textit{jumping measure}, and $(M, d, \mu, \mathcal{E}, \mathcal{F})$ as the \textit{metric measure Dirichlet space}. We let $\mathcal{E}(f) := \mathcal{E}(f, f)$ for all $f \in \mathcal{F}$, and we sometimes refer to $\mathcal{E}(f)$ as the \textit{energy} of $f$.

Let $X = \{X_t: t \geq 0; \mathbb{P}^x: x \in M \setminus \mathcal{N}\}$ be the $\mu$-symmetric Hunt process associated with $(\mathcal{E}, \mathcal{F})$, where $\mathcal{N}$ is a \textit{properly exceptional set}, meaning that $\mu(\mathcal{N})=0$ and $\mathbb{P}^x(\mbox{$X_t\in\mathcal{N}$ for some $t>0$}) = 0$ for all $x \in M \setminus \mathcal{N}$.
The Hunt process associated with a regular Dirichlet form is unique up to the properly exceptional set (see \cite[Theorem 4.2.8]{fukush}). Fix $X$ and $\mathcal{N}$, and let $M_0 := M\setminus \mathcal{N}$.
We refer to the process $\{X_t\}$ as the \textit{jump process}.

We say that a continuous, strictly increasing function $\phi : [0, \infty) \rightarrow [0,\infty)$ is \textit{of regular growth} if $\phi(0)=0$, $\phi(1)=1$, and there exist some constants $c_1, c_2 > 0$ and $\beta_2 \geq \beta_1 > 0$ such that
\begin{equation}\label{regularGrowth}
    c_1 \left(\frac{R}{r}\right)^{\beta_1} \leq \frac{\phi(R)}{\phi(r)} \leq c_2 \left(\frac{R}{r}\right)^{\beta_2} \qquad\mbox{for all $R \geq r > 0$}.
\end{equation}
From this point on, for the rest of the paper, let us fix a function $\phi$ of regular growth.

We are primarily interested in two conditions on the metric measure Dirichlet space $(M, d, \mu, \mathcal{E}, \mathcal{F})$, each of which may or may not be satisfied:
\begin{itemize}
    \item Condition $\HK$ (see Definition \ref{HKphiDefinition}), which gives two-sided estimates for the heat kernel of $\{X_t\}$.
    \item Condition $\PHI$ (see Definition \ref{PhiDefinition}): the parabolic Harnack inequality, which controls the growth of non-negative caloric functions.
\end{itemize}
Both of these conditions depend on $\phi$, which may be broadly thought of as the scaling function between space and time: it takes time on the order of $\phi(r)$ for the jump process to travel a distance of $r$.

The heat kernel estimates given by $\HK$ have tails that decay polynomially with respect to distance from the diagonal. These heavy tails are characteristic of jump processes, in contrast to diffusions, which often have Gaussian tails.

Let us give a bit of background to the significance of $\PHI$.
Harnack inequalities are the subject of significant research in probability, harmonic analysis, and partial differential equations. The earliest Harnack inequality, proved by Carl Gustav Axel von Harnack \cite{harnack}, was for harmonic functions on the plane: if $B(x_0, r)$ and $B(x_0, R)$ are balls in $\mathbb{R}^2$, with $R > r > 0$, and $u$ is a non-negative solution to the Laplace equation ($\Delta u = 0$) on $B(x_0, R)$, then
\begin{equation} \label{EarliestEhi}
    \sup_{B(x_0, r)} u \leq C \inf_{B(x_0, r)} u,
\end{equation}
where $C$ is a constant depending on $\frac{R}{r}$ but not on $u$.
Pini \cite{pini} and Hadamard \cite{hadamard} established an anologous inequality for caloric functions: if $R>r>0$, $t_4 > t_3 > t_2 > t_1 > 0$, and $u \in C^\infty((0, \infty) \times \mathbb{R}^d)$ is a non-negative solution to the heat equation ($\frac{\partial u}{\partial t} - \Delta u = 0$) on $(0, t_4) \times B(0, R)$, then
\begin{equation}\label{EarliestPhi}
    \sup_{B(0, r) \times [t_1, t_2]} u(t, x) \leq C \inf_{B(0, r) \times [t_3, t_4]} u(t, x),
\end{equation}
where $C$ is a constant depending on $d$, $r$, $R$, $t_1$, $t_2$, $t_3$, and $t_4$. Both of these inequalities have since been generalized to many other settings, and to other operators playing the role of $\Delta$. Generalizations of \eqref{EarliestEhi} are called \textit{elliptic} Harnack inequalities and generalizations of \eqref{EarliestPhi} are called \textit{parabolic} Harnack inequalities. For a more detailed introduction to the history and basic theory of Harnack inequalities, we refer to reader to \cite{kassman}.

A major use of Harnack inequalities is that they imply H\"{o}lder continuity for harmonic/caloric functions. 
This theory is well-developed in the case of diffusions, where Harnack inequalities are a central ingredient in the De Giorgi-Nash-Moser theory in harmonic analysis and partial differential equations. In the case of jump process, the theory is in its infancy, but \cite{ckw} establishes results of this type for jump processes.

Conveniently, the Harnack inequalities \eqref{EarliestEhi} and \eqref{EarliestPhi} are stable, in the sense that they are preserved when the operator playing the role of $\Delta$ is perturbed. Our goal is to show that in our setting of a jump process on a general metric measure space, under minimal assumptions about volume-growth in $(M, d, \mu)$, our conditions $\HK$ and $\PHI$ are stable in a similar sense.

Recall the jumping measure $J$ from \eqref{energyDef}. If $J'$ is another symmetric Radon measure on $M\times M\setminus\diag_M$, and if there exist constants $C \geq c > 0$ such that
\begin{equation*}
    c J(A) \leq J'(A) \leq CJ(A) \qquad\mbox{for all Borel-measurable $A \subseteq M\times M \setminus \diag_M$}
\end{equation*}
then $J'$ is called a \textit{bounded perturbation} of $J$.
Let $\COND$ be some condition on $(M, d, \mu, \mathcal{E}, \mathcal{F})$ (for example, $\COND=\HK$ or $\COND=\PHI$). We say that $\COND$ is \textit{stable} if whenever $J'$ is a bounded perturbation of $J$, and $(\mathcal{E}', \mathcal{F})$ is the pure-jump Dirichlet form given by
\begin{equation*}
    \mathcal{E}'(f, g) = \int_{M \times M \setminus \diag_M} (f(x)-f(y))(g(x)-g(y)) \, J'(dx, dy) \qquad\mbox{for $f, g \in \mathcal{F}$},
\end{equation*}
we have
\begin{equation*}
    \mbox{$\COND$ for $(M, d, \mu, \mathcal{E}, \mathcal{F})$} \Longleftrightarrow \mbox{$\COND$ for $(M, d, \mu, \mathcal{E}', \mathcal{F})$}.
\end{equation*}

There have been many stability results of the flavor we seek. Grigor'yan \cite{gr} and Saloff-Coste \cite{sa} independently showed that for diffusions of walk dimension $2$ on a smooth, geodesically complete Riemannian manifold, $\PHI$ is equivalent to Aronson type Gaussian heat kernel estimates, which are in turn equivalent to the volume-doubling condition $\VD$ (see Definition \ref{VdDefinition}) plus a Poincar\'e inequality. Since both $\VD$ and the Poincar\'e inequality are stable, this means that both $\PHI$ and the Aronson type heat kernel estimates are stable too. Sturm \cite{sturm1, sturm2} extended this result to symmetric diffusions on metric measure spaces, and Delmotte \cite{del} extended it to nearest-neighbor random walks on graphs (with edges of not-necessarily equal weight). Other subsequent generalizations (such as \cite{bb}, \cite{bbk}, and \cite{ab}) allow for the walk dimension to exceed $2$, but require a cut-off Sobolev inequality in addition to the Poincar\'e inequality.
The Poincar\'e and cut-off Sobolev inequalities are both stable under bounded perturbations, so each of \cite{gr}, \cite{sa}, \cite{sturm1}, \cite{sturm2}, \cite{del}, \cite{bb}, \cite{bbk}, and \cite{ab} provides a stable characterization of both $\PHI$ and its respective heat kernel estimates, for its respective context. (Note that these heat kernel estimates are different from our $\HK$, in that they are Gaussian.)

However, in this paper we are interested not in diffusions or nearest-neighbor random walks, but in jump processes.
Chen and Kumagai \cite{ck1} showed that for jump processes, if $M$ is an Ahlfors-regular $d$-set in $\mathbb{R}^n$, and $\phi(r)=r^\alpha$ for some $\alpha \in (0, 2)$, then $\HK$ is equivalent to the existence of constants $C \geq c>0$ such that
\begin{equation} \label{AhlforsRegJ}
    cd(x, y)^{-(d+\alpha)}\leq J(x, y) \leq C d(x, y)^{-(d+\alpha)} \qquad\mbox{for all $x, y \in M\times M\setminus\diag_M$}.
\end{equation}
Equation \eqref{AhlforsRegJ} is stable under bounded perturbations (if we allow the constants to change), so this is a stable 
characterization of $\HK$. There have been some generalizations of \cite{ck1}, such as \cite{ck2} and \cite{ck3}, but each of these requires the constant $\beta_2$ in \eqref{regularGrowth} to be less than $2$.
In two recent papers (\cite{ckw} and \cite{ckw2}), Chen, Kumagai, and Wang provide stable characterizations of both $\HK$ and $\PHI$, assuming $(M, d, \mu)$ satisfies the volume doubling condition $\VD$ (see Definition \ref{VdDefinition}) and the reverse volume doubling condition $\RVD$ (see Definition \ref{RvdDefinition}), and $\phi$ is of regular growth. Notably, \cite{ckw} answered a previously open problem by showing that $\HK$ is stable for $\alpha$-stable-like processes, even when $\alpha \geq 2$.

Unfortunately, $\RVD$ does not hold in any metric measure space that contains atoms, as we will see in Proposition \ref{RvdImpliesNoAtoms}. (By an \textit{atom}, we mean a point $x \in M$ such that $\mu(\{x\})>0$.) Thus, the results of \cite{ckw} and \cite{ckw2} do not apply to graphs, since every point in a graph is an atom. Heat kernels and the parabolic Harnack inequality
on graphs are of interest (for example, see \cite{del}, \cite{ms1}, \cite{ms2} \cite{ms3}, \cite{bb}, \cite{groups}, \cite{kumagaihkandharnack}).
In this paper, we generalize the characterizations from \cite{ckw} and \cite{ckw2} to graphs. In fact, our results hold not only for graphs, but also for ``mixed" spaces that contain both atoms and non-atoms. See Example \ref{MixedExample} for a simple example of such a mixed space.

We define a new volume-growth condition $\QRVD$ (or ``quasi-$\RVD$," see Definition \ref{QrvdDefinition}) which is a weaker version of $\RVD$ and can be thought of as ``$\RVD$ at large scales." We show that the assumptions of \cite{ckw} and \cite{ckw2} can be relaxed from $\VD$ and $\RVD$ to $\VD$ and $\QRVD$. As a result, spaces like graphs are now allowed. Rather than try to adapt the proofs of \cite{ckw} and \cite{ckw2} to the weaker setting of $\VD$ and $\QRVD$, or prove the stable characterizations from scratch, our solution is to construct a new metric measure Dirichlet space entirely, that preserves the important structure from $(M, d, \mu, \mathcal{E}, \mathcal{F})$, but ``smooths out" the atoms and satisfies $\VD$ and $\RVD$.

In Section \ref{auxiliarySection}, we construct a new metric measure space $(\Mhat, \dhat, \muhat)$, which we call the \textit{auxiliary metric measure space}. Essentially, each atom $x \in M$ is replaced with a continuous mass $W_x$, and each non-atom in $M$ remains the same in the auxiliary metric measure space. The resulting metric measure space has no atoms.
See Figures \ref{IntegerLatticeFigure} and \ref{MixedFigure} for examples of what the auxiliary metric measure space looks like.
We show that if $(M, d, \mu)$ satisfies $\VD$ and $\QRVD$, then $(\Mhat, \dhat, \muhat)$ satisfies $\VD$ and $\RVD$.
Therefore, the results of \cite{ckw} and \cite{ckw2} can be applied to $(\Mhat, \dhat, \muhat)$. For each atom $x \in M$, the set $W_x$ that replaces $x$ in the auxiliary metric measure space has the \textit{ultrametric property} (see \eqref{WxUltrametric} and \eqref{WxUltrametricBalls}), which will have many nice consequences.

Later in Section \ref{auxiliarySection}, we construct a pure-jump regular Dirichlet form $(\Ehat, \Fhat)$ on $L^2(\Mhat, \muhat)$. We refer to the metric measure Dirichlet space $(\Mhat, \dhat, \muhat, \Ehat, \Fhat)$ as simply the \textit{auxiliary space}, and we refer to $(M, d, \mu, \mathcal{E}, \mathcal{F})$ as the \textit{original space}.
We let $\{\Xhat_t\}_{t \geq 0}$ denote the Hunt process associated with $(\Ehat, \Fhat)$. We then show that the processes $\{X_t\}$ and $\{\Xhat_t\}$ can be coupled so that if $x \in M$ is an atom, then the process $\Xhat_t$ jumps around $W_x$ (the mass in the auxiliary space associated with $x$) whenever $X_t=x$, and if $x \in M$ is not an atom, then $\Xhat_t = x$ whenever $X_t=x$.

In Section \ref{proofMainResults}, we state (without proof) a list of results of the form
\begin{align} \label{typeOfResult}
\begin{split}
    \mbox{$\COND$ for $(M, d, \mu, \mathcal{E}, \mathcal{F})$} \Longrightarrow \mbox{$\COND$ for $(\Mhat, \dhat, \muhat, \Ehat, \Fhat)$}\\
    \mbox{or} \qquad \mbox{$\COND$ for $(M, d, \mu, \mathcal{E}, \mathcal{F})$} \Longleftarrow \mbox{$\COND$ for $(\Mhat, \dhat, \muhat, \Ehat, \Fhat)$}
\end{split}
\end{align}
for various conditions $\COND$. We then use these to prove our desired characterizations of $\HK$ and $\PHI$. See Figures \ref{HkStableDiagram}-\ref{PhiStableDiagram} for diagrams detailing these arguments.
The general gist of each argument is as follows: if $\CHAR_1$ and $\CHAR_2$ are two characterizations, then in order to prove that $\CHAR_1$ implies $\CHAR_2$, we argue as follows:
\begin{align*}
   \boxed{\mbox{$\CHAR_1$ for $(M, d, \mu, \mathcal{E}, \mathcal{F})$}} &\Longrightarrow \boxed{\mbox{$\CHAR_1$ for $(\Mhat, \dhat, \muhat, \Ehat, \Fhat)$}} \qquad\mbox{(by our results of the form \eqref{typeOfResult})}\\
   &\Longrightarrow \boxed{\mbox{$\CHAR_2$ for $(\Mhat, \dhat, \muhat, \Ehat, \Fhat)$}} \qquad\mbox{(by \cite{ckw} and \cite{ckw2})}\\
   &\Longrightarrow \boxed{\mbox{$\CHAR_2$ for $(M, d, \mu, \mathcal{E}, \mathcal{F})$}} \qquad\mbox{(by our results of the form \eqref{typeOfResult})}.\\
\end{align*}

In Sections \ref{EscapeSection}-\ref{CsPoincSection}, we prove the results of the form \eqref{typeOfResult} that we stated without proof at the beginning of Section \ref{proofMainResults}. 
Many of the proofs in these sections are made simpler by the ultrametric property of $W_x$.

In
Section \ref{EscapeSection}, we prove the results of the form \eqref{typeOfResult} relating to escape times.

In Section \ref{HeatKernelSection1}, we derive a formula for the heat kernel of the jump process $\{\Xhat_t\}$ on the auxiliary space.

In Section \ref{HeatKernelSection2}, we prove the results of the form \eqref{typeOfResult} relating to heat kernels. The proofs in Section \ref{HeatKernelSection2} all use the formula derived in Section \ref{HeatKernelSection1} for the heat kernel of $\{\Xhat_t\}$.

In Section \ref{CsPoincSection}, we prove the results of the form \eqref{typeOfResult} that relate to the cut-off Sobolev and Poincar\'e inequalities. These require us to first establish some technical results about the domains $\mathcal{F}$ and $\Fhat$.

Later in this introduction, after we state our main results, we show that for the special case of graphs with infinite diameter, $\VD$ implies $\QRVD$, so $\HK$ and $\PHI$ are stable for graphs of infinite diameter that satisfy $\VD$.

The novel ingredient in our solution is the construction of the auxiliary space. This idea allows us to extend arguments that rely heavily on $\RVD$ to a space in which $\RVD$ does not hold. We note that the auxiliary space we construct is analogous to the cable system of a graph: the space formed by replacing each edge in the graph with a ``cable" between its two endpoints. (See \cite[Section 2]{bb} for more details on cable systems.) The cable system construction has many properties (such as connectedness) that make it convenient for studying nearest-neighbor random walks on a graph (by instead looking at a diffusion on the cable system), but our construction has properties (such as the ultrametric property on $W_x$ for all atoms $x \in M$) that make it more suitable for studying jump processes.
We are hopeful that other results can be extended from continuous-space to discrete-space settings by turning to an auxiliary space similar to the one we construct here. 

We would like to emphasize that our results themselves are novel. The stability of $\HK$ was not known for $\alpha$-stable-like jump processes until \cite{ckw}, which did not address the case of graphs. The stability of $\PHI$ for graphs was not known for the kinds of jump processes we consider, since results like \cite{bb} only apply to nearest-neighbor random walks on locally finite graphs, whereas our results apply to jump processes that may jump from any $x$ to any $y$ in a single step.

\LongVersionShortVersion
{}
{ 
Many of our arguments require us to verify statements that are easy (but long) to check. We omit some of these details here, but include them all in \cite{longversion}, a longer version of this paper, published on arXiv. This way, the reader need not get bogged down in these details, but has the option of seeing how they are verified.
} 

Before we state our main results in their entirety, we must define some terminology.

\subsection*{Volume growth}

For all $x \in M$ and $r>0$, let $B(x, r)$ denote the open ball $\{y \in M : d(x, y) < r\}$ with center $x$ and radius $r$. Let $V(x, r)$ denote $\mu(B(x, r))$, the volume of the open ball with center $x$ and radius $r$. Let us define the following conditions relating to volume growth in $(M, d, \mu)$:

\begin{definition}\label{VdDefinition}
We say that $(M, d, \mu)$ satisfies condition $\VD$ (\textit{volume doubling}) if there exists a constant $C>0$ such that
\begin{equation} \label{VdFormula}
    V(x,2r) \leq CV(x, r) \qquad\mbox{for all $x\in M$, $r>0$}.
\end{equation}
\end{definition}

Note that $\VD$ would remain true (or remain false) if the constant $2$ in \eqref{VdFormula} was replaced with any other $\ell>1$. (However, the constant $C$ would depend on $\ell$.) Also note that $\VD$ is equivalent to the existence of constants $C>0$ and $d>0$ such that
\begin{equation} \label{VdAlternate}
    \frac{V(x, R)}{V(x, r)} \leq C \left( \frac{R}{r} \right)^d \qquad\mbox{for all $x \in M$ and $R \geq r>0$}.
\end{equation}

\begin{definition} \label{RvdDefinition}
We say that condition $(M, d, \mu)$ satisfies condition $\RVD$ (\textit{reverse volume doubling}) holds if there exist constants $\ell>1$ and $c>1$ such that
\begin{equation} \label{rvdDef}
    V(x,\ell r) \geq cV(x, r) \qquad\mbox{whenever $x \in M$ and $r>0$}.
\end{equation}
\end{definition}

Qualitatively, if we fix a point $x$ and consider open balls centered at $x$ with growing radii, $\VD$ means that the volumes do not grow too quickly, and $\RVD$ means that the volumes do not grow too slowly.

We say that a point $x \in M$ is an \textit{atom} if $\mu(x) > 0$. In the following proposition, we see that if $\RVD$ holds, there can be no atoms.

\begin{prop} \label{RvdImpliesNoAtoms}
If $\RVD$ holds, then $\mu(x)=0$ for all $x \in M$.
\end{prop}

\begin{proof}

Let $c$ be the constant from \eqref{rvdDef}.
For all $x \in M$, by induction and $\RVD$, $V(x, \ell^{-n}) \leq c^{-n} V(x, 1)$ for all $n$, so $\mu(x) = \lim_{n\rightarrow \infty} V(x, \ell^{-n}) = 0$.
\end{proof}

Recall that \cite{ckw} and \cite{ckw2} assume $(M, d, \mu)$ satisfies $\VD$ and $\RVD$. Proposition \ref{RvdImpliesNoAtoms} tells us that this can never be the case for metric measure spaces containing atoms. In order to extend the results of \cite{ckw} and \cite{ckw2} to metric measure spaces that contain atoms, we introduce a new condition $\QRVD$ (or \textit{quasi-$\RVD$}), which is the same as $\RVD$, with the exception that it does not require \eqref{rvdDef} to hold when $x$ is an atom and $r < D_x$.
This new condition $\QRVD$ allows for atoms.
We will show that the assumption of $\VD$ and $\RVD$ in \cite{ckw} and \cite{ckw2} can be relaxed to $\VD$ and $\QRVD$.

For all $x \in M$, let
\begin{equation} \label{DxFormula}
    D_x := \inf_{y \in M \setminus \{x\}} d(x, y).
\end{equation}
If $D_x>0$, we say that $x$ is \textit{isolated}.

\begin{definition}\label{QrvdDefinition}
We say $(M, d, \mu)$ satisfies $\QRVD$ (\textit{quasi-reverse volume doubling}) if there exist constants $\ell>1$ and $c>1$ such that
\begin{equation}\label{QrvdFormula}
    V(x, \ell r) \geq cV(x, r) \qquad\mbox{for all $x\in M$, $r \geq D_x$, $r>0$}
\end{equation}
where $D_x$ is as defined in \eqref{DxFormula}.
\end{definition}

While $\VD$ and $\RVD$ are standard conditions, $\QRVD$ is a condition of our own invention.

\subsection*{Jump kernel}

We say that $(\mathcal{E}, \mathcal{F})$ \textit{admits a jump kernel} if $$J(dx, dy) = J(x, y) \mu(dy) \mu(dx)$$ for a  non-negative function $J(x, y)$ on $M \times M \setminus \diag_M$, which we call the \textit{jump kernel}. The jump kernel $J(x, y)$ is not to be confused with the jumping measure $J(dx, dy)$. We are abusing notation by using ``$J$" to refer to both.

\begin{definition}\label{JphiDefinition}
We say that condition $\J$ holds if $(\mathcal{E}, \mathcal{F})$ admits a jump kernel and there exist constants $C\geq c>0$ such that
\begin{equation}\label{JphiDef}
    \frac{c}{V(x, d(x, y)) \phi(d(x,y))} \leq J(x, y) \leq \frac{C}{V(x, d(x, y)) \phi(d(x,y))} \qquad\mbox{for $\mu \times \mu$-almost all distinct $x, y \in M$}.
\end{equation}
We will refer to the upper and lower bounds of \eqref{JphiDef}, respectively, as $\Jleq$ and $\Jgeq$.
\end{definition}

If $(\mathcal{E}, \mathcal{F})$ admits a jump kernel $J(x, y)$, we can alter its values on a null set (without changing the Dirichlet form $(\mathcal{E}, \mathcal{F})$) so that $J(x, y)$ is symmetric, and the ``$\mu \times \mu$-almost all" in the definitions of $\Jleq$ and $\Jgeq$ is not necessary. Let us always do so. (See \cite[Remark 1.3]{ckw}.)

\begin{definition} \label{UjsDefinition}
We say that $\UJS$ holds if $(\mathcal{E}, \mathcal{F})$ admits a jump kernel and there exists a constant $C>0$ such that for $\mu\times\mu$-almost all distinct $x, y \in M$,
\begin{equation*}
    J(x, y) \leq \frac{C}{V(x, r)} \int_{z \in B(x, r)} J(z, y) \mu(dz) \qquad\mbox{for all $0<r \leq \frac{d(x, y)}{2}$}.
\end{equation*}
\end{definition}

\subsection*{Escape times}

For any open $U \subseteq M$, let
\begin{equation} \label{tauFormula}
    \tau_U := \inf\{t \geq 0 : X_t \notin U\}.
\end{equation}

Recall that $M_0 := M \setminus \mathcal{N}$, where $\mathcal{N}$ is the properly exceptional set of the Hunt process $\{X_t\}$.

\begin{definition}\label{EphiDefinition}
We say that condition $\E$ (\textit{escape times}) holds if there exist constants $C\geq c>0$ such that
\begin{equation} \label{EphiDef}
    c\phi(r) \leq \mathbb{E}^x \tau_{B(x, r)} \leq C \phi(r) \qquad\mbox{for all $x\in M_0$, $r>0$}.
\end{equation}
We will refer to the upper and lower bounds of \eqref{EphiDef}, respectively, as $\Eleq$ and $\Egeq$.
\end{definition}

\LongVersionShortVersion
{ 
In Appendix \ref{EphiNotPossibleAppendix}, we prove that $\Eleq$ never holds on any space $(M, d, \mu)$ containing atoms. Therefore, we will need to consider the following weaker form of $\E$.
}
{ 
The condition $\Eleq$ never holds on any space $(M, d, \mu)$ containing atoms. (For a proof of this, see \cite[Appendix B.3]{longversion}, from the longer version of this paper that we put on arXiv.) Therefore, we will need to consider the following weaker form of $\E$.
}

\begin{definition}\label{QEphiDefinition}

Recall how we defined the quantity $D_x$ in \eqref{DxFormula}.
We say that $\QE$ (\textit{quasi}-$\E$) holds if there exist constants $C \geq c>0$ such that
\begin{equation} \label{qeDef}
    c \phi(r) \leq \mathbb{E}^x \tau_{B(x, r)} \leq C\phi(r) \qquad\mbox{for all $x\in M$, $r \geq D_x$, $r>0$}.
\end{equation}
As usual, let $\QEleq$ and $\QEgeq$ refer to the upper and lower bounds of $\QE$.
\end{definition}

Note that the definition of $\QE$ (like $\QRVD$) is new, whereas $\E$ is standard.

\subsection*{Heat kernel}

Recall that $M_0 := M \setminus \mathcal{N}$, where $\mathcal{N}$ is the properly exceptional set of the Hunt process $\{X_t\}$.
If it exists, let $p : (0, \infty) \times M_0 \times M_0 \to [0, \infty)$ denote the \textit{heat kernel} on $(M, d, \mu, \mathcal{E}, \mathcal{F})$ such that
\begin{align}
    \mathbb{E}^x f(X_t) = \int_M p(t, x, y) f(y) \mu(dy) & \qquad \mbox{for all $f \in L^\infty(M, \mu)$, $x \in M_0$, $t>0$}. \label{HeatKernelDefinitionDistribution} \\
    p(t, x, y) = p(t, y, x) & \qquad \mbox{for all $x, y \in M_0$, $t>0$}. \label{HeatKerelDefinitionSymmetry} \\
    p(s+t, x, z) = \int_M p(s, x, y) p(t, y, z) \mu(dy) & \qquad\mbox{for all $x, z \in M_0$ and $s,t > 0$}. \label{HeatKernelDefinitionChapmanKolmogorov}
\end{align}

Equation \eqref{HeatKernelDefinitionChapmanKolmogorov} is called \textit{Chapman-Kolmogorov}.

Recall that throughout the paper, we have a fixed function $\phi$ of regular growth. Let
\begin{equation} \label{QtxyFormula}
    q(t, x, y) := \frac{1}{V(x, \phi^{-1}(t))} \wedge \frac{t}{V(x, d(x, y)) \phi(d(x,y))}.
\end{equation}
We take $\frac10$ to be $+\infty$, so that on the diagonal both $V(x, d(x, x))$ and $\phi(d(x, x))$ are $0$, so
\begin{equation} \label{onDiagonalQt}
    q(t, x, x) = \frac{1}{V(x, \phi^{-1}(t))} \wedge (+\infty) = \frac{1}{V(x, \phi^{-1}(t))}.
\end{equation}

\begin{definition} \label{HKphiDefinition}
We say that $\HK$ (\textit{kernel estimates}) holds if there exist constants $C\geq c>0$ such that
\begin{equation}\label{HKphiFormula}
    cq(t, x, y) \leq p(t, x, y) \leq Cq(t, x, y) \qquad\mbox{for all $x, y \in M_0$, $t>0$}.
\end{equation}
We will refer to the upper and lower bounds of \eqref{HKphiFormula}, respectively, as $\UHK$ and $\LHK$.
\end{definition}

\begin{definition}\label{UhkdDefinition}
We say that $\UHKD$ (\textit{upper bound for heat kernel on the diagonal}) holds if there exists a constant $C>0$ such that
\begin{equation*}
    p(t, x, x) \leq \frac{C}{V(x, \phi^{-1}(t))} \qquad\mbox{for all $x \in M_0$, $t>0$}.
\end{equation*}
\end{definition}

For any open $U$, let $p^U : (0, \infty) \times M_0 \times M_0 \to [0, \infty)$ be the \textit{Dirichlet heat kernel} (in other words, the heat kernel of the jump+death process $\{Y_t\}$ defined by
\begin{equation*}
    Y_t := \left\{ \begin{matrix}
        X_t &:& \mbox{if $t < \tau_U$}\\
        \death &:& \mbox{if $t \geq \tau_U$}
    \end{matrix}\right.
\end{equation*}
where $\death$ is the death-state, and $\tau_U$ is as in \eqref{tauFormula}).

\begin{definition}\label{NdlDefinition}
We say $\NDL$ (\textit{near-diagonal lower bound}) holds if there exist $\varepsilon \in (0, 1)$ and $c_1>0$ such that for all $B=B(x_0, r)$,
\begin{equation*}
    p^B(t, x, y) \geq \frac{c_1}{V(x_0, \phi^{-1}(t))} \qquad\mbox{for all $x, y \in B(x_0, \varepsilon\phi^{-1}(t)) \cap M_0$, $0 < t \leq \phi(\varepsilon r)$}.
\end{equation*}
\end{definition}

\subsection*{Parabolic Harnack inequality}

Let $\{Z_t\}_{t \geq 0} = \{ (V_t, X_t) \}_{t \geq 0}$ be the space-time process where $V_t = V_0 - t$ (for some starting time $V_0$). Let $\mathbb{P}^{(t_0, x_0)}$ be the law of $t \mapsto Z_t$ starting from $(t_0, x_0)$. For all open $D \subseteq [0, \infty) \times M$, let
\begin{equation*}
    \tau_D := \inf \{t \geq 0 : Z_t \notin D \}.
\end{equation*}
We say that a set $A \subseteq [0, \infty) \times M$ is \textit{nearly Borel measurable} if for any probability measure $\mu_0$ on $[0, \infty) \times M$, there exist Borel measurable sets $A_1$ and $A_2$ such that $A_1 \subseteq A \subseteq A_2$ and
\begin{equation*}
    \mathbb{P}^{\mu_0}(\mbox{$Z_t \in A_2 \setminus A_1$ for some $t \geq 0$}) = 0.
\end{equation*}
We say that a nearly Borel measurable function $u(t, x)$ on $[0, \infty) \times M$ is \textit{caloric} on $D = (a, b) \times B(x_0, r)$ if there is a properly exceptional set $\mathcal{N}_\mu$ of $X$ such that for every relatively compact open subset $U$ of $D$,
\begin{equation*}
    u(t, x) = \mathbb{E}^{(t, x)} u(Z_{\tau_U}) \qquad\mbox{for all $(t, x) \in U \cap ([0, \infty) \times (M \setminus \mathcal{N}_\mu))$}.
\end{equation*}

\begin{definition}\label{PhiDefinition}
We say that $\PHI$ (the \textit{parabolic Harnack inequality}) holds if there exist constants $0<C_1<C_2<C_3<C_4$, $C_5 \in (0, 1)$, and $C_6>0$, such that for all $x_0 \in M$, $t_0 \geq 0$, $R>0$, and every non-negative $u$ that is caloric on the cylinder $Q(t_0, x_0, C_4 \phi(R)) := (t_0, t_0 + C_4 \phi(R)) \times B(x_0, R)$,
\begin{equation*}
    \esssup_{\left(t_0 + C_1 \phi(R), t_0 + C_2 \phi(R)\right) \times B(x_0, C_5 R)} u \leq C_6 \essinf_{\left(t_0 + C_3 \phi(R), t_0 + C_4\phi(R)\right) \times B(x_0, C_5 R)} u.
\end{equation*}
\end{definition}

We also consider the following variant of the parabolic Harnack inequality, which is more restrictive about the constants:

\begin{definition}\label{PhiplusDefinition}
We say that $\PHIplus$ holds if we have $\PHI$ with $C_k = k C_1$ for all $k \in \{2, 3, 4\}$.
\end{definition}

\subsection*{H\"{o}lder regularity}

Let $D$ be an open subset of $M$. We say that a nearly Borel measurable function $u : M \to \mathbb{R}$ is \textit{harmonic} on $D$ (with respect to $X$) if for any relatively compact subset $U \subseteq D$, the process $\{u(X_t)\}_{t\geq 0}$ is a uniformly integrable martingale under $\mathbb{P}^x$ for q.e. $x \in U$.

\begin{definition}\label{EhrDefinition}
We say that $\EHR$ (\textit{elliptic H\"{o}lder regularity}) holds if there exist constants $c >0$, $\theta \in (0, 1]$, and $\varepsilon \in (0, 1)$ such that for all $x_0\in M$ and $r>0$, for every bounded measurable function $u$ on $M$ that is harmonic on $B(x_0, r)$, there is a properly exceptional set $\mathcal{N}_u \supseteq \mathcal{N}$ such that
\begin{equation*}
    |u(x)-u(y)| \leq c \left( \frac{d(x, y)}{r} \right)^\theta \esssup_M |u|
\end{equation*}
for all $x, y \in B(x_0, \varepsilon r) \setminus \mathcal{N}_u$.
\end{definition}

\begin{definition}\label{PhrDefinition}
We say that $\PHR$ (parabolic H\"{o}lder regularity) holds if there exist constants $c>0$, $\theta \in (0, 1]$, and $\varepsilon \in (0, 1)$ such that for all $x_0 \in M$, $t\geq0$, and $r>0$, for every bounded measurable function $u(t, x)$ that is caloric on $Q(t_0, x_0, \phi(r), r) := (t_0, t_0 + \phi(r)) \times B(x_0, r)$, there exists a properly exceptional set $\mathcal{N}_u \subseteq \mathcal{N}$ such that
\begin{equation} \label{PhrFormula}
    |u(s, x)-u(t, y)| \leq c \left( \frac{\phi^{-1}(|s-t|) + d(x, y)}{r} \right)^\theta \esssup_{[t_0, t_0 + \phi(r)] \times M} |u|
\end{equation}
for all $s, t \in (t_0, t_0+\phi(r)))$ and $x, y \in B(x_0, \varepsilon r) \setminus \mathcal{N}_u$.
\end{definition}

Clearly, $\PHR \Longrightarrow \EHR$.

\subsection*{Cut-off Sobolev}

For $f, g \in \mathcal{F}$, we define the \textit{carr\'e du-Champ operator} $\Gamma(f, g)$ by
\begin{equation} \label{carreDef}
    \Gamma(f, g) (dx) := \int_{y \in M} (f(x)-f(y))(g(x)-g(y))\, J(dx, dy).
\end{equation}
Let $\Gamma(f):=\Gamma(f, f)$.

If $U$ and $V$ are open sets in $M$ such that $\overline{U} \subseteq V$, we define $\cutoff(U, V)$ to be the set of functions $\varphi \in \mathcal{F}$ such that $0 \leq \varphi \leq 1$ everywhere, $\varphi=1$ inside $U$, and $\varphi=0$ on $V^c$.

Equipped with the carr\'e du-Champ operator and cut-off functions, we are now ready to define the cut-off Sobolev inequalities. Recall that cut-off Sobolev inequalities are often required for characterizations of heat kernel estimates or parabolic Harnack inequality when $\beta_2$ from \eqref{regularGrowth} is greater than or equal to $2$. The forms of cut-off Sobolev that we consider, $\CSJ$ and $\SCSJ$, were first introduced in \cite{ckw}.

\begin{definition}\label{CsjDefinition}
We say that $\CSJ$ holds if there exist constants $C_0 \in (0, 1]$ and $C_1, C_2>0$ such that for all $R \geq r>0$, for $\mu$-almost all $x_0 \in M$, and for all $f \in \mathcal{F}$, there exists a $\varphi \in \cutoff(B(x_0, R), B(x_0, R+r))$ such that
\begin{equation} \label{CsjFormula}
    \int_{B(x_0, R+(1+C_0)r)} f^2 \, d\Gamma(\varphi) \leq C_1\int_{U} \int_{U^*} (f(x)-f(y))^2 J(dx, dy) + \frac{C_2}{\phi(r)} \int_{B(x_0, R+(1+C_0)r)} f^2 \, d\mu 
\end{equation}
where
\begin{align*}
    U &:= B(x_0, R+r) \setminus B(x_0, R),\\
    \mbox{and}\qquad U^* &:= B(x_0, R+(1+C_0)r) \setminus B(x_0, R-C_0 r).
\end{align*}

\end{definition}

\begin{definition}\label{ScsjDefinition}

We say that $\SCSJ$ holds if if there exist constants $C_0 \in (0, 1]$ and $C_1, C_2>0$ such that for all $R\geq r>0$, for $\mu$-almost all $x_0 \in M$, there exists a $\varphi \in \cutoff(B(x_0, R), B(x_0, R+r))$ for which \eqref{CsjFormula} holds for all $f \in \mathcal{F}$.

\end{definition}

Clearly, $\SCSJ \Longrightarrow \CSJ$, since $\SCSJ$ is more restrictive than $\CSJ$ (in that $\varphi$ is not allowed to depend on $f$). Note that $\CSJ$ and $\SCSJ$ are stable.

\subsection*{Poincar\'e inequality}

Let $\mathcal{F}_b := \{f \in \mathcal{F} : \norm{f}_\infty<\infty\}$.

\begin{definition} \label{PoincareDefinition}
We say that the \textit{Poincar\'e inequality} $\PI$ holds if there exist constants $C>0$ and $\kappa\geq1$ such that for all $f \in \mathcal{F}_b$,
\begin{equation} \label{PoincareFormula}
    \int_{B} (f-f_B)^2 \, d\mu \leq C\phi(r) \int_{\kappa B \times \kappa B} (f(x)-f(y))^2 J(dx, dy) \qquad\mbox{for all $x_0 \in M$, $r>0$}
\end{equation}
where
\begin{align*}
    B:&=B(x_0, r),\\
    \kappa B :&= B(x_0, \kappa r),\\
    f_B :&= \frac{1}{\mu(B)} \int_B f \, d\mu.
\end{align*}
\end{definition}

\subsection*{Basic assumptions}

The majority of our results have the same basic assumptions.
So that we do not need to keep stating them over and over again, let us wrap them all together into Assumption \ref{BasicAssumptions}.

\begin{assumption} \label{BasicAssumptions}
For the metric measure Dirichlet space $(M, d, \mu, \mathcal{E}, \mathcal{F})$,

(a) $(M, d)$ is locally compact and separable, $\mu$ is a positive Radon measure on $M$ with full support, and $\mu(M)=\infty$.

(b) $(\mathcal{E}, \mathcal{F})$ is a pure-jump regular Dirichlet form.
\end{assumption}

After we construct the auxiliary space $(\Mhat, \dhat, \muhat, \Ehat, \Fhat)$ in Section \ref{auxiliarySection}, we will need to show that $(\Mhat, \dhat, \muhat, \Ehat, \Fhat)$ also satisfies Assumption \ref{BasicAssumptions}, in order to apply the results of \cite{ckw} and \cite{ckw2} to the auxiliary space.

\subsection*{Previous results}

Let us state the results of Chen, Kumagai, and Wang that we will extend.

\begin{thm113}
Suppose $(M, d, \mu, \mathcal{E}, \mathcal{F})$ satisfies Assumption \ref{BasicAssumptions}, $\VD$ and $\RVD$ hold, and $\phi$ is of regular growth. The following are equivalent:
\begin{enumerate}
    \item $\HK$
    \item $\J + \E$
    \item $\J + \CSJ$ \label{J+CSJ}
    \item $\J + \SCSJ$. \label{J+SCSJ}
\end{enumerate}
\end{thm113}
Characterizations \ref{J+CSJ} and \ref{J+SCSJ} from \cite[Theorem 1.13]{ckw} are stable.

\begin{thm117}
Suppose $(M, d, \mu, \mathcal{E}, \mathcal{F})$ satisfies Assumptions \ref{BasicAssumptions}, $\VD$ and $\RVD$ hold, and $\phi$ is of regular growth. The following are equivalent:

\begin{enumerate}
    \item $\PHI$
    \item $\PHIplus$
    \item $\UHK + \NDL + \UJS$
    \item $\NDL+\UJS$
    \item $\PHR +\Eleq + \UJS$ \label{ConditionInvolvingPhr}
    \item $\EHR+\E + \UJS$ \label{ConditionInvolvingEhr}
    \item $\PI + \Jleq + \CSJ + \UJS$ \label{PI+Jleq+CSJ+UJS}
\end{enumerate}
\end{thm117}
Characterization \ref{PI+Jleq+CSJ+UJS} of \cite[Theorem 1.17]{ckw2} is stable.
Technically, \cite{ckw2} also assumes that for each $x \in M$, there exists a kernel $J(x, dy)$ such that $J(dx, dy) = J(x, dy) \, \mu(dx)$. However, Liu and Murugan \cite{LiMu} later proved that $\PHI$ implies the existence of a jump kernel, so this extra assumption is not necessary.

\subsection*{Our main results}

We show that the characterizations of $\HK$ from \cite{ckw} still hold when the assumption of $\RVD$ is relaxed to $\QRVD$, and the characterization $\J+\E$ is replaced with $\J+\QE$.

\begin{theorem} \label{ourMainResult1}
Suppose $(M, d, \mu, \mathcal{E}, \mathcal{F})$ satisfies Assumption \ref{BasicAssumptions}, $\VD$ and $\QRVD$ hold, and $\phi$ is of regular growth. The following are equivalent:

\begin{itemize}
    \item $\HK$
    \item $\J + \QE$
    \item $\J + \CSJ$
    \item $\J + \SCSJ$.
\end{itemize}
\end{theorem}

We also show that the stable characterization of $\PHI$ from \cite{ckw2} holds when the assumption of $\RVD$ is relaxed to $\VD$, as do many of the other characterizations.

\begin{theorem}\label{ourMainResult3}
Suppose $(M, d, \mu, \mathcal{E}, \mathcal{F})$ satisfies Assumptions \ref{BasicAssumptions}, $\VD$ and $\QRVD$ hold, and $\phi$ is of regular growth. The following are equivalent:

\begin{itemize}
    \item $\PHI$
    \item $\PHIplus$
    \item $\UHK + \NDL + \UJS$
    \item $\NDL+\UJS$
    \item $\PI + \Jleq + \SCSJ + \UJS$
    \item $\PI + \Jleq + \CSJ + \UJS$.
\end{itemize}
Furthermore, if any of these equivalent conditions holds, then so does $\PHR$ (and therefore so does $\EHR$).
\end{theorem}

Note that the characterizations \ref{ConditionInvolvingPhr} and \ref{ConditionInvolvingEhr} from \cite[Theorem 1.17]{ckw2} are absent from our Theorem \ref{ourMainResult3}.
\LongVersionShortVersion
{ 
This is because, as we will see in Appendix \ref{EphiNotPossibleAppendix}, $\Eleq$ does not hold for any $(M, d, \mu, \mathcal{E}, \mathcal{F})$ containing atoms.
It is possible that we could have replaced $\Eleq$ with $\QEleq$, but we did not investigate this possibility since these characterizations were not necessary for the proof of stability.
}
{ 
This is because, as we show in \cite[Appendix B.3]{longversion}, condition $\Eleq$ does not hold for any $(M, d, \mu, \mathcal{E}, \mathcal{F})$ containing atoms.
It is possible that we could have replaced $\Eleq$ with $\QEleq$, but we did not investigate this possibility since these characterizations were not necessary for the proof of stability.
} 

Note however that we do prove $\PHI \Longrightarrow \PHR \Longrightarrow \EHR$.
 
\subsection*{Main results applied to graphs} \label{graphSubsection}

The condition $\QRVD$ may seem artificial. In this subsection, we show if the underlying metric measure space is a graph, then $\QRVD$ is a direct consequence of $\VD$.

\begin{definition}
We say that $(M, d)$ is \textit{uniformly perfect} if there exists a constant $C>1$ such that the annulus $B(x, Cr) \setminus B(x, r)$ is non-empty for all $x \in M$, $r>0$ such that $B(x, r) \neq M$.
\end{definition}

\begin{definition}
We say that $(M, d)$ is \textit{quasi-uniformly perfect} if there exists a constant $C>1$ such that $B(x, Cr) \setminus B(x, r)$ is non-empty for all $x \in M$, $r \geq D_x$ such that $B(x, r) \neq M$.
\end{definition}

Uniform perfectness is a standard term, whereas quasi-uniform perfectness is a term of our invention, like $\QRVD$ and $\QE$.

It is a standard result (see \cite[Exercise 13.1]{heinonen}) that for a metric measure space, $\VD$ and uniform perfectness imply $\RVD$. It is not hard to show the analogous result that $\VD$ and quasi-uniform perfectness imply $\QRVD$.

\begin{lemma}\label{VdPlusQufImpliesQrvd}
Suppose $(M, d)$ is quasi-uniformly perfect and $\mu$ is a Radon measure on $M$ with full support, such that $\mu(M) = \infty$. Then $\VD \Longrightarrow \QRVD$.
\end{lemma}

The proof of Lemma \ref{VdPlusQufImpliesQrvd} is virtually identical to the standard proof that $\VD$ and $\mbox{uniform perfectness} \Longrightarrow \RVD$.
\LongVersionShortVersion
{ 
For the sake of completeness, we include a proof of Lemma \ref{VdPlusQufImpliesQrvd} in Appendix \ref{UniformPerfectnessAppendix}.
}
{ 
For the sake of completeness, we prove Lemma \ref{VdPlusQufImpliesQrvd} in \cite[Appendix B.1]{longversion}, in the longer version of this paper that we put on arXiv.
} 

It is easy to see that if $M$ is the vertex set of a connected, undirected graph with infinite diameter, and $d$ is the graph metric (i.e. the metric of shortest path) for $G$, then $(M, d)$ is uniformly perfect. Therefore, we obtain the following corollary.

\begin{corollary} \label{ourMainResultGraphs}
Suppose $G=(M, E)$ is a connected, undirected, countable graph of infinite diameter, $d$ is the graph metric of $G$, $\mu$ is a positive Radon measure on $M$ such that $\mu(M)=\infty$, and $(\mathcal{E}, \mathcal{F})$ is a pure-jump regular Dirichlet form on $L^2(M, \mu)$. If $(M, d, \mu)$ satisfies $\VD$, and $\phi$ is a function of the regular growth, then
\begin{align*}
    \HK &\Longleftrightarrow \J+\QE\\ &\Longleftrightarrow \J+\CSJ\\ &\Longleftrightarrow \J+\SCSJ
\end{align*}
and
\begin{align*}
    \PHI &\Longleftrightarrow \PHIplus\\ &\Longleftrightarrow \UHK+\NDL+\UJS\\
    &\Longleftrightarrow \NDL+\UJS\\
    &\Longleftrightarrow \PI+\Jleq+\SCSJ+\UJS\\
    &\Longleftrightarrow \PI+\Jleq+\CSJ+\UJS\\
    &\Longrightarrow \PHR \Longrightarrow \EHR.
\end{align*}
\end{corollary}

\begin{proof}
Fix $x \in M$ and $r \geq D_x = 1$. Since $r \geq 1$, there exists an integer $k \in [r, 2r)$. Since the graph has infinite diameter, there exists a $y$ such that $d(x, y) = k$. Therefore, the annulus $B(x, 2r) \setminus B(x, r)$ is non-empty. Since this holds for all $x \in M$ and $r \geq D_x$, $(M, d)$ is quasi-uniformly perfect. By Lemma \ref{VdPlusQufImpliesQrvd}, $(M, d, \mu)$ satisfies $\QRVD$. Therefore, we can apply Theorems \ref{ourMainResult1} and \ref{ourMainResult3} to $(M, d, \mu, \mathcal{E}, \mathcal{F})$.
\end{proof}

By \cite[Remark 1.7]{ckw}, if the constant $\beta_2$ from \eqref{regularGrowth} is less than $2$, then $\VD+\Jleq \Longrightarrow \SCSJ$. Therefore, if $(M, d, \mu, \mathcal{E}, \mathcal{F})$ satisfies the conditions of Theorems \ref{ourMainResult1}-\ref{ourMainResult3} or Corollary \ref{ourMainResultGraphs}, and the constant $\beta_2$ from \eqref{regularGrowth} is less than $2$, then we have the even nicer stable characterizations
\begin{equation*}
    \HK \Longleftrightarrow \J \Longrightarrow \QE + \SCSJ \qquad\mbox{and}\qquad \PHI \Longleftrightarrow \PI + \Jleq + \UJS.
\end{equation*}

\subsection*{Acknowledgement}

My deepest gratitude goes to Mathav Murugan, for proposing the problem tackled in this paper (originally as a Masters Essay under his supervision), teaching me a great amount so that I could understand the necessary background information for it, and offering invaluable feedback throughout the writing process. I would also like to thank the reviewer for their helpful comments, and for suggesting I include the implication of H\"{o}lder continuity in Theorem \ref{ourMainResult3}.

\section{Preliminaries} \label{preliminariesSection}

Before we construct the auxiliary space, let us state a few facts that we will use throughout the paper. Most of these facts are straightforward to prove.
\LongVersionShortVersion
{ 
For the sake of completeness, every result that is not proved in this section is proved in Appendix \ref{preliminariesSectionProofs}.
}
{ 
For the sake of completeness, every result that is not proved in this section is proved in \cite[Appendix C.2]{longversion}, in the longer version of this paper that we put on arXiv.
} 

Occasionally throughout this paper, we deal with exponentially-distributed random variables. We say that a random variable $\xi$ is exponential($\lambda$), or that $\xi$ is exponential with \textit{rate} $\lambda$, if the probability density function of $\xi$ is $\lambda e^{-\lambda x}$ (in which case the mean of $\xi$ is $1/\lambda$).

Recall how the heat kernel of $\{X_t\}$ is defined in Definition \ref{HKphiDefinition}. In order to be a heat kernel, $p(t, x, y)$ must satisfy \eqref{HeatKernelDefinitionDistribution}-\eqref{HeatKernelDefinitionChapmanKolmogorov}. The following elementary proposition tells us that if $p(t, x, y)$ satisfies \eqref{HeatKernelDefinitionDistribution}, then it satisfies Chapman-Kolmogorov \eqref{HeatKernelDefinitionChapmanKolmogorov} almost everywhere.

\begin{prop}\label{AeChapmanKolmogorov}
If $p(t, x, y)$ is a kernel that satisfies \eqref{HeatKernelDefinitionDistribution}, then for any fixed $x \in M_0$ and $s, t > 0$, equation \eqref{HeatKernelDefinitionChapmanKolmogorov} holds for $\mu$-almost every $z \in M$.
\end{prop}

Our  only use of Proposition \ref{AeChapmanKolmogorov} is in Section \ref{HeatKernelSection1}, when we calculate the heat kernel of the jump process on the auxiliary space. We first propose a kernel $\phat$, then show that $\phat$ satisfies \eqref{HeatKernelDefinitionDistribution}-\eqref{HeatKernelDefinitionChapmanKolmogorov} for the auxiliary space. In our proof that $\phat$ satisfies \eqref{HeatKernelDefinitionChapmanKolmogorov}, we use Proposition \ref{AeChapmanKolmogorov} to show that $\phat$ satisfies Chapman-Kolmogorov almost everywhere, and then argue by further calculations that $\phat$ satisfies Chapman-Kolmogorov everywhere.

The following lemma gives a recipe for constructing a regular Dirichlet form on a general measure space $(\mathcal{X}, m)$. A collection $\mathcal{D}$ of functions on $\mathcal{X}$ is called \textit{Markovian} if for all $f \in \mathcal{D}$, the function
\begin{equation} \label{Markovian}
    f^*(x) := \max\{ 0, \min\{ f(x), 1 \} \}
\end{equation}
also belongs to $\mathcal{D}$. Let us use the notation $C_c(\mathcal{X})$ for the class of continuous, compactly supported (measurable) functions on $\mathcal{X}$.

\begin{lemma}\label{constructRegularDirichletForms}

Let $\mathcal{X}$ be a locally compact separable metric space. Let $m$ be a positive Radon measure on $\mathcal{X}$ with full support. Let $\diag_{\mathcal{X}}$ denote the diagonal of $\mathcal{X}$, and let $j$ be a symmetric non-negative measurable function on $\mathcal{X} \times \mathcal{X} \setminus \diag_{\mathcal{X}}$.
For all $f \in L^2(\mathcal{X}, m)$, let
\begin{equation*}
    \mathscr{E}(f) := \int_{\mathcal{X} \times\mathcal{X} \setminus \diag_{\mathcal{X}}} (f(x)-f(y))^2 \,j(x, y) \, m(dx) m(dy).
\end{equation*}
Let
\begin{equation*}
    \Fscrmax := \left\{ f \in L^2(\mathcal{X},m) : \mathscr{E}(f) < \infty \right\}.
\end{equation*}
For all $f, g \in \Fscrmax$, let
\begin{equation} \label{EscrFormula}
    \mathscr{E}(f, g) := \int_{\mathcal{X} \times\mathcal{X} \setminus \diag_{\mathcal{X}}} (f(x)-f(y))(g(x)-g(y)) \, j(x, y) \, m(dx) m(dy)
\end{equation}
and
\begin{equation*}
    \mathscr{E}_1(f, g) := \mathscr{E}(f, g) + \int_{\mathcal{X}} fg \, dm.
\end{equation*}
Let $\mathcal{D}$ be a Markovian subspace of $\Fscrmax \cap C_c(\mathcal{X})$, such that $\mathcal{D}$ is dense in $C_c(\mathcal{X})$, under the uniform norm. Let $\mathscr{F}$ be the closure of $\mathcal{D}$, under the $\mathscr{E}_1$-norm. Then we can conclude that $(\mathscr{E}, \mathscr{F})$ is a regular Dirichlet form, with $\mathcal{D}$ as a core.

\end{lemma}

The setting of Lemma \ref{constructRegularDirichletForms} is a general measure space $(\mathcal{X}, m)$, which may be distinct from $(M, \mu)$. This is because after we construct the auxiliary metric measure space $(\Mhat, \dhat, \muhat)$, we will apply Lemma \ref{constructRegularDirichletForms} to $(\mathcal{X}, m) = (\Mhat, \muhat)$ in order to construct the regular Dirichlet form $(\Ehat, \Fhat)$ for the auxiliary space.

Lemma \ref{constructRegularDirichletForms} is inspired by the recipe provided by \cite[Theorem 2.2(II)]{um} to construct regular Dirichlet forms on ultrametric spaces. Since the $(\Mhat, \dhat)$ that we construct in Section \ref{auxiliarySection} is not quite an ultrametric space, we could not simply cite \cite[Theorem 2.2(II)]{um}. Our Lemma \ref{constructRegularDirichletForms} is slightly more general.

Next, we will see that in the setting of our main results, each $x \in M$ is isolated if and only if it is an atom.
Let
\begin{equation} \label{MaMcDef}
    M_A = \left\{ x \in M : \mbox{$x$ is an isolated atom} \right\} \qquad\mbox{and} \qquad M_C = \left\{ x \in M : \mu(x)=0, D_x=0 \right\}.
\end{equation}
The subscripts $A$ and $C$ stand for ``atom" and ``continuous."
Obviously, $M_A$ and $M_C$ are disjoint subsets of $M$. If we assume $\QRVD$, it turns out that every $x \in M$ belongs to either $M_A$ or $M_C$:

\begin{prop} \label{Atom=Isolated}
If $(M, d, \mu)$ satisfies Assumption \ref{BasicAssumptions}(a) and $\QRVD$, then $M = M_A \cup M_C$.
\end{prop}

\begin{proof}

If there was an $x$ with $\mu(x) > 0$ but $D_x=0$, then just as in the proof of Proposition \ref{RvdImpliesNoAtoms}, we would have $V(x, \ell^{-n}) \leq c^{-n} V(x, 1)$ and $\mu(x) = \lim_{n\to\infty} V(x, \ell^{-n}) = 0$, a contradiction.

If there was an $x$ with $\mu(x)=0$ but $D_x>0$, then $x$ would not belong to the support of $\mu$, contradicting the assumption that $\mu$ is of full support.

Therefore, for all $x$, either $\mu(x)>0$ and $D_x>0$, or $\mu(x)=0$ and $D_x=0$.
\end{proof}

The assumption that $M = M_A \cup M_C$ is necessary for our construction of the auxiliary space. Luckily, Proposition \ref{Atom=Isolated} tells us that $M=M_A \cup M_C$ is guaranteed in the setting of our main results.

Suppose $(M, d, \mu, \mathcal{E}, \mathcal{F})$ satisfies Assumption \ref{BasicAssumptions}, and $(\mathcal{E}, \mathcal{F})$ admits a jump kernel.
For all $x \in M$ and $\rho>0$, let
\begin{equation} \label{MathcaljFormula}
    \mathcal{J}(x, \rho) := \int_{M \setminus B(x, \rho)} J(x, y) \mu(dy).
\end{equation}
Given $x$ and $\rho$, $\mathcal{J}(x, \rho)$ measures the rate at which jumps from $x$ of magnitude at least $\rho$ occur.
For all $x \in M_A$, let
\begin{equation} \label{v(x)Formula}
    v(x) := \int_{M \setminus \{x\}} J(x, y) \mu(dy)= \mathcal{J}(x, D_x)
\end{equation}
(where $D_x$ is as defined in \eqref{DxFormula}).

For all $E \subseteq M$, let $1_E$ denote the indicator
\begin{equation*}
    1_E(x) := \left\{ \begin{matrix}
        1 &:& \mbox{if $x \in E$}\\
        0 &:& \mbox{if $x \notin E$}.
    \end{matrix}\right.
\end{equation*}
For all $x \in M$, let $\delta_x:= 1_{\{x\}}$.

The following lemma is a collection of facts about $\mathcal{J}(x,\rho)$ and $v(x)$.

\begin{lemma} \label{JumpsAndAtomsPreliminaryFactsLemma}
If $(M, d, \mu, \mathcal{E}, \mathcal{F})$ satisfies Assumption \ref{BasicAssumptions}, $\phi$ is of regular growth, and $(\mathcal{E}, \mathcal{F})$ admits a jump kernel, then

(a) For all $x \in M_A$, the indicator $\delta_x$ belongs to $\mathcal{F}$ (the domain of the regular Dirichlet form $(\mathcal{E}, \mathcal{F})$).

(b) For all $x \in M_A$, $0 \leq v(x) < \infty$.

(c) If $\VD$ and $\Jleq$ hold, there exists a $C_{\mathcal{J}}>0$ such that
\begin{equation}\label{bigJumpsVd}
    \mathcal{J}(x, \rho) \leq \frac{C_{\mathcal{J}}}{\phi(\rho)} \qquad\mbox{for all $x \in M$, all $\rho>0$}.
\end{equation}

(d) If $\QRVD$ and $\Jgeq$ hold, there exists a $c_{\mathcal{J}}>0$ such that
\begin{equation}\label{bigJumpsQrvd}
    \mathcal{J}(x, \rho) \geq \frac{c_{\mathcal{J}}}{\phi(\rho)} \qquad\mbox{for all $x \in M$, all positive $\rho \geq D_x$}.
\end{equation}
\end{lemma}

Lemma \ref{JumpsAndAtomsPreliminaryFactsLemma}(c) allows us to put a useful upper bound on the quantity $v(x)$, defined in \eqref{v(x)Formula}.
For all $x \in M_A$, by applying Lemma \ref{JumpsAndAtomsPreliminaryFactsLemma}(c) to $\rho=D_x$, we have
\begin{equation} \label{bowtie1}
    v(x) \leq \frac{C_{\mathcal{J}}}{\phi(D_x)} \qquad\mbox{whenever we have $\VD$ and $\Jleq$.}
\end{equation}
Similarly, Lemma \ref{JumpsAndAtomsPreliminaryFactsLemma}(d) gives us the lower bound
\begin{equation} \label{bowtie2}
    v(x) \geq \frac{c_{\mathcal{J}}}{\phi(D_x)} \qquad\mbox{whenever we have $\QRVD$ and $\Jgeq$.}
\end{equation}

Lemma \ref{JumpsAndAtomsPreliminaryFactsLemma}(a) is used to establish our technical results relating the domains $\mathcal{F}$ and $\Fhat$ in Appendix \ref{DomainsAppendix}, which are in turn used to prove that $\CSJ$ for the original space implies $\CSJ$ for the auxiliary space, and that $\PI$ for the original space implies $\PI$ for the auxiliary space.

The upper bound on $v(x)$ from Lemma \ref{JumpsAndAtomsPreliminaryFactsLemma}(c) and \eqref{bowtie1} is used in the construction of the Dirichlet form $(\Ehat, \Fhat)$ in Section \ref{auxiliarySection}, and in our proofs of $\QEgeq$ and $\CSJ$ for the auxiliary space (under suitable conditions) in Sections \ref{EscapeSection} and \ref{CsPoincSection}.

The quantity $v(x)$ (and the fact that it is finite, by Lemma \ref{JumpsAndAtomsPreliminaryFactsLemma}(b)) also appears in Section \ref{HeatKernelSection1} in our formula for the heat kernel of the jump process on the auxiliary space.

We do not actually use Lemma \ref{JumpsAndAtomsPreliminaryFactsLemma}(d) or \eqref{bowtie2} anywhere in this paper, but point them out as analogs of Lemma \ref{JumpsAndAtomsPreliminaryFactsLemma}(c) and \eqref{bowtie1}.

\section{The auxiliary space} \label{auxiliarySection}

In this section, we construct the auxiliary metric measure space $(\Mhat, \dhat, \muhat, \Ehat, \Fhat)$. The construction will only make sense if we have the following assumptions: $(M, d, \mu, \mathcal{E}, \mathcal{F})$ satisfies Assumption \ref{BasicAssumptions}, $\phi$ is of regular growth, $(\mathcal{E}, \mathcal{F})$ admits a jump kernel, and $M=M_A \cup M_C$. Luckily, we will have each of these whenever we use the auxiliary space in our proofs of our main results.

Recall how $M_A$ and $M_C$ are defined in \eqref{MaMcDef}. The construction of $(\Mhat, \dhat, \muhat)$ that we are about to give relies on the assuumption that $M=M_A \cup M_C$. Luckily for us, Proposition \ref{Atom=Isolated} tells us that $M=M_A \cup M_C$ under the conditions of our main results.

The underlying set $\Mhat$ of the auxiliary space is defined by replacing each atom $x \in M_A$ with a continuous mass of points, which we will call $W_x$. Each non-atom $x \in M_C$ stays the same in the auxiliary space.
Recall that we must construct the auxiliary space such that if $(M, d, \mu)$ satisfies $\VD$ and $\QRVD$, then $(\Mhat, \dhat, \muhat)$ satisfies $\VD$ and $\RVD$. This way, the results of \cite{ckw} and \cite{ckw2} can be applied to $(\Mhat, \dhat, \muhat)$.

Before diving into the construction, let us motivate some of our choices.
Recall the notation $D_x$, defined in \eqref{DxFormula}, for the distance between an isolated point $x$ and its nearest neighbor. We would like to construct the auxiliary space in such a way that $\muhat(W_x)=\mu(x)$ for all $x \in M_A$, and $\dhat(z, z')$ is less than $D_x$ for any points $z, z' \in W_x \subseteq \Mhat$, but $\sup \left\{ \dhat(z, z') : z, z' \in W_x \right\}$ is on the order of $D_x$.

We would also like to construct $(\Mhat, \dhat, \muhat)$ so that escape times, jump kernels, heat kernels, and cut-off Sobolev inequalities are as convenient as possible to study.
Because of these considerations, we choose to construct $(\Mhat, \dhat, \muhat)$ so that for all atoms $x \in M_A$, the continuous mass $W_x$ in the auxiliary space has the \textit{ultrametric property}.
A metric space $(\mathcal{X}, d_{\mathcal{X}})$ is called \textit{ultrametric} if the following inequality (which is stronger than the triangle inequality) is satisfied:
\begin{equation*}
    d_{\mathcal{X}}(x, z) \leq \max \left\{ d_{\mathcal{X}}(x,y), d_{\mathcal{X}}(y, z) \right\} \qquad\mbox{for all $x, y, z \in \mathcal{X}$}.
\end{equation*} The ultrametric property has the following crucial consequence: 
\begin{equation} \label{UltrametricProperty}
    B_{\mathcal{X}}(x, r) = B_{\mathcal{X}}(x_0, r) \qquad\mbox{whenever $x_0 \in \mathcal{X}$, $r>0$, and $x \in B_{\mathcal{X}}(x_0, r)$}
\end{equation}
(where $B_{\mathcal{X}}(\cdot, \cdot)$ denotes the open balls in $(\mathcal{X}, d_{\mathcal{X}})$). See \cite{um} for an exploration of consequences of the ultrametric property for jump processes. Property \eqref{UltrametricProperty} will greatly simplify the work of Sections \ref{EscapeSection}-\ref{CsPoincSection}.

The simplest way to construct $(\Mhat, \dhat, \muhat)$ such that all of these desired properties hold is to let $W_x$ (for each $x \in M_A$) be a binary-tree-like structure, whose elements are of the form $(x, w)$, where $w$ is an infinite binary string (or equivalently, a path to infinity in an infinite binary tree). Each non-atom $x \in M_C$ stays the same in $\Mhat$, but we will also use the notation $W_x$ to refer to the singleton $\{x\}$ for all $x \in M_C$. This way, each $x \in M$ has an associated $W_x \subseteq \Mhat$: for $x \in M_A$, $W_x$ is a binary-tree-like set, while for $x \in M_C$, $W_x$ is a singleton.

To illustrate the idea behind our construction, we include the following rough drawings of the original metric measure space and the auxiliary metric measure space for some examples of $(M, d, \mu)$ that satisfy $\VD$ and $\QRVD$ but not $\RVD$. In these examples, most elements of $M$ and $\Mhat$ are colored red, while some specifed elements of $M$ and $\Mhat$ are colored blue. Hopefully, these figures help make it clear what an element of the auxiliary space looks like.

\begin{example}\label{IntegerLatticeExample}\textbf{The integer lattice.} Suppose $M=\mathbb{Z}^1$, $d$ is the Euclidean metric, and $\mu$ is the counting measure. 
The original space and auxiliary space for this $(M, d, \mu)$ are shown in Figure \ref{IntegerLatticeFigure}.
Every $x \in M$ is an atom, so every point is replaced with a rooted tree-like structure in the auxiliary space. Each infinite path to the bottom in one of these trees is an element of $\Mhat$. 


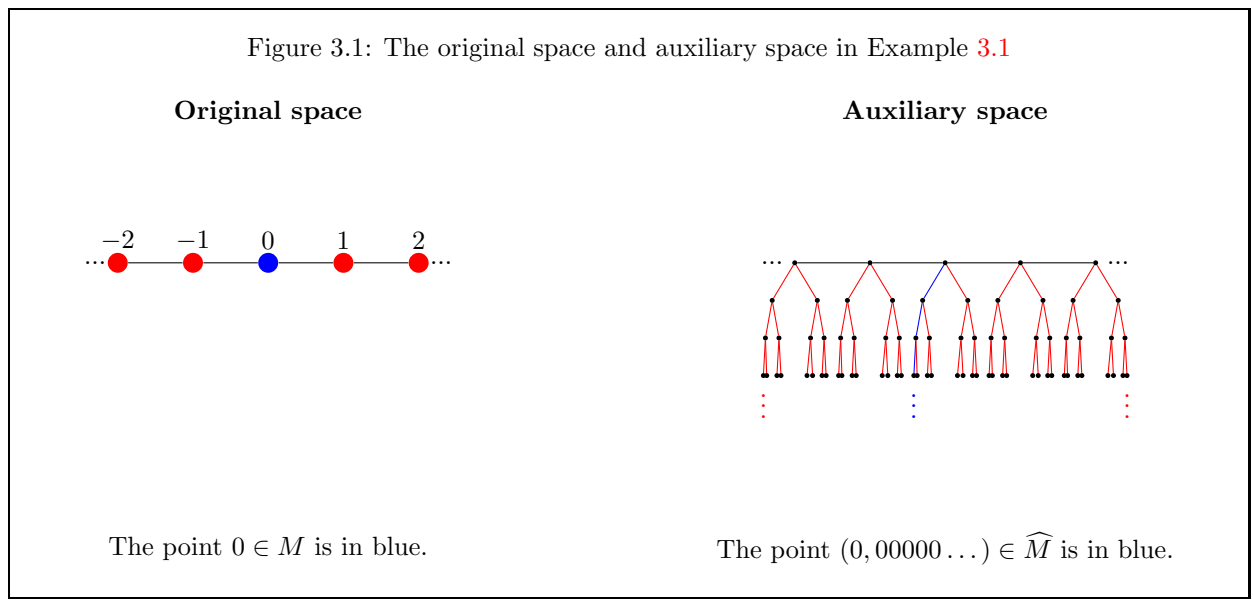
\begin{figure}[ht]
\begin{framed}
\captionof{figure}{The original space and auxiliary space in Example \ref{IntegerLatticeExample}}
\label{IntegerLatticeFigure}
    \centering
    \begin{tikzpicture}
        \node (original) at (0, 2) {\textbf{Original space}};
        
        \node at (-2, .3) {$-2$};
        \node at (-1, .3) {$-1$};
        \node at (0, .3) {$0$};
        \node at (1, .3) {$1$};
        \node at (2, .3) {$2$};
        
        \node (leftDotDotDot1) at (-2.3, 0) {...};
        \node[shape=circle, fill=red, scale=0.8] (neg2) at (-2, 0) {};
        \node[shape=circle, fill=red, scale=0.8] (neg1) at (-1, 0) {};
        \node[shape=circle, fill=blue, scale=0.8] (0) at (0, 0) {};
        \node[shape=circle, fill=red, scale=0.8] (1) at (1, 0) {};
        \node[shape=circle, fill=red, scale=0.8] (2) at (2, 0) {};
        \node (rightDotDotDot1) at (2.3, 0) {...};
        
        \draw (neg2) -- (neg1) -- (0) -- (1) -- (2);
        
        \node (auxiliary) at (9, 2) {\textbf{Auxiliary space}};
        
        \node (leftDotDotDot2) at (6.7, 0) {...};
        \node[shape=circle, fill=black, scale=0.2] (auxneg2) at (7, 0) {};
        \node[shape=circle, fill=black, scale=0.2] (auxneg1) at (8, 0) {};
        \node[shape=circle, fill=black, scale=0.2] (aux0) at (9, 0) {};
        \node[shape=circle, fill=black, scale=0.2] (aux1) at (10, 0) {};
        \node[shape=circle, fill=black, scale=0.2] (aux2) at (11, 0) {};
        \node (rightDotDotDot2) at (11.3, 0) {...};
        
        \draw (auxneg2) -- (auxneg1) -- (aux0) -- (aux1) -- (aux2);
        
        \node[shape=circle, fill=black, scale=0.2] (auxneg2L) at (6.7, -0.5) {};
        \node[shape=circle, fill=black, scale=0.2] (auxneg2R) at (7.3, -0.5) {};
        \node[shape=circle, fill=black, scale=0.2] (auxneg2LL) at (6.61, -1) {};
        \node[shape=circle, fill=black, scale=0.2] (auxneg2LR) at (6.79, -1) {};
        \node[shape=circle, fill=black, scale=0.2] (auxneg2RL) at (7.21, -1) {};
        \node[shape=circle, fill=black, scale=0.2] (auxneg2RR) at (7.39, -1) {};
        \node[shape=circle, fill=black, scale=0.2] (auxneg2LLL) at (6.583, -1.5) {};
        \node[shape=circle, fill=black, scale=0.2] (auxneg2LLR) at (6.627, -1.5) {};
        \node[shape=circle, fill=black, scale=0.2] (auxneg2LRL) at (6.763, -1.5) {};
        \node[shape=circle, fill=black, scale=0.2] (auxneg2LRR) at (6.817, -1.5) {};
        \node[shape=circle, fill=black, scale=0.2] (auxneg2RLL) at (7.163, -1.5) {};
        \node[shape=circle, fill=black, scale=0.2] (auxneg2RLR) at (7.237, -1.5) {};
        \node[shape=circle, fill=black, scale=0.2] (auxneg2RRL) at (7.363, -1.5) {};
        \node[shape=circle, fill=black, scale=0.2] (auxneg2RRR) at (7.417, -1.5) {};
        \draw[color=red] (auxneg2) -- (auxneg2L) -- (auxneg2LL) -- (auxneg2LLL);
        \draw[color=red] (auxneg2LL) -- (auxneg2LLR);
        \draw[color=red] (auxneg2L) -- (auxneg2LR) -- (auxneg2LRL); 
        \draw[color=red] (auxneg2LR) -- (auxneg2LRR);
        \draw[color=red] (auxneg2) -- (auxneg2R) -- (auxneg2RL) -- (auxneg2RLL);
        \draw[color=red] (auxneg2RL) -- (auxneg2RLR);
        \draw[color=red] (auxneg2R) -- (auxneg2RR) -- (auxneg2RRL); 
        \draw[color=red] (auxneg2RR) -- (auxneg2RRR);
        
        \node[shape=circle, fill=black, scale=0.2] (auxneg1L) at (7.7, -0.5) {};
        \node[shape=circle, fill=black, scale=0.2] (auxneg1R) at (8.3, -0.5) {};
        \node[shape=circle, fill=black, scale=0.2] (auxneg1LL) at (7.61, -1) {};
        \node[shape=circle, fill=black, scale=0.2] (auxneg1LR) at (7.79, -1) {};
        \node[shape=circle, fill=black, scale=0.2] (auxneg1RL) at (8.21, -1) {};
        \node[shape=circle, fill=black, scale=0.2] (auxneg1RR) at (8.39, -1) {};
        \node[shape=circle, fill=black, scale=0.2] (auxneg1LLL) at (7.583, -1.5) {};
        \node[shape=circle, fill=black, scale=0.2] (auxneg1LLR) at (7.627, -1.5) {};
        \node[shape=circle, fill=black, scale=0.2] (auxneg1LRL) at (7.763, -1.5) {};
        \node[shape=circle, fill=black, scale=0.2] (auxneg1LRR) at (7.817, -1.5) {};
        \node[shape=circle, fill=black, scale=0.2] (auxneg1RLL) at (8.163, -1.5) {};
        \node[shape=circle, fill=black, scale=0.2] (auxneg1RLR) at (8.237, -1.5) {};
        \node[shape=circle, fill=black, scale=0.2] (auxneg1RRL) at (8.363, -1.5) {};
        \node[shape=circle, fill=black, scale=0.2] (auxneg1RRR) at (8.417, -1.5) {};
        \draw[color=red] (auxneg1) -- (auxneg1L) -- (auxneg1LL) -- (auxneg1LLL);
        \draw[color=red] (auxneg1LL) -- (auxneg1LLR);
        \draw[color=red] (auxneg1L) -- (auxneg1LR) -- (auxneg1LRL); 
        \draw[color=red] (auxneg1LR) -- (auxneg1LRR);
        \draw[color=red] (auxneg1) -- (auxneg1R) -- (auxneg1RL) -- (auxneg1RLL);
        \draw[color=red] (auxneg1RL) -- (auxneg1RLR);
        \draw[color=red] (auxneg1R) -- (auxneg1RR) -- (auxneg1RRL); 
        \draw[color=red] (auxneg1RR) -- (auxneg1RRR);
        
        \node[shape=circle, fill=black, scale=0.2] (aux0L) at (8.7, -0.5) {};
        \node[shape=circle, fill=black, scale=0.2] (aux0R) at (9.3, -0.5) {};
        \node[shape=circle, fill=black, scale=0.2] (aux0LL) at (8.61, -1) {};
        \node[shape=circle, fill=black, scale=0.2] (aux0LR) at (8.79, -1) {};
        \node[shape=circle, fill=black, scale=0.2] (aux0RL) at (9.21, -1) {};
        \node[shape=circle, fill=black, scale=0.2] (aux0RR) at (9.39, -1) {};
        \node[shape=circle, fill=black, scale=0.2] (aux0LLL) at (8.583, -1.5) {};
        \node[shape=circle, fill=black, scale=0.2] (aux0LLR) at (8.627, -1.5) {};
        \node[shape=circle, fill=black, scale=0.2] (aux0LRL) at (8.763, -1.5) {};
        \node[shape=circle, fill=black, scale=0.2] (aux0LRR) at (8.817, -1.5) {};
        \node[shape=circle, fill=black, scale=0.2] (aux0RLL) at (9.163, -1.5) {};
        \node[shape=circle, fill=black, scale=0.2] (aux0RLR) at (9.237, -1.5) {};
        \node[shape=circle, fill=black, scale=0.2] (aux0RRL) at (9.363, -1.5) {};
        \node[shape=circle, fill=black, scale=0.2] (aux0RRR) at (9.417, -1.5) {};
        \draw[color=blue] (aux0) -- (aux0L) -- (aux0LL) -- (aux0LLL);
        \draw[color=red] (aux0LL) -- (aux0LLR);
        \draw[color=red] (aux0L) -- (aux0LR) -- (aux0LRL); 
        \draw[color=red] (aux0LR) -- (aux0LRR);
        \draw[color=red] (aux0) -- (aux0R) -- (aux0RL) -- (aux0RLL);
        \draw[color=red] (aux0RL) -- (aux0RLR);
        \draw[color=red] (aux0R) -- (aux0RR) -- (aux0RRL); 
        \draw[color=red] (aux0RR) -- (aux0RRR);
        
        \node[shape=circle, fill=black, scale=0.2] (aux1L) at (9.7, -0.5) {};
        \node[shape=circle, fill=black, scale=0.2] (aux1R) at (10.3, -0.5) {};
        \node[shape=circle, fill=black, scale=0.2] (aux1LL) at (9.61, -1) {};
        \node[shape=circle, fill=black, scale=0.2] (aux1LR) at (9.79, -1) {};
        \node[shape=circle, fill=black, scale=0.2] (aux1RL) at (10.21, -1) {};
        \node[shape=circle, fill=black, scale=0.2] (aux1RR) at (10.39, -1) {};
        \node[shape=circle, fill=black, scale=0.2] (aux1LLL) at (9.583, -1.5) {};
        \node[shape=circle, fill=black, scale=0.2] (aux1LLR) at (9.627, -1.5) {};
        \node[shape=circle, fill=black, scale=0.2] (aux1LRL) at (9.763, -1.5) {};
        \node[shape=circle, fill=black, scale=0.2] (aux1LRR) at (9.817, -1.5) {};
        \node[shape=circle, fill=black, scale=0.2] (aux1RLL) at (10.163, -1.5) {};
        \node[shape=circle, fill=black, scale=0.2] (aux1RLR) at (10.237, -1.5) {};
        \node[shape=circle, fill=black, scale=0.2] (aux1RRL) at (10.363, -1.5) {};
        \node[shape=circle, fill=black, scale=0.2] (aux1RRR) at (10.417, -1.5) {};
        \draw[color=red] (aux1) -- (aux1L) -- (aux1LL) -- (aux1LLL);
        \draw[color=red] (aux1LL) -- (aux1LLR);
        \draw[color=red] (aux1L) -- (aux1LR) -- (aux1LRL); 
        \draw[color=red] (aux1LR) -- (aux1LRR);
        \draw[color=red] (aux1) -- (aux1R) -- (aux1RL) -- (aux1RLL);
        \draw[color=red] (aux1RL) -- (aux1RLR);
        \draw[color=red] (aux1R) -- (aux1RR) -- (aux1RRL); 
        \draw[color=red] (aux1RR) -- (aux1RRR);
        
        \node[shape=circle, fill=black, scale=0.2] (aux2L) at (10.7, -0.5) {};
        \node[shape=circle, fill=black, scale=0.2] (aux2R) at (11.3, -0.5) {};
        \node[shape=circle, fill=black, scale=0.2] (aux2LL) at (10.61, -1) {};
        \node[shape=circle, fill=black, scale=0.2] (aux2LR) at (10.79, -1) {};
        \node[shape=circle, fill=black, scale=0.2] (aux2RL) at (11.21, -1) {};
        \node[shape=circle, fill=black, scale=0.2] (aux2RR) at (11.39, -1) {};
        \node[shape=circle, fill=black, scale=0.2] (aux2LLL) at (10.583, -1.5) {};
        \node[shape=circle, fill=black, scale=0.2] (aux2LLR) at (10.627, -1.5) {};
        \node[shape=circle, fill=black, scale=0.2] (aux2LRL) at (10.763, -1.5) {};
        \node[shape=circle, fill=black, scale=0.2] (aux2LRR) at (10.817, -1.5) {};
        \node[shape=circle, fill=black, scale=0.2] (aux2RLL) at (11.163, -1.5) {};
        \node[shape=circle, fill=black, scale=0.2] (aux2RLR) at (11.237, -1.5) {};
        \node[shape=circle, fill=black, scale=0.2] (aux2RRL) at (11.363, -1.5) {};
        \node[shape=circle, fill=black, scale=0.2] (aux2RRR) at (11.417, -1.5) {};
        \draw[color=red] (aux2) -- (aux2L) -- (aux2LL) -- (aux2LLL);
        \draw[color=red] (aux2LL) -- (aux2LLR);
        \draw[color=red] (aux2L) -- (aux2LR) -- (aux2LRL); 
        \draw[color=red] (aux2LR) -- (aux2LRR);
        \draw[color=red] (aux2) -- (aux2R) -- (aux2RL) -- (aux2RLL);
        \draw[color=red] (aux2RL) -- (aux2RLR);
        \draw[color=red] (aux2R) -- (aux2RR) -- (aux2RRL); 
        \draw[color=red] (aux2RR) -- (aux2RRR);

        \node[color=red] at (6.583, -1.8) {$\vdots$};
        \node[color=red] at (11.417, -1.8) {$\vdots$};
        \node[color=blue] at (8.583, -1.8) {$\vdots$};
        
        \node at (0, -3.8) {The point $0 \in M$ is in blue.};
        \node at (9, -3.8) {The point $(0, 00000\dots) \in \Mhat$ is in blue.};
        
    \end{tikzpicture}
\end{framed}
\end{figure}


\end{example}

\begin{example}\label{MixedExample}\textbf{The real line, with $(-1, 1)$ replaced by an atom.}
Let $M = (-\infty, -1] \cup \{0\} \cup [1, \infty)$. Let $d$ be the Euclidean metric restricted to $M$. Let $\mu$ be the measure given by
\begin{equation*}
    \mu(A) = \left\{ \begin{matrix}
        m(A) &:& \mbox{if $ 0 \notin A$}\\
        m(A) + 2 &:& \mbox{if $0 \in A$}
    \end{matrix}\right.
\end{equation*}
for all Borel $A \subseteq M$, where $m$ is the Lebesgue measure. The original space and auxiliary space for this example are shown in Figure \ref{MixedFigure}. In this case, $0$ is the only atom in $M$, so $0$ is replaced by a tree-like structure while every other point of $M$ stays the same in the auxiliary metric measure space. Elements of the auxiliary space are either points $(0, w)$ (where $w\in W$) or $x$ (where $x \in (-\infty, -1] \cup [1, \infty)$).


\begin{figure}[ht]
\begin{framed}
\captionof{figure}{The original space and auxiliary space in Example \ref{MixedExample}}
\label{MixedFigure}
    \centering
    \begin{tikzpicture}
        \node (original) at (0, 2) {\textbf{Original space}};
        \node (auxiliary) at (8, 2) {\textbf{Auxiliary space}};
        
        \node at (-1, .3) {$-1$};
        \node at (0, .3) {$0$};
        \node at (1, .3) {$1$};
        \node at (2, .3) {$2$};
        
        \node (neg3) at (-3, 0) {};
        \node[shape=circle, fill=red, scale=0.5] (neg1) at (-1, 0) {};
        \node[shape=circle, fill=blue, scale=0.8] (0) at (0, 0) {};
        \node[shape=circle, fill=red, scale=0.5] (1) at (1, 0) {};
        \node[shape=circle, fill=blue, scale=0.5] (2) at (2, 0) {};
        \node (3) at (3, 0) {};
        \path[->, color=red, scale=0.3] (1) edge node {} (3);
        \path[->, color=red, scale=0.3] (neg1) edge node {} (neg3);
        
        \node at (7, .3) {$-1$};
        \node at (9, .3) {$1$};
        \node at (10, .3) {$2$};
        
        \node (auxneg3) at (5, 0) {};
        \node[shape=circle, fill=red, scale=0.5] (auxneg1) at (7, 0) {};
        \node[shape=circle, fill=black, scale=0.3] (aux0) at (8, 0) {};
        \node[shape=circle, fill=red, scale=0.5] (aux1) at (9, 0) {};
        \node[shape=circle, fill=blue, scale=0.5] (aux2) at (10, 0) {};
        \node (aux3) at (11, 0) {};
        \path[->, color=red, scale=0.3] (aux1) edge node {} (aux3);
        \path[->, color=red, scale=0.3] (auxneg1) edge node {} (auxneg3);
        
        \node[shape=circle, fill=black, scale=0.2] (aux0L) at (7.7, -0.5) {};
        \node[shape=circle, fill=black, scale=0.2] (aux0R) at (8.3, -0.5) {};
        \node[shape=circle, fill=black, scale=0.2] (aux0LL) at (7.61, -1) {};
        \node[shape=circle, fill=black, scale=0.2] (aux0LR) at (7.79, -1) {};
        \node[shape=circle, fill=black, scale=0.2] (aux0RL) at (8.21, -1) {};
        \node[shape=circle, fill=black, scale=0.2] (aux0RR) at (8.39, -1) {};
        \node[shape=circle, fill=black, scale=0.2] (aux0LLL) at (7.583, -1.5) {};
        \node[shape=circle, fill=black, scale=0.2] (aux0LLR) at (7.627, -1.5) {};
        \node[shape=circle, fill=black, scale=0.2] (aux0LRL) at (7.763, -1.5) {};
        \node[shape=circle, fill=black, scale=0.2] (aux0LRR) at (7.817, -1.5) {};
        \node[shape=circle, fill=black, scale=0.2] (aux0RLL) at (8.163, -1.5) {};
        \node[shape=circle, fill=black, scale=0.2] (aux0RLR) at (8.237, -1.5) {};
        \node[shape=circle, fill=black, scale=0.2] (aux0RRL) at (8.363, -1.5) {};
        \node[shape=circle, fill=black, scale=0.2] (aux0RRR) at (8.417, -1.5) {};
        \draw[color=blue] (aux0) -- (aux0L) -- (aux0LL) -- (aux0LLL);
        \draw[color=red] (aux0LL) -- (aux0LLR);
        \draw[color=red] (aux0L) -- (aux0LR) -- (aux0LRL); 
        \draw[color=red] (aux0LR) -- (aux0LRR);
        \draw[color=red] (aux0) -- (aux0R) -- (aux0RL) -- (aux0RLL);
        \draw[color=red] (aux0RL) -- (aux0RLR);
        \draw[color=red] (aux0R) -- (aux0RR) -- (aux0RRL); 
        \draw[color=red] (aux0RR) -- (aux0RRR);
        
        \node[color=blue] at (7.583, -1.8) {$\vdots$};
        \node[color=red] at (8.417, -1.8) {$\vdots$};
        
        \node at (0, -3.8) {The points $0 \in M$ and $2 \in M$ are in blue.};
        \node at (8.3, -3.8) {The points $(0, 00000\dots) \in \Mhat$ and $2 \in \Mhat$ are in blue.};
    \end{tikzpicture}
\end{framed}
\end{figure}
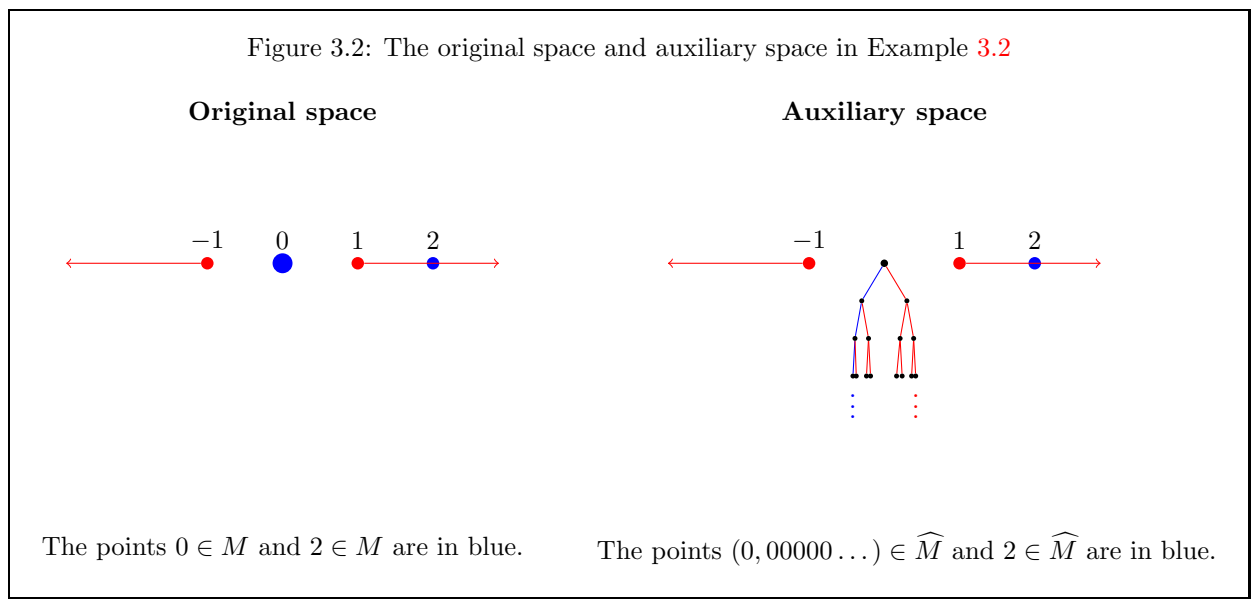
\end{example}


In Section \ref{convenientSection}, we construct a simple ultrametric measure space $(W, \rho^D, \nu)$. The underlying set $W$ is the set of infininte binary strings $w=(w(1), w(2), w(3), \dots)$ such that $w(i) \in \{0, 1\}$ for all $i \in \mathbb{N}$. The metric $\rho^D$ depends on a parameter diameter $D>0$. The measure $\nu$ is the uniform self-similar probability measure (or the ``coin-flipping measure") on $W$. As a metric space, $(W, \rho^D)$ satisfies the ultrametric property \eqref{UltrametricProperty}.

In Section \ref{auxiliaryMmsConstruction}, we construct the auxiliary metric space $(\Mhat, \dhat, \muhat)$, by replacing each atom $x \in M$ with a copy of $(W, \rho^{D_x}, \mu(x) \cdot \nu)$, where $D_x$ is as defined in \eqref{DxFormula} and $\mu(x) \cdot \nu$ refers to the measure $\nu$ multiplied by $\mu(x)$.

In Section \ref{auxiliaryEFconstruction}, we construct the regular Dirichlet form $(\Ehat, \Fhat)$ on $L^2(\Mhat, \muhat)$.

\subsection{The convenient tree-like ultrametric space \texorpdfstring{$(W, \rho^D, \nu)$}{TEXT}} \label{convenientSection}

Let $W$ be the set of infinite binary strings $w = (w(1), w(2), w(3), \dots)$, where $w(i) \in \{0, 1\}$ for all $i \in \mathbb{N}$. For all $w \in W$, $m\in\mathbb{N}$, let
\begin{dmath} \label{AwmDef}
    A^w_m = {\left\{ w' \in W : \mbox{$m$ is the first index of disagreement between $w'$ and $w$} \right\}}
    ={\left\{ w' \in W : \mbox{$w'(i) = w(i)$ for all $i < m$, but $w'(m) \neq w(m)$} \right\}}.
\end{dmath}
For all $w \in W$, $m \in \mathbb{Z}_+$, let
\begin{dmath} \label{EwmDef}
    E^w_m = {\left\{ w \in W : \mbox{$w'(i) = w(i)$ for all $i \leq m$} \right\}}
    = \{w\} \cup \bigcup_{j>m} A^w_j.
\end{dmath}

Fix $D>0$. In this subsection, we construct a metric measure space $(W, \rho^D, \nu)$, where the metric $\rho^D$ depends on $D$. This space will have total measure $\nu(W)=1$ and diameter less than (but on the order of) $D$. The open balls of $(W, \rho^D)$ are the sets $E^w_m$, and the spheres of $(W, \rho^D)$ are the sets $A^w_m$.
Recall that when we construct $(\Mhat, \dhat, \muhat)$ in the next subsection, we replace each atom $x \in M$ with a copy of $(W, \rho^{D_x}, \mu(x) \cdot \nu)$, where $D_x$ is as defined in \eqref{DxFormula} and $\mu(x) \cdot \nu$ refers to the measure $\nu$ multiplied by $\mu(x)$. Therefore, the reader may think of $D$ as playing the role of $D_x$ for some atom $x$.

Let $\mathscr{B}_W$ be the minimal $\sigma$-field of $W$ that contains $E^w_m$ for all $w, m$. Let $\nu$ be the measure on $(W, \mathscr{B}_W)$ such that
\begin{equation} \label{NuDef}
    \nu(E^w_m) = 2^{-m} \qquad\mbox{for all $w, m$}.
\end{equation}
(This is the uniform self-similar probability measure.) Note that
\begin{dmath} \label{NuOnSpheres}
    \nu(A^w_m) = \nu \left( E^w_{m-1} \setminus E^w_m \right) = \nu(E^w_{m-1}) - \nu(E^w_m) = 2^{-(m-1)} - 2^{-m} = 2^{-m} \qquad\mbox{for all $w \in W$, $m \in\mathbb{Z}_+$}.
\end{dmath}

For all $m \in \mathbb{Z}_+$, let
\begin{equation} \label{dDm}
    d^D_m := \phi^{-1}\left( \frac{\phi(D)}{2^m} \right).
\end{equation}
and let $\rho^D$ be the metric
\begin{equation}\label{RhoFormula}
    \rho^D(w, w') = \left\{ \begin{matrix}
        0 &:& \mbox{if $w=w'$}\\
        \\
        d^D_m &:& \mbox{if $w' \in A^w_m, m \in \mathbb{N}$}.
    \end{matrix}\right.
\end{equation}
Note that $(W, \rho^D)$ is ultrametric. The diameter of $(W, \rho^D)$ is not quite $D$, but
\begin{equation*}
    \diam(W, \rho^D) = d^D_1 = \phi^{-1}\left(\frac{\phi(D)}{2}\right).
\end{equation*}

Let $\Btilde^D(w, r)$ denote the open ball in $(W, \rho^D)$ with center $w$ and radius $r$, and let $\Vtilde^D(w, r) := \nu(\Btilde^D(w, r))$ denote its volume. If $r \leq D$, then
\begin{equation} \label{VtildeIntermsofM}
    \left.\begin{matrix}
        \Btilde^D(w, r) = E^w_m\\
        \Vtilde^D(w, r) = 2^{-m}
    \end{matrix}\quad\right\}\qquad\mbox{for the $m$ such that $d^D_{m+1} < r \leq d^D_m$}.
\end{equation}
As an explicit function of $r$, we have
\begin{equation} \label{VtildeFloorFormula}
    \Vtilde^D(r) = 2^{-\floor{\log_2 \left( \frac{\phi(D)}{\phi(r)} \right)}} \in \left[ \frac{\phi(r)}{\phi(D)}, 2 \frac{\phi(r)}{\phi(D)} \right) \qquad\mbox{for $0 < r \leq D$}.
\end{equation}
If $r>D$, then $\Btilde^D(w, r) = W$ and $\Vtilde^D(w, r) = 1$.

Since $\Vtilde^D(w, r)$ does not depend on $w$, we will just call it $\Vtilde^D(r)$.

Since the open balls in $(W, \rho^D)$ are precisely the sets $E^w_m$, no matter what the value of $D$ is, $\mathscr{B}_W$ is the Borel $\sigma$-field of $(W, \rho^D)$.

\subsection{The auxiliary metric measure space \texorpdfstring{$(\Mhat, \dhat, \muhat)$}{TEXT}} \label{auxiliaryMmsConstruction}

We remind the reader that the construction of $(\Mhat, \dhat, \muhat)$ we are about to give relies on the assumption that $M=M_A \cup M_C$. For all $x \in M$, let
\begin{equation*}
    W_x := \left\{ \begin{matrix}
        \{x\} \times W &:& \mbox{if $x \in M_A$}\\
        \\
        \{x\} &:& \mbox{if $x \in M_C$}
    \end{matrix}\right.
\end{equation*}
where $W$ is the set from Section \ref{convenientSection}.
Let
\begin{equation*}
    \Mhat := \bigcup_{x \in M} W_x = (M_A \times W) \cup M_C.
\end{equation*}
In other words, $\widehat{M}$ contains all the points in $M$ that are not isolated atoms, but replaces each isolated atom $x \in M_A$ with $\{x\} \times W$.

We will usually use the variable names $x$ and $y$ for elements of $M$, $w$ for elements of $W$, and $z$ for elements of $\Mhat$.

Let us endow $\widehat{M}$ with the metric
\begin{equation}\label{DefinitionOfDhat}
    \widehat{d}(z_1, z_2) = \left\{ \begin{matrix}
        d(x, y) &:& \mbox{if $\pi(z_1) = x \neq y = \pi(z_2)$ for some $x, y \in M$}\\
        \\
        \rho^{D_x}(w_1, w_2) &:& \mbox{if $z_1=(x, w_1)$ and $z_2=(x, w_2)$ for some $x \in M_A$ and $w_1, w_2 \in W$}
    \end{matrix}\right.
\end{equation}
where $D_x$ is as defined in \eqref{DxFormula} and $\rho^{D_x}$ is as defined in \eqref{RhoFormula}.

Let $\borel_M$ be the Borel $\sigma$-field of $(M, d)$, and let $\borel_{\Mhat}$ be the Borel $\sigma$-field of $(\Mhat, \dhat)$.
Let us now define the measure $\muhat$ on $(\Mhat, \borel_{\Mhat})$. Let $\mu_{M_A}$ and $\mu_{M_C}$ be the restrictions of $\mu$ to $M_A$ and $M_C$ respectively. That is,
\begin{equation*}
    \mu_{M_A}(E) = \mu(E \cap M_A) \quad \mbox{and} \quad \mu_{M_C}(E) = \mu(E \cap M_c) \qquad\mbox{for all $E \in \borel_M$}.
\end{equation*}
Then let $\muhat$ be the measure on $\Mhat = (M_A \times W) \cup M_C$ defined by
\begin{equation} \label{MuhatDefinition}
    \muhat := (\mu_{M_A} \times \nu) + \mu_{M_C}.
\end{equation}
(where $\nu$ is the uniform self-similar measure on $W$, as defined in \eqref{NuDef}).
Equivalently, for all $S$ belonging to the Borel $\sigma$-field of $\Mhat$,
\begin{equation*}
    \widehat{\mu}(S) = \left( \sum_{x \in M_A} \mu(x) \nu \left( \left\{ w \in W: (x, w) \in S \right\}\right) \right) + \mu(S \cap M_C).
\end{equation*}

The triple $(\Mhat, \dhat, \muhat)$ forms a metric measure space.
Let us use $\Bhat$ and $\Vhat$ refer to balls and their volumes in $(\Mhat, \dhat, \muhat)$. A consequence of \eqref{MuhatDefinition} is that
\begin{equation} \label{MeasuresAgree}
    \muhat(\pi^{-1}(E)) = \mu(E) \qquad\mbox{for all Borel $E \subseteq M$}.
\end{equation}
Recall how $\pi : \Mhat \to M$ is the projection that maps each point in $W_x$ to $x$. If $f$ is a function on $M$, then $f \circ \pi$ is the function on $\Mhat$ that is constant on $W_x$ for each $x$, assigning every point in $W_x$ the value $f(x)$. A consequence of \eqref{MeasuresAgree} is
\begin{equation} \label{integralsAgree}
    \int_{\pi^{-1}(E)} (f \circ \pi) \, d\muhat = \int_E f \, d\mu \qquad\mbox{for all Borel $E \subseteq M$, integrable $f : M \to \mathbb{R}$}.
\end{equation}
Recall that $\Btilde^D(\cdot, \cdot)$ and $\Vtilde^D( \cdot, \cdot)$ denote balls and their volumes in the metric measure space $(W, \rho^D, \nu)$ parametrized by $D>0$. A simple application of \eqref{DefinitionOfDhat} tells us that for all $x \in M$ and $z_0 \in W_x$,
\begin{equation} \label{BhatFormula}
    \Bhat(z_0, r) = \left\{ \begin{matrix}
        \pi^{-1}(B(x, r)) &:& \mbox{if $r \geq D_x$}\\
        \{x\} \times \Btilde^{D_x}(w_0, r) &:& \mbox{if $r \leq D_x$ and $z_0 = (x, w_0)$}
    \end{matrix}\right.
\end{equation}
By using \eqref{MuhatDefinition} and \eqref{MeasuresAgree} to calculate the volumes of the sets in \eqref{BhatFormula},
\begin{equation} \label{VhatFormula}
    \Vhat(z_0, r) = \left\{ \begin{matrix}
        V(x, r) &:& \mbox{if $r \geq D_x$}\\
        \mu(x) \Vtilde^{D_x}(r) &:& \mbox{if $r \leq D_x$}
    \end{matrix}\right. \qquad \mbox{for all $x \in M$, $z_0 \in W_x$, $r>0$}.
\end{equation}
Note that $\left(W_x, \left.\dhat\right|_{W_x}\right)$ is ultrametric:
\begin{equation} \label{WxUltrametric}
    \dhat(z, z'') \leq \max \left\{ \dhat(z, z'), \dhat(z', z'') \right\} \qquad\mbox{for all $z, z', z'' \in W_x$, for all $x \in M_A$}.
\end{equation}
Consequently,
\begin{equation} \label{WxUltrametricBalls}
    \Bhat(z_0, r) = \Bhat(z, r) \qquad\mbox{for all $z_0 \in \Mhat$, $r>0$, and $z \in \Bhat(z_0, r)$}.
\end{equation}
(Note however that $(\Mhat, \dhat)$ is not necessarily an ultrametric space itself, even if $\left(W_x, \left.\dhat\right|_{W_x}\right)$ is for all $x \in M_A$.)

In the next proposition, we show that $(\Mhat, \dhat, \muhat)$ inherits the regularity assumptions from $(M, d, \mu)$.

\begin{prop} \label{volumeGrowthProp}
Suppose $(M, d, \mu)$ satisfies Assumption \ref{BasicAssumptions}(a), $M=M_A \cup M_C$, and $\phi$ is a function of regular growth. Then

(a) $(\Mhat, \dhat, \muhat)$ satisfies Assumption \ref{BasicAssumptions}(a).

(b) If $(M, d, \mu)$ satisfies $\VD$, then so does $(\Mhat, \dhat, \muhat)$.

(c) If $(M, d, \mu)$ satisfies $\QRVD$, then $(\Mhat, \dhat, \muhat)$ satisfies $\RVD$.
\end{prop}

\begin{proof}
\LongVersionShortVersion
{ 
(a) is proven in Appendix \ref{AppendixAssumption1}. The proof of (a) is not deep, but requires several lemmas.
}
{ 
The proof of (a) is not deep, but requires several lemmas. The interested reader may find it in \cite[Appendix B.4]{longversion}.
} 

In order to prove (b) and (c), first observe that if $z_0 \in W_x$ for some $x \in M_A$, and $0 < r \leq R \leq D_x$, then by \eqref{VhatFormula} and \eqref{VtildeFloorFormula},
\begin{equation*}
    \frac{\Vhat(z_0, R)}{\Vhat(z_0, r)} = \frac{\Vtilde^{D_x}(R)}{\Vtilde^{D_x}(r)} = \frac{2^{-\floor{\log_2 \left( \frac{\phi(D)}{\phi(R)}\right)}}}{2^{-\floor{\log_2 \left( \frac{\phi(D)}{\phi(r)}\right)}}}
\end{equation*}
or equivalently,
\begin{equation} \label{Log2VolumesWx}
    \log_2 \left( \frac{\Vhat(z_0, R)}{\Vhat(z_0, r)} \right) = \floor{\log_2 \left( \frac{\phi(D)}{\phi(r)} \right)} - \floor{\log_2 \left( \frac{\phi(D)}{\phi(R)} \right)}
\end{equation}

Let $c_1$, $c_2$, $\beta_1$, and $\beta_2$ be the constants from \eqref{regularGrowth}.

Proof of (b): Assume $(M, d, \mu)$ satisfies $\VD$. Let $C$ be the constant from \eqref{VdFormula}. Fix $z_0 \in \Mhat$ and $r>0$. Let $x=\pi(z_0)$. If $r \geq D_x$, then by \eqref{VhatFormula},
\begin{equation} \label{MhatVdProp1}
    \frac{\Vhat(z_0, 2r)}{\Vhat(z_0, r)} = \frac{V(x, 2r)}{V(x, r)} \leq C. 
\end{equation}
If $2r \leq D_x$, then
\begin{align*}
    \log_2 \left( \frac{\Vhat(z_0, 2r)}{\Vhat(z_0, r)} \right)& = \floor{\log_2 \left( \frac{\phi(D)}{\phi(r)} \right)} - \floor{\log_2 \left( \frac{\phi(D)}{\phi(2r)} \right)} &\mbox{(by \eqref{Log2VolumesWx})}\\
    &\leq 1 + \log_2 \left( \frac{\phi(D)}{\phi(r)} \right) - \log_2 \left( \frac{\phi(D)}{\phi(2r)} \right) &\mbox{(since $\floor{b-a} \leq 1+b-a$ for all $a \leq b$)}\\
    &= 1 + \log_2 \left( \frac{\phi(2r)}{\phi(r)}\right)\\
    &\leq 1 + \log_2(c_2 2^{\beta_2}) =: c_3,
\end{align*}
or equivalently
\begin{equation} \label{MhatVdProp2}
    \frac{\Vhat(z_0, 2r)}{\Vhat(z_0, r)} \leq 2^{c_3}.
\end{equation}
The only other case is $D_x/2 < r < D_x < 2r < 2D_x$. In this case, by \eqref{MhatVdProp1} and \eqref{MhatVdProp2},
\begin{equation} \label{MhatVdProp3}
    \frac{\Vhat(z_0, 2r)}{\Vhat(z_0, r)} \leq \frac{\Vhat(z_0, 2D_x)}{\Vhat(z_0, D_x/2)} = \frac{\Vhat(z_0, 2D_x)}{\Vhat(z_0, D_x)} \cdot \frac{\Vhat(z_0, D_x)}{\Vhat(z_0, D_x/2)} \leq C \cdot 2^{c_3}.
\end{equation}
By \eqref{MhatVdProp1}, \eqref{MhatVdProp2}, and \eqref{MhatVdProp3}, $\VD$ holds for $(\Mhat, \dhat, \muhat)$.

Proof of (c): Assume $(M, d, \mu)$ satisfies $\QRVD$. Let $\ell>1$ and $c>1$ be the constants from \eqref{QrvdFormula}. Fix $z_0 \in \Mhat$ and $r>0$. Let $x=\pi(z_0)$. If $r \geq D_x$, then by \eqref{VhatFormula},
\begin{equation} \label{MhatRvdProp1}
    \frac{\Vhat(z_0, \ell r)}{\Vhat(z_0, r)} \geq c.
\end{equation}
Let $\ell_1 := (2/c_1)^{1/\beta_1} \vee 1$, for reasons that will soon be clear. Suppose $\ell_1 r \geq D_x$. Then
\begin{align*}
    \log_2 \left( \frac{\phi(D)}{\phi(r)} \right) - \log_2 \left( \frac{\phi(D)}{\phi(\ell_1 r)} \right) &= \log_2 \left( \frac{\phi(\ell_1 r)}{\phi(r)} \right)\\
    &\geq \log_2 \left( c_1 \ell_1^{\beta_1} \right) &\mbox{(by \eqref{regularGrowth})}\\
    &\geq \log_2 \left(c_1 \left(\frac{2}{c_1}\right)\right) &\mbox{(by our choice of $\ell_1$)}\\
    &=1.
\end{align*}
Note that if $b-a \geq 1$, then $\floor{b}-\floor{a} \geq 1$. Applying this fact to $a = \log_2 \left( \frac{\phi(D)}{\phi(\ell_1 r)} \right)$ and $b=\log_2 \left( \frac{\phi(D)}{\phi(r)} \right)$, we see that
\begin{equation} \label{MhatRvdProp2}
    \floor{\log_2 \left( \frac{\phi(D)}{\phi(r)} \right)} - \floor{\log_2 \left( \frac{\phi(D)}{\phi(2r)} \right)} \geq 1.
\end{equation}
Thus,
\begin{align*}
    \log_2 \left( \frac{\Vhat(z_0, \ell_1 r)}{\Vhat(z_0, r)} \right) &= \floor{\log_2 \left( \frac{\phi(D)}{\phi(r)} \right)} - \floor{\log_2 \left( \frac{\phi(D)}{\phi(\ell_1 r)} \right)} &\mbox{(by \eqref{Log2VolumesWx})}\\
    &\geq 1 &\mbox{(by \eqref{MhatRvdProp2})},
\end{align*}
or equivalently,
\begin{equation} \label{MhatRvdProp3}
    \frac{\Vhat(z_0, \ell_1 r)}{\Vhat(z_0, r)} \geq 2.
\end{equation}
Finally, if $r<D_x <\ell_1 r$, then
\begin{equation} \label{MhatRvdProp4}
    \frac{\Vhat(z_0, \ell \cdot \ell_1 r)}{\Vhat(z_0, r)} \geq \frac{\Vhat(z_0, \ell \cdot \ell_0 r)}{\Vhat(z_0, \ell_1 r)} \geq c.
\end{equation}
By \eqref{MhatRvdProp1}, \eqref{MhatRvdProp3}, \eqref{MhatRvdProp4}, and the fact that $\ell \cdot \ell_1$ is greater than or equal to both $\ell$ and $\ell_1$, we have $\Vhat(z_0, \ell \cdot \ell_1 r) \geq \max\{2, c\} \Vhat(z_0, r)$ for all $r>0$. Therefore, $(\Mhat, \dhat, \muhat)$ satisfies $\RVD$.
\end{proof}

\subsection{The regular Dirichlet form \texorpdfstring{$(\Ehat, \Fhat)$}{TEXT} on the auxiliary space} \label{auxiliaryEFconstruction}

Assume $(M, d, \mu, \mathcal{E}, \mathcal{F})$ satisfies Assumption \ref{BasicAssumptions}, $\phi$ is of regular growth, $(\mathcal{E}, \mathcal{F})$ admits a jump kernel, and $M=M_A \cup M_C$ (where $M_A$ and $M_C$ are as defined in \eqref{MaMcDef}).

In this subsection, we construct the jump process on $\Mhat$. Recall that the jump process $\{X_t\}$ on $M$ was the Hunt process associated with a pure-jump regular Dirichlet form $(\mathcal{E}, \mathcal{F})$ on $L^2(M, \mu)$. Similarly, the jump process $\{\Xhat_t\}$ on the auxiliary space will be the Hunt process associated with a pure-jump regular Dirichlet form $(\Ehat, \Fhat)$ on $L^2(\Mhat, \muhat)$.

Lemma \ref{constructRegularDirichletForms} gives us a recipe to construct a regular Dirichlet form on $L^2(\Mhat, \muhat)$. It is up to us to determine a symmetric non-negative measurable function $\Jhat(z, z')$ on $\Mhat\times\Mhat\setminus\diag_{\Mhat}$ (which will play the role of $j$ in Lemma \ref{constructRegularDirichletForms}) and a set $\Dhat$ (which will play the role of $\mathcal{D}$). When we feed a suitable $\Jhat$ and $\Dhat$ into Lemma \ref{constructRegularDirichletForms}, the lemma will give us a regular Dirichlet form $(\Ehat, \Fhat)$.

Let us briefly discuss what motivates our choice of $\Jhat(z, z')$. We would like the process $\{\Xhat_t\}$ to make jumps from $W_x$ to $W_y$ with the same rate as $\{X_t\}$ makes jumps from $x$ to $y$. More precisely, if $\pi : \Mhat \to M$ is the projection that maps every point in $W_x$ to $x$, we would like the process $\{\pi(\Xhat_t)\}$ to have the same law as $\{X_t\}$. In other words, $\{X_t\}$ and $\{\Xhat_t\}$ can be coupled so that $\{\Xhat_t\}$ jumps around on $W_x$ while $\{X_t\}$ is being held at $x$, and jumps to $W_y$ whenever $\{X_t\}$ jumps to $y$. This will make it possible to show that
\begin{equation*} \label{WantEphiBothWays}
    \boxed{\mbox{$\QE$ for $(M, d, \mu, \mathcal{E}, \mathcal{F})$}} \Longleftrightarrow \boxed{\mbox{$\E$ for $(\Mhat, \dhat, \muhat, \Ehat, \Fhat)$}} \qquad\mbox{(under suitable assumptions)}.
\end{equation*}
The only remaining question is how to set the rates for the jumps $\{\Xhat_t\}$ takes \textit{within} $W_x$. We would like to choose these transition rates so that
\begin{equation*} \label{WantJphiBothWays}
    \boxed{\mbox{$\J$ for $(M, d, \mu, \mathcal{E}, \mathcal{F})$}} \Longleftrightarrow \boxed{\mbox{$\J$ for $(\Mhat, \dhat, \muhat, \Ehat, \Fhat)$}}.
\end{equation*}

Let us reiterate that the construction of $(\Mhat, \dhat, \muhat, \Ehat, \Fhat)$ relies on the assumptions that $(M, d, \mu, \mathcal{E}, \mathcal{F})$ satisfies Assumptions \ref{BasicAssumptions}, $(\mathcal{E}, \mathcal{F})$ admits a jump kernel, and $M=M_A \cup M_C$. In the proofs of our main results, we only refer to the auxiliary space when all of these assumptions hold (see Section \ref{proofMainResults}). The construction of $(\Mhat, \dhat, \muhat, \Ehat, \Fhat)$ also depends on $\phi$ (and our permanent assumption that $\phi$ is of regular growth).

Let us start by defining the kernel $\Jhat : \Mhat\times\Mhat\setminus\diag_{\Mhat} \to [0, \infty)$, which will become the jump kernel of $\{\Xhat_t\}$. If $J(x, y)$ be the jump kernel of $\{X_t\}$, let
\begin{equation}\label{JhatFormula}
    \Jhat(z, z') = \left\{
    \begin{matrix}
        J(\pi(z), \pi(z')) &:& \mbox{if $\pi(z)\neq \pi(z')$}\\
       \\
       \frac{1}{\Vhat(z, \dhat(z, z')) \phi(\dhat(z, z'))} &:& \mbox{if $\pi(z) = \pi(z')$}.
    \end{matrix}
    \right.
\end{equation}
It will be useful to have an explicit formula for $\Jhat(z, z')$. Suppose $z=(x, w)$ and $z'=(x, w')$ for some $x \in M_A$, where $w$ and $w'$ are distinct elements of $W$. Let $m$ be the first index of disagreement between the words $w$ and $w'$. By \eqref{DefinitionOfDhat} and \eqref{RhoFormula}, $\dhat(z, z') = \phi^{-1} \left( 2^{-m} \phi(D_x) \right)$. Thus, $\phi(\dhat(z, z')) = 2^{-m} \phi(D_x)$. By \eqref{VhatFormula} and \eqref{VtildeIntermsofM}, $\Vhat(z, \dhat(z, z')) = \mu(x) \cdot \nu(E^w_m) = \mu(x) \cdot 2^{-m}$. Therefore, another formula for \eqref{JhatFormula} is

\begin{equation}\label{JhatFormulaIntermsofM}
    \Jhat(z, z') = \left\{
    \begin{matrix}
        J(x, y) &:& \mbox{if $x = \pi(z)\neq \pi(z') = y$}\\
       \\
       \frac{4^m}{\mu(x) \phi(D_x)} &:& \mbox{if $x \in M_A$, $z=(x, w)$, $z'=(x, w')$, and $w' \in A^w_m$}.
    \end{matrix}
    \right.
\end{equation}
Let us abuse notation by also using $\Jhat$ to refer to the measure $$\Jhat(dz, dz') = \Jhat(z, z') \muhat(dz) \muhat(dz')$$ on $\Mhat\times\Mhat\setminus\diag_{\Mhat}$.

Now let us follow the recipe of Lemma \ref{constructRegularDirichletForms} to construct a regular Dirichlet form with jump kernel $\Jhat$.
Let $\Ehat : L^2(\Mhat, \muhat) \to [0, \infty]$ be the function
\begin{equation*}
    \Ehat(f) := \int_{\Mhat\times\Mhat\setminus\diag_{\Mhat}} (f(z)-f(z'))^2 \Jhat(dz, dz')
\end{equation*}
Let
\begin{equation*}
    \Fhatmax := \left\{ f \in L^2(\Mhat, \muhat) : \Ehat(f) < \infty \right\}.
\end{equation*}
For all $f, g \in \Fhatmax$, let
\begin{equation}\label{EhatFormula}
    \Ehat(f, g) := \int_{\Mhat\times\Mhat\setminus\diag_{\Mhat}} (f(z)-f(z'))(g(z)-g(z')) \Jhat(dz, dz')
\end{equation}
and
\begin{equation}\label{Ehat1normFormula}
    \Ehat_1(f, g) := \Ehat(f, g) + \int_{\Mhat} fg \, d\muhat.
\end{equation}

In order to use Lemma \ref{constructRegularDirichletForms}, we must find a Markovian space $\Dhat \subseteq \Fhatmax \cap C_c(\Mhat)$ that is dense in $C_c(\Mhat)$ under the uniform norm.

For any function $h : W \to \mathbb{R}$, and any $x \in M_A$, let $H_{x, h} : \Mhat\to\mathbb{R}$ be the function
\begin{equation} \label{HxhFormula}
    H_{x, h}(z) = \left\{ \begin{matrix}
        h(w) &:& \mbox{if $z = (x, w)$ for some $w \in W$}\\
        \\
        0 &:& \mbox{if $z \notin W_x$}.
    \end{matrix}
    \right.
\end{equation}
Recall the projection $\pi : \Mhat \to M$, which maps every point on $W_x$ to $x$, for all $x \in M$. Recall that for any function $g : M\to\mathbb{R}$, $g \circ \pi$ is the function from $\Mhat$ to $\mathbb{R}$ that is constant on each $W_x$, mapping every point on $W_x$ to $g(x)$.

\begin{definition}\label{DhatDef}
Let $\Dhat$ be the set of functions $f$ of the form
\begin{equation*}
    f = (g \circ \pi) + \sum_{j=1}^N H_{x_j, h_j}
\end{equation*}
where $g \in \mathcal{F} \cap C_c(M)$, each $x_j$ belongs to $M_A$, and each $h_j$ is a locally constant function on $W$. Let $\Fhat$ be the closure of $\Dhat$ under the $\Ehat_1$-norm.
\end{definition}

\LongVersionShortVersion
{ 
The following lemma is proved in Appendix \ref{DhatAppendix}. It says that $\Dhat$ satisfies the assumptions of Lemma \ref{constructRegularDirichletForms}.
}
{ 
The following lemma is proved in \cite[Appendix B.5]{longversion}. It says that $\Dhat$ satisfies the assumptions of Lemma \ref{constructRegularDirichletForms}.
} 

\begin{lemma} \label{DhatSatisfiesLemma}
If $(M, d, \mu, \mathcal{E}, \mathcal{F})$ satisfies Assumption \ref{BasicAssumptions}, $\phi$ is of regular growth, $(\mathcal{E}, \mathcal{F})$ admits a jump kernel, and $M=M_A \cup M_C$,
then the set $\Dhat$ (constructed in Definition \ref{DhatDef}) satisfies the following properties:

(a) $\Dhat$ is a subspace of $\Fhatmax \cap C_c(\Mhat)$.

(b) $\Dhat$ is Markovian (in the sense of \eqref{Markovian}).

(c) $\Dhat$ is dense in $\Fhatmax \cap C_c(\Mhat)$, under the uniform norm.
\end{lemma}

By Lemma \ref{DhatSatisfiesLemma} and Lemma \ref{constructRegularDirichletForms}, $(\Ehat, \Fhat)$ is a regular Dirichlet form. By \eqref{EhatFormula}, $(\Ehat, \Fhat)$ is pure-jump (and thus $(\Mhat, \dhat, \muhat, \Ehat, \Fhat)$ satisfies Assumption \ref{BasicAssumptions}(b)). From now on, let us refer to $(M, d, \mu, \mathcal{E}, \mathcal{F})$ as ``the original space" and $(\Mhat, \dhat, \muhat, \Ehat, \Fhat)$ as ``the auxiliary space." The following proposition summarizes what we have shown about the auxiliary space so far.

\begin{prop} \label{CkwAppliesToAuxiliaryProp}
Let $\phi$ be a function of regular growth.
Suppose $(M, d, \mu, \mathcal{E}, \mathcal{F})$ satisfies Assumption \ref{BasicAssumptions}, $\VD$, and $\QRVD$. Suppose also that $(\mathcal{E}, \mathcal{F})$ admits a jump kernel. Then $(\Mhat, \dhat, \muhat, \Ehat, \Fhat)$ satisfies Assumption \ref{BasicAssumptions}, $\VD$, and $\RVD$.
\end{prop}

In other words, if the original space satisfies the assumptions of Theorems \ref{ourMainResult1} and \ref{ourMainResult3} (our main results), and $(\mathcal{E}, \mathcal{F})$ admits a jump kernel, then the auxiliary space satisfies the assumptions of \cite{ckw} and \cite{ckw2}.

\begin{proof}[Proof of Proposition \ref{CkwAppliesToAuxiliaryProp}]
By Proposition \ref{Atom=Isolated}, Assumption \ref{BasicAssumptions}(a) and $\QRVD$ are enough to guarantee $M=M_A \cup M_C$, so the construction of $(M, d, \mu)$ is well-defined. By Proposition \ref{volumeGrowthProp}, the auxiliary space satisfies Assumption \ref{BasicAssumptions}(a), $\VD$, and $\RVD$. The extra assumption that $(\mathcal{E}, \mathcal{F})$ admits a jump kernel means that the construction of $(\Ehat, \Fhat)$ is well-defined, and the auxiliary space satisfies Assumption \ref{BasicAssumptions}(b) by \eqref{EhatFormula}.
\end{proof}

Let $\Xhat = \left\{\Xhat_t : t\geq 0 ; \mathbb{P}^z : z \in \Mhat\setminus\Nhat \right\}$ be the $\muhat$-symmetric Hunt process associated with $(\Ehat, \Fhat)$, where $\Nhat$ is a properly exceptional set. For any open $U \subseteq \Mhat$, let $\widehat{\tau}_U$ denote the exit time of $\Xhat$ from $U$.

For all $x \in M$, $z_0 \in W_x$, and $A \subseteq M \setminus \{x\}$,
\begin{equation*}
    \int_{\pi^{-1}(A)} \Jhat(z_0, z) \muhat(dz) = \int_{\pi^{-1}(A)} J(x, \pi(z)) \muhat(dz) = \int_A J(x,y) \mu(dy).
\end{equation*}
In other words, for all $z_0 \in W_x$, jumps in $\Xhat$ from $z_0$ to $\pi^{-1}(A)$ occur with the same rate as jumps in $X$ from $x$ to $A$. Therefore, $\{\pi(\Xhat_t)\}$ and $\{X_t\}$ have the same jump kernel, so
\begin{equation} \label{couple}
    \{ \pi(\Xhat_t)\}_{t\geq0} \overset{d}{=} \left\{X_t\right\}_{t\geq0}.
\end{equation}

\begin{remark}\label{NhatRemark}
By \eqref{couple}, we can take $\Nhat$ to be $\pi^{-1}(\mathcal{N})$. Let us do so. Since $\mu(\mathcal{N}) = 0$, $\mathcal{N}$ must be a subset of $M_C$, so $\Nhat = \pi^{-1}(\mathcal{N}) = \mathcal{N}$. Recall that we defined $M_0$ to be $M \setminus \mathcal{N}$. Let us similarly define $\Mohat$ to be $\Mhat \setminus \Nhat = \Mhat \setminus \mathcal{N}$.
\end{remark}

We emphasize that our construction of $(\Mhat, \dhat, \muhat, \Ehat, \Fhat)$ only makes sense if $(M, d, \mu, \mathcal{E}, \mathcal{F})$ satisfies Assumption \ref{BasicAssumptions}, $\phi$ is of regular growth, $(\mathcal{E}, \mathcal{F})$ admits a jump kernel, and $M=M_A \cup M_C$ (where $M_A$ and $M_C$ are as defined in \eqref{MaMcDef}), and that all of these conditions must be verified whenever we refer to the auxiliary space during the proof of our main results.

\section{Proof of main results}\label{proofMainResults}

In this section, we present diagrams (Figures \ref{HkStableDiagram}-\ref{PhiStableDiagram}) showing how to prove our main results, assuming the following proposition. This proposition will be proved over the course of Sections \ref{EscapeSection}-\ref{CsPoincSection}.

\begin{prop} \label{MasterProp}
If $(M, d, \mu, \mathcal{E}, \mathcal{F})$ satisfies Assumption \ref{BasicAssumptions}, $\phi$ is of regular growth, $(\mathcal{E}, \mathcal{F})$ admits a jump kernel, and $M=M_A \cup M_C$, then we have the following implications between the original space and the auxiliary space:
\begin{align}
    &\boxed{\mbox{$\Jleq$ for $(M, d, \mu, \mathcal{E}, \mathcal{F})$}} \Longleftrightarrow \boxed{\mbox{$\Jleq$ for $(\Mhat, \dhat, \muhat, \Ehat, \Fhat)$}}. \label{JleqBothways}\\
    &\boxed{\mbox{$\Jgeq$ for $(M, d, \mu, \mathcal{E}, \mathcal{F})$}} \Longleftrightarrow \boxed{\mbox{$\Jgeq$ for $(\Mhat, \dhat, \muhat, \Ehat, \Fhat)$}}.\label{JgeqBothways}\\
    &\boxed{\mbox{$\UJS$ for $(M, d, \mu, \mathcal{E}, \mathcal{F})$}} \Longleftrightarrow \boxed{\mbox{$\UJS$ for $(\Mhat, \dhat, \muhat, \Ehat, \Fhat)$}}.\label{UjsBothways}\\
    &\boxed{\mbox{$\QEleq$ for $(M, d, \mu, \mathcal{E}, \mathcal{F})$}} \Longleftrightarrow \boxed{\mbox{$\Eleq$ for $(\Mhat, \dhat, \muhat, \Ehat, \Fhat)$}}.\label{EleqBothways}\\
    &\boxed{\mbox{$\QEgeq$ for $(M, d, \mu, \mathcal{E}, \mathcal{F})$}} \Longleftarrow \boxed{\mbox{$\Egeq$ for $(\Mhat, \dhat, \muhat, \Ehat, \Fhat)$}}.\label{EgeqBackwards}\\
    &\boxed{\mbox{$\UHK$ for $(M, d, \mu, \mathcal{E}, \mathcal{F})$}} \Longleftarrow \boxed{\mbox{$\UHK$ for $(\Mhat, \dhat, \muhat, \Ehat, \Fhat)$}}.\label{UhkBackwards}\\
    &\boxed{\mbox{$\LHK$ for $(M, d, \mu, \mathcal{E}, \mathcal{F})$}} \Longleftarrow \boxed{\mbox{$\LHK$ for $(\Mhat, \dhat, \muhat, \Ehat, \Fhat)$}}.\label{LhkBackwards}\\
    &\boxed{\mbox{$\PI$ for $(M, d, \mu, \mathcal{E}, \mathcal{F})$}} \Longrightarrow \boxed{\mbox{$\PI$ for $(\Mhat, \dhat, \muhat, \Ehat, \Fhat)$}}.\label{PoincareForward}\\
    &\boxed{\mbox{$\PHR$ for $(M, d, \mu, \mathcal{E}, \mathcal{F})$}} \Longleftarrow \boxed{\mbox{$\PHR$ for $(\Mhat, \dhat, \muhat, \Ehat, \Fhat)$}}.\label{PhrBackward}
\end{align}
Under the additional assumption that $\VD$ holds on the original space, we also have
\begin{align}
    &\boxed{\mbox{$\NDL$ for $(M, d, \mu, \mathcal{E}, \mathcal{F})$}} \Longleftarrow \boxed{\mbox{$\NDL$ for $(\Mhat, \dhat, \muhat, \Ehat, \Fhat)$}}. \label{NdlBackwards}\\
    &\boxed{\mbox{$\UHKD$ for $(M, d, \mu, \mathcal{E}, \mathcal{F})$}} \Longrightarrow \boxed{\mbox{$\UHKD$ for $(\Mhat, \dhat, \muhat, \Ehat, \Fhat)$}}. \label{UhkdForwards}
\end{align}
Under the additional assumption that $\VD$ and $\Jleq$ hold on the original space, we also have
\begin{align}
    &\boxed{\mbox{$\QEgeq$ for $(M, d, \mu, \mathcal{E}, \mathcal{F})$}} \Longrightarrow \boxed{\mbox{$\Egeq$ for $(\Mhat, \dhat, \muhat, \Ehat, \Fhat)$}}.\label{EgeqForwards}\\
    & \boxed{\mbox{$\CSJ$ for $(M, d, \mu, \mathcal{E}, \mathcal{F})$}} \Longrightarrow \boxed{\mbox{$\CSJ$ for $(\Mhat, \dhat, \muhat, \Ehat, \Fhat)$}}. \label{CsjForwards}
\end{align}
\end{prop}

\LongVersionShortVersion
{ 
Implications \eqref{JleqBothways}-\eqref{UjsBothways} are straightforward consequences of how we defined $\Jhat$ in \eqref{JhatFormula}. They are proven in Appendix \ref{JumpKernelAppendix}.
Implication \eqref{PhrBackward} is straightforward using the definition of caloric. We prove it in \ref{PhrAppendix}.
In Section \ref{EscapeSection}, we prove \eqref{EleqBothways}, \eqref{EgeqBackwards}, and \eqref{EgeqForwards} (the implications relating to escape times). These are also relatively straightforward, but they do require us to first estimate the escape times $\mathbb{E}^{z_0} \widehat{\tau}_{\Bhat(z_0, r)}$ in the auxiliary space, when $\Bhat(z_0, r) \subseteq W_x$ for some $x \in M_A$.
In Section \ref{HeatKernelSection1}, we derive an explicit formula for the heat kernel of $\{\Xhat_t\}$ in terms of the heat kernel of $\{X_t\}$, and then in Section \ref{HeatKernelSection2} we prove implications \eqref{UhkBackwards}-\eqref{LhkBackwards} and \eqref{NdlBackwards}-\eqref{UhkdForwards} by directly comparing the two heat kernels. In Section \ref{CsPoincSection}, we first make some observations about the domain $\Fhat$ of the regular Dirichlet form on the auxiliary space, and then use these to prove \eqref{CsjForwards} and \eqref{PoincareForward}.
}
{ 
Implications \eqref{JleqBothways}-\eqref{UjsBothways} are straightforward consequences of how we defined $\Jhat$ in \eqref{JhatFormula}.
They are proven in \cite[Appendix B.6]{longversion}.
Implication \eqref{PhrBackward} is straightforward using the definition of caloric. We prove it in \cite[Appendix B.7]{longversion}.
In Section \ref{EscapeSection}, we prove \eqref{EleqBothways}, \eqref{EgeqBackwards}, and \eqref{EgeqForwards} (the implications relating to escape times). These are also relatively straightforward, but they do require us to first estimate the escape times $\mathbb{E}^{z_0} \widehat{\tau}_{\Bhat(z_0, r)}$ in the auxiliary space, when $\Bhat(z_0, r) \subseteq W_x$ for some $x \in M_A$.
In Section \ref{HeatKernelSection1}, we derive an explicit formula for the heat kernel of $\{\Xhat_t\}$ in terms of the heat kernel of $\{X_t\}$, and then in Section \ref{HeatKernelSection2} we prove implications \eqref{UhkBackwards}-\eqref{LhkBackwards} and \eqref{NdlBackwards}-\eqref{UhkdForwards} by directly comparing the two heat kernels. In Section \ref{CsPoincSection}, we first make some observations about the domain $\Fhat$ of the regular Dirichlet form on the auxiliary space, and then use these to prove \eqref{CsjForwards} and \eqref{PoincareForward}.
} 

Suppose the assumptions of Proposition \ref{MasterProp} hold.
Let us ``clean up" some of the statements among \eqref{JleqBothways}-\eqref{CsjForwards}.
By combining \eqref{JleqBothways} and \eqref{JgeqBothways},
\begin{equation} \label{JphiBothways}
    \boxed{\mbox{$\J$ for $(M, d, \mu, \mathcal{E}, \mathcal{F})$}} \Longleftrightarrow \boxed{\mbox{$\J$ for $(\Mhat, \dhat, \muhat, \Ehat, \Fhat)$}}.
\end{equation}
Assuming $\VD$ and $\Jleq$ hold on the original space, by combining \eqref{EleqBothways}, \eqref{EgeqBackwards}, and \eqref{EgeqForwards}, we also have
\begin{equation}\label{EscapeBothways}
    \boxed{\mbox{$\QE$ for $(M, d, \mu, \mathcal{E}, \mathcal{F})$}} \Longleftrightarrow \boxed{\mbox{$\E$ for $(\Mhat, \dhat, \muhat, \Ehat, \Fhat)$}}.
\end{equation}
Finally, by combining \eqref{LhkBackwards} and \eqref{UhkBackwards},
\begin{equation}\label{HkBackwards}
    \boxed{\mbox{$\HK$ for $(M, d, \mu, \mathcal{E}, \mathcal{F})$}} \Longleftarrow \boxed{\mbox{$\HK$ for $(\Mhat, \dhat, \muhat, \Ehat, \Fhat)$}}.
\end{equation}

It is worth noting that the converse of \eqref{CsjForwards} holds as well:

\begin{prop} \label{CsjBackwardsProp}
If $(M, d, \mu, \mathcal{E}, \mathcal{F})$ satisfies Assumption \ref{BasicAssumptions}, $\phi$ is of regular growth, $(\mathcal{E}, \mathcal{F})$ admits a jump kernel, $M=M_A \cup M_C$, and $\VD$ and $\Jleq$ hold for $(M, d, \mu, \mathcal{E}, \mathcal{F})$, then
\begin{equation*}
    \boxed{\mbox{$\CSJ$ for $(M, d, \mu, \mathcal{E}, \mathcal{F})$}} \Longleftarrow \boxed{\mbox{$\CSJ$ for $(\Mhat, \dhat, \muhat, \Ehat, \Fhat)$}}.
\end{equation*}
\end{prop}

\LongVersionShortVersion
{ 
For a proof of Proposition \ref{CsjBackwardsProp}, see Appendix \ref{CsjBackwardsAppendix}.
We do not use Proposition \ref{CsjBackwardsProp} anywhere in the proof of our main results.
}
{ 
For a proof of Proposition \ref{CsjBackwardsProp}, see \cite[Appendix B.8]{longversion}. We do not use Proposition \ref{CsjBackwardsProp} anywhere in the proof of our main results.
} 

\subsection*{Stability of \texorpdfstring{$\HK$}{TEXT}} \label{StabHkSection}

\begin{figure}
\begin{framed}
\captionof{figure}{Proof of Theorem \ref{ourMainResult1} (stable characterization of $\HK$)}
\label{HkStableDiagram}
    \centering
    \begin{tikzpicture}
        
        \node (M) at (0, 2) {\underline{$\mathbf(M, d, \mu, \mathcal{E}, \mathcal{F})$}};
        \node (M') at (10, 2) {{\underline{$(\Mhat, \dhat, \muhat, \Ehat, \Fhat)$}}};
        
        \node (hk) at (0, 0) {$\HK$};
        \node (hk') at (10, 0) {$\HK$};
        
        \node (scsj) at (0, -2) {$\J+\SCSJ$};
        \node (scsj') at (10, -2) {$\J+\SCSJ$};
        
        \node (csj) at (0, -4) {$\J+\CSJ$};
        \node (csj') at (10, -4) {$\J+\CSJ$};
        
        \node (qe) at (0, -6) {$\J + \QE$};
        \node (e') at (10, -6) {$\J + \E$};
        
        \path[->] (hk) edge node[right] {Prop \ref{HkImpliesJScsj}} (scsj);
        
        \path[->] (scsj) edge node[right] {by definition} (csj);
        
        \path[->] (hk') edge node[below] {\eqref{HkBackwards}} (hk);
        
        \path[->] (csj) edge node[below] {\eqref{JphiBothways} and \eqref{CsjForwards}} (csj');
        
        \path[<->] (qe) edge node[below] {\eqref{JphiBothways} and \eqref{EscapeBothways}} (e');
        
        \path[<->] (hk') edge node[left] {\cite[Theorem 1.13]{ckw}} (scsj');
        \path[<->] (scsj') edge node[left] {\cite[Theorem 1.13]{ckw}} (csj');
        \path[<->] (csj') edge node[left] {\cite[Theorem 1.13]{ckw}} (e');
    \end{tikzpicture}
\end{framed}
\end{figure}

\begin{figure}
\begin{framed}
\captionof{figure}{Proof that $\UHKD + \QE + \Jleq \Longrightarrow \UHK$}
\label{UhkdImpliesUhkDiagram}
    \centering
    \begin{tikzpicture}
        
        \node (M) at (0, 2) {\underline{$\mathbf(M, d, \mu, \mathcal{E}, \mathcal{F})$}};
        \node (M') at (10, 2) {{\underline{$(\Mhat, \dhat, \muhat, \Ehat, \Fhat)$}}};
        
        \node (uhkd*) at (0, 0) {$\UHKD + \QE + \Jleq$};
        
        \node (uhkd*hat) at (10, 0) {$\UHKD + \E + \Jleq$};
        
        \node (uhk) at (0, -2) {$\UHK$};
        
        \node (uhkhat) at (10, -2) {$\UHK$};
        
        \path[->] (uhkd*) edge node[above] {\eqref{UhkdForwards}, \eqref{EscapeBothways}, and \eqref{JleqBothways}} (uhkd*hat);
        
        \path[->] (uhkd*hat) edge node[left] {\cite[Prop 5.3]{ckw}} (uhkhat);
        
        \path[->] (uhkhat) edge node[below] {\eqref{HkBackwards}} (uhk);
    \end{tikzpicture}
\end{framed}
\end{figure}

\begin{figure}
\begin{framed}
\captionof{figure}{Proof that $\PHI \Longrightarrow \PI+\Jleq+\SCSJ+\UJS$}
\label{PhiImpliesDiagram}
    \centering
    \begin{tikzpicture}
        
        \node (phi) at (0, 1) {$\boxed{\PHI}$};
        
        \node (uhkd) at (-6, -2) {$\UHKD$};
        
        \node (ndl) at (0, -2) {$\NDL$};
        
        \node (ujs) at (4, -2) {$\boxed{\UJS}$};
        
        \node (qe) at (-2.6, -4) {$\QE$};
        
        \node (pi) at (-4, -4) {$\boxed{\PI}$};
        
        \node (jleqParent) at (4, -4) {$\NDL + \UJS$};
        
        \node (jleq) at (4, -6) {$\boxed{\Jleq}$};
        
        \node (3things) at (-4, -10) {$\UHKD + \QE + \Jleq$};
        
        \node (conservative) at (4, -10) {$(\mathcal{E}, \mathcal{F})$ is conservative};
        
        \node (uhk) at (-4, -12) {$\UHK$};
        
        \node (2things) at (0, -14) {$\UHK$ + $(\mathcal{E}, \mathcal{F})$ is conservative};
        
        \node (scsj) at (0, -16) {$\boxed{\SCSJ}$};
        
        \path[->] (phi) edge node[left] {\cite[Prop 3.1]{ckw2}} (uhkd);
        
        \path[->] (phi) edge node[left] {\cite[Prop 3.2]{ckw2}} (ndl);
        
        \path[->] (phi) edge node[right] {\cite[Prop 3.3]{ckw2}} (ujs);
        
        \path[->] (ndl) edge node[right] {Prop \ref{NdlImpliesQe}} (qe);
        
        \path[->] (ndl) edge node[left] {\cite[Prop 3.5.i]{ckw2}} (pi);
        
        \path[->] (ndl) edge node[left] {} (jleqParent);
        
        \path[->] (ujs) edge node[right] {} (jleqParent);
        
        \path[->] (jleqParent) edge node[right] {\cite[Cor 3.4]{ckw2}} (jleq);
        
        \path[->] (uhkd) edge node[left] {} (3things);
        
        \path[->] (qe) edge node[right] {} (3things);
        
        \path[->] (jleq) edge node[left] {} (3things);
        
        \path[->] (ndl) edge node[left] {\cite[Prop 2.4]{ckw2}} (conservative);
        
        \path[->] (3things) edge node[right] {Figure \ref{UhkdImpliesUhkDiagram}} (uhk);
        
        \path[->] (uhk) edge node[right] {} (2things);
        
        \path[->] (conservative) edge node[right] {} (2things);
        
        \path[->] (2things) edge node[right] {\cite[Prop 5.3]{ckw}} (scsj);
        
    \end{tikzpicture}
\end{framed}
\end{figure}

\begin{figure}[ht!]
\begin{framed}
\captionof{figure}{Proof of Theorem \ref{ourMainResult3} (stable characterization of $\PHI$)}
\label{PhiStableDiagram}
    \centering
    \begin{tikzpicture}
        \node (M) at (0, 2) {\underline{$\mathbf(M, d, \mu, \mathcal{E}, \mathcal{F})$}};
        \node (M') at (9, 2) {\underline{$\mathbf(\Mhat, \dhat, \muhat, \Ehat, \Fhat)$}};
        
        \node (top) at (0, 0) {$\UHK + \NDL + \UJS$};
        \node (top') at (9, 0) {$\UHK + \NDL + \UJS$};
        
        \node (phi+) at (0, -2) {$\PHIplus$};
        
        \node (ndlujs) at (-2, -4) {$\NDL + \UJS$};
        \node (phi) at (2, -4) {$\PHI$};
        
        \node (ScsjThing) at (0, -6) {$\PI + \Jleq + \SCSJ + \UJS$};
        \node (bottom) at (0, -8) {$\PI + \Jleq + \CSJ + \UJS$};
        \node (bottom') at (9, -8) {$\PI + \Jleq + \CSJ + \UJS$};
        
        \path[->] (top) edge node[right] {\cite[Theorem 4.3]{ckw2}} (phi+);
        \path[<->] (phi+) edge node[left] {\cite[Prop 4.4]{ckw2}} (ndlujs);
        \path[->] (phi+) edge node[right] {by definition} (phi);
        \path[->] (phi) edge node[right] {Figure \ref{PhiImpliesDiagram}} (ScsjThing);
        \path[->] (ScsjThing) edge node[right] {by definition} (bottom);
        
        \path[->] (top') edge node[below] {\eqref{UhkBackwards}, \eqref{NdlBackwards}, and \eqref{UjsBothways}} (top);
        \path[->] (bottom) edge node[below] {\eqref{PoincareForward}, \eqref{JleqBothways}, \eqref{CsjForwards}, and \eqref{UjsBothways}} (bottom');
        
        \path[<->] (bottom') edge node[left] {\cite[Theorem 1.17]{ckw2}}  (top');
        
        \node (phr') at (9, -10) {$\PHR$};
        \node (phr) at (0, -10) {$\PHR$};
        \node (ehr) at (0, -12) {$\EHR$};
        
        \path[->] (bottom') edge node[left] {\cite[Prop 1.17]{ckw2}} (phr');
        \path[->] (phr') edge node[below] {\eqref{PhrBackward}} (phr);
        \path[->] (phr) edge node[right] {clear from definition} (ehr);
        
    \end{tikzpicture}
\end{framed}
\end{figure}

The implication $\HK \Longrightarrow \J + \SCSJ$ in \cite[Theorem 1.13]{ckw} does not require $\RVD$:

\begin{prop} \label{HkImpliesJScsj}
Suppose $(M, d, \mu, \mathcal{E}, \mathcal{F})$ satisfies Assumption \ref{BasicAssumptions}, $\VD$ holds, and $\phi$ is of regular growth. Then
\begin{equation*}
    \HK \Longrightarrow \J + \SCSJ.
\end{equation*}
\end{prop}

\begin{proof}
Under $\VD$,
\begin{equation*}
    \HK \Longrightarrow \J \qquad\mbox{(by \cite[Proposition 3.3]{ckw})}
\end{equation*}
and
\begin{align*}
    \HK &\Longleftrightarrow \UHK + \LHK &\qquad\mbox{(by definition)}\\
    & \Longrightarrow \UHK + (\mbox{$(\mathcal{E}, \mathcal{F})$ is conservative}) &\qquad\mbox{(by \cite[Proposition 3.1]{ckw})}\\
    &\Longrightarrow \SCSJ &\qquad\mbox{(by \cite[Proposition 3.6]{ckw})}.
\end{align*}
\end{proof}

This allows us to prove Theorem \ref{ourMainResult1}.

\begin{proof}[Proof of Theorem \ref{ourMainResult1}]
See Figure \ref{HkStableDiagram}. The use of the auxiliary space is justified because if $(M, d, \mu, \mathcal{E}, \mathcal{F})$ satisfies any of $\HK$, $\J+\SCSJ$, $\J+\CSJ$, or $\J+\QE$, then $(\mathcal{E}, \mathcal{F})$ admits a jump kernel. By Proposition \ref{CkwAppliesToAuxiliaryProp}, the auxiliary space satisfies Assumption \ref{BasicAssumptions}, $\VD$, and $\RVD$ so the main result of \cite{ckw} applies to the auxiliary space.
\end{proof}

\subsection*{Stability of \texorpdfstring{$\PHI$}{TEXT}}

For the proof of Theorem \ref{ourMainResult3}, it will help to first establish that $\NDL \Longrightarrow \QE$.
We obtain this by a minor change to the proof that $\NDL \Longrightarrow \E$ (under $\VD$ and $\RVD$) in \cite[Proposition 3.5]{ckw2}.

\begin{prop} \label{NdlImpliesQe} Suppose $(M, d, \mu, \mathcal{E}, \mathcal{F})$ satisfies Assumption \ref{BasicAssumptions}, $\VD$ and $\QRVD$ hold, and $\phi$ is of regular growth. Then $\NDL \Longrightarrow \QE$.
\end{prop}

\begin{proof}
Assume $\NDL$.
The proof of $\Egeq$ is the same as in \cite[Proposition 3.5.ii]{ckw2} (and only uses $\VD$). The proof of $\QEleq$ follows from applying the same argument that \cite{ckw2} uses to prove $\Eleq$ (assuming $\VD$ and $\RVD$), but only to sufficiently large $r$ (which is all that is needed for $\QEleq$).
\end{proof}

\begin{proof}[Proof of Theorem \ref{ourMainResult3}]
In Figure \ref{UhkdImpliesUhkDiagram}, we prove that $\UHKD + \QE + \Jleq \Longrightarrow \UHK$. Then in Figure \ref{PhiImpliesDiagram}, we use this to show that $\PHI \Longrightarrow \PI + \Jleq + \CSJ + \UJS$. We complete the argument with Figure \ref{PhiStableDiagram}.
\end{proof}


\section{Escape times} \label{EscapeSection}

In this section, we prove \eqref{EleqBothways}, \eqref{EgeqBackwards}, and \eqref{EgeqForwards}: the implications from Proposition \ref{MasterProp} relating to escape times.

Let us start by calculating the escape times $\mathbb{E}^{z_0}\widehat{\tau}_{\Bhat(z_0, r)}$ for all $z_0 \in \Mhat$, $r>0$. For large values of $r$ (larger than $D_{\pi(z_0)}$), we will use \eqref{couple} to compare $\mathbb{E}^{z_0} \widehat{\tau}_{\Bhat(z_0, r)}$ to an escape time on the original space. For small $r$ (smaller than $D_{\pi(z_0)}$), we use the ultrametric property \eqref{WxUltrametricBalls} on $W_{\pi(z_0)}$ to show that $\widehat{\tau}_{\Bhat(z_0, r)}$ is exponentially-distributed, and then we compute its parameter.

Recall from \eqref{couple} that $\{\pi(\Xhat_t)\} \overset{d}{=} \{X_t\}$. In other words, we can couple the jump processes $\{X_t\}$ and $\{\Xhat_t\}$ so that $\{\Xhat_t\}$ jumps around on $W_x$ while $\{X_t\}$ is being held at $x$, and $\{\Xhat_t\}$ jumps from $W_x$ to $W_y$ whenever $\{X_t\}$ jumps from $x$ to $y$. If the jump processes are coupled this way, then $ \widehat{\tau}_{\Bhat(z_0, r)} = \tau_{B(x, r)}$ whenever $x \in M$, $z_0 \in W_x$, $r>0$, and $r \geq D_x$. Therefore,
\begin{equation}\label{escapeTimePropBigR}
    \mathbb{E}^{z_0} \widehat{\tau}_{\Bhat(z_0, r)} = \mathbb{E}^x \tau_{B(x, r)} \qquad\mbox{whenever $x \in M$, $z_0 \in W_x$, and $r \geq D_x$}.
\end{equation}
If on the other hand we have $r<D_{\pi(z_0)}$, the following lemma tells us the distribution and expectation of $\mathbb{E}^{z_0}\widehat{\tau}_{\Bhat(z_0, r)}$.
Recall that we say an exponential random variable $\xi$ has \textit{rate} $\lambda$ if $\mathbb{E}\xi = 1/\lambda$. Also recall how the notation $v(x)$ was defined, in \eqref{v(x)Formula}.

\begin{lemma}\label{escapeTimePropSmallR}
Suppose $(M, d, \mu, \mathcal{E}, \mathcal{F})$ satisfies Assumptions \ref{BasicAssumptions}, $\phi$ is of regular growth, $(\mathcal{E}, \mathcal{F})$ admits a jump kernel, and $M=M_A \cup M_C$. Suppose $x \in M_A$, $z_0=(x, w_0) \in W_x$, and $0<r \leq D_x$ Let $m$ be the non-negative integer such that $r \in \left( d^{D_x}_{m+1}, d^{D_x}_m \right]$ (where $d^{D_x}_{m+1}$ and $d^{D_x}_{m}$ are as defined in \eqref{dDm}), or equivalently,
\begin{equation} \label{PhirComparedToPhidx}
    2^{-(m+1)} \phi(D_x) < \phi(r) \leq 2^{-m} \phi(D_x).
\end{equation}
Then

(a) The escape time $\widehat{\tau}_{\Bhat(z_0, r)}$ is exponentially-distributed, with rate
\begin{equation*}
    \frac{v(x) \phi(D_x) + 2^{m+1}-2}{\phi(D_x)}.
\end{equation*}

(b) The expected value of $\widehat{\tau}_{\Bhat(z_0, r)}$ is bounded above and below by
\begin{equation}\label{EscapeTimeExponentialBounds}
    \frac{2^m}{v(x) \phi(D_x) + 2^{m+1}-2} \phi(r) \leq \mathbb{E}^{z_0} \widehat{\tau}_{\Bhat(z_0, r)} < \frac{2^{m+1}}{v(x) \phi(D_x) + 2^{m+1}-2} \phi(r).
\end{equation}
\end{lemma}

\begin{proof}
Recall the notation $A^w_m$ for spheres (from $\eqref{AwmDef}$) and $E^w_m$ for balls (from \eqref{EwmDef}).

(a) By \eqref{BhatFormula}, $\Bhat(z_0, r) = \{x\} \times E^{w_0}_m$.
Let $z_1=(x, w_1)$ be any point in $\Bhat(z_0, r)$. By the ultrametric property \eqref{WxUltrametricBalls} on $W_x$, $\Bhat(z_0, r) = \Bhat(z_1, r)$. Therefore,
\begin{align*}
    \int_{\Mhat \setminus \Bhat(z_0, r)} \Jhat(z_1, z) \muhat(dz) &= \int_{\Mhat \setminus \Bhat(z_1, r)} \Jhat(z_1, z) \muhat(dz) \\
    &= \int_{\Mhat \setminus W_x} \Jhat(z_1, z) \muhat(dz) + \int_{W_x \setminus \Bhat(z_1, r)} \Jhat(z_1, z) \muhat(dz)\\
    &= \int_{M \setminus \{x\}} J(y, x) \mu(dy) + \sum_{j=1}^m \int_{\{x\} \times A^{w_1}_j} \Jhat(z_1, z) \muhat(dz) &\mbox{(by \eqref{JhatFormula} and \eqref{integralsAgree})}\\
    &= v(x) + \sum_{j=1}^m \mu(x) \nu(A^{w_1}_j) \cdot \frac{4^j}{\mu(x) \phi(D_x)} &\mbox{(by \eqref{v(x)Formula} and \eqref{JhatFormulaIntermsofM})} \\
    &= v(x) + \sum_{j=1}^m 2^{-j} \cdot \frac{4^j}{\phi(D_x)} &\mbox{(by \eqref{NuDef})}\\
    &= v(x) + \frac{1}{\phi(D_x)} \sum_{j=1}^m 2^j = \frac{v(x) \phi(D_x) + 2^{m+1}-2}{\phi(D_x)}
\end{align*}
In the third equality, we use the fact that $W_x \setminus \Bhat(z_1, r) = \bigcup_{j=1}^m \left( \{x\} \times A^{w_1}_m \right)$ because $\Bhat(z_0, r) = E^{w_1}_m$. The quantity $\int_{\Mhat \setminus \Bhat(z_0, r)} \Jhat(z_1, z) \muhat(dz)$ should be thought of as the rate of jumps outside $\Bhat(z_0, r)$ from $z_1$. Since it is the same for all $z_1 \in \Bhat(z_0, r)$, we can conclude that $\widehat{\tau}_{\Bhat(z_0, r)}$ is exponentially distributed, with rate $\left( v(x) \phi(D_x) + 2^{m+1}-2\right)/\phi(D_x)$.

(b)
Since the mean of an exponential random variable is the reciprocal of its rate,
\begin{equation} \label{EscapePropExponential}
    \mathbb{E}^{z_0} \widehat{\tau}_{\Bhat(z_0, r)} = \frac{\phi(D_x)}{v(x) \phi(D_x) + 2^{m+1}-2}.
\end{equation}
By \eqref{PhirComparedToPhidx},
\begin{equation}\label{EscapeProp1}
    2^m \phi(r) \leq \phi(D_x) < 2^{m+1} \phi(r).
\end{equation}
Equation \eqref{EscapeTimeExponentialBounds} follows from plugging \eqref{EscapeProp1} into \eqref{EscapePropExponential}.
\end{proof}

Now that we understand $\mathbb{E}^z \widehat{\tau}_{\Bhat(z_0, r)}$ for all $z_0 \in \Mhat$ and $r>0$, we can prove \eqref{EleqBothways}, \eqref{EgeqBackwards}, and \eqref{EgeqForwards}. The proofs of \eqref{EgeqBackwards} and the $\Longleftarrow$ direction of \eqref{EleqBothways} are direct applications of \eqref{escapeTimePropBigR}. For the $\Longrightarrow$ direction of \eqref{EleqBothways}, we also use the upper bound from Lemma \ref{escapeTimePropSmallR}(b) to handle small values of $r$. For \eqref{EgeqForwards}, we use the lower bound from Lemma \ref{escapeTimePropSmallR}(b) to handle small values of $r$, but this is only useful if the set $\left\{ v(x) \phi(D_x) \right\}_{x \in M_A}$ is bounded above. Therefore, we also need to use \eqref{bowtie1} to control $v(x) \phi(D_x)$. This is where the additional assumption of $\VD$ and $\Jleq$ in \eqref{EgeqForwards} comes from.

\begin{proof}[Proof of Proposition \ref{MasterProp}, implication \eqref{EleqBothways} ($\Longleftarrow$ direction)]
Suppose $\Eleq$ holds for the auxiliary space.\\
There exists a $C>0$ such that $\mathbb{E}^{z_0} \widehat{\tau}_{\Bhat(z_0, r)} \leq C\phi(r)$ for all $z_0 \in \Mohat$, $r>0$. Fix a point $x \in M_0$ and a positive $r \geq D_x$. If $z_0$ is some representative of $W_x$, then by \eqref{escapeTimePropBigR},
\begin{equation*}
    \mathbb{E}^x \tau_{B(x, r)} = \mathbb{E}^{z_0} \widehat{\tau}_{\Bhat(z_0, r)} \leq C\phi(r).
\end{equation*}
Thus, $\QEleq$ holds for the original space.
\end{proof}

\begin{proof}[Proof of Proposition \ref{MasterProp}, implication \eqref{EgeqBackwards}]
Suppose $\Egeq$ holds for the auxiliary space. There exists a $c>0$ such that $\mathbb{E}^{z_0} \widehat{\tau}_{\Bhat(z_0, r)} \geq c\phi(r)$ for all $z_0 \in \Mohat$, $r>0$. Fix a point $x \in M_0$ and a positive $r \geq D_x$. If $z_0$ is some representative of $W_x$, then by \eqref{escapeTimePropBigR},
\begin{equation*}
    \mathbb{E}^x \tau_{B(x, r)} = \mathbb{E}^{z_0} \widehat{\tau}_{\Bhat(z_0, r)} \geq c\phi(r).
\end{equation*}
Thus, $\QEgeq$ holds for the original space.
\end{proof}

\begin{proof}[Proof of Proposition \ref{MasterProp}, implication \eqref{EleqBothways} ($\Longrightarrow$ direction)]
Suppose $\QEleq$ holds for the original space.\\ There exists a $C>0$ such that $\mathbb{E}^{z_0} \tau_{B(x, r)} \leq C\phi(r)$ for all $x \in M_0$ and all positive $r \geq D_x$.

Fix $z_0 \in \Mohat$ and $r>0$. Let $x=\pi(z_0)$.
If $r \geq D_x$, then by \eqref{escapeTimePropBigR},
\begin{equation*}
    \mathbb{E}^{z_0} \widehat{\tau}_{\Bhat(z_0, r)} = \mathbb{E}^x \tau_{B(x, r)} \leq C \phi(r).
\end{equation*}
Suppose instead that $r \leq D_x$. Let $m$ be the non-negative integer such that $2^{-(m+1)} \phi(D_x) < \phi(r) \leq 2^{-m} \phi(D_x)$. We will handle the cases of $m=0$ and $m \geq 1$ separately. If $m=0$, then $\widehat{\tau}_{\Bhat(z_0, r))} \leq \widehat{\tau}_{W_x}$, since the process escapes $\Bhat(z_0, r)$ before it escapes $W_x$, so
\begin{align*}
    \mathbb{E}^{z_0} \widehat{\tau}_{\Bhat(z_0, r)} &\leq \mathbb{E}^{z_0} \widehat{\tau}_{W_x} = \mathbb{E}^{z_0} \widehat{\tau}_{\Bhat(z_0, D_x)}\\
    &= \mathbb{E}^x \tau_{B(x, D_x)} &\mbox{(by \eqref{escapeTimePropBigR})}\\
    &\leq C \phi(D_x) = C \frac{\phi(D_x)}{\phi(r)} \phi(r) \leq 2C \phi(r).
\end{align*}
If $m \geq 1$, then $\frac{2^{m+1}}{2^{m+1} - 2} \leq \frac{4}{2} = 2$, so
\begin{align*}
    \mathbb{E}^{z_0} \widehat{\tau}_{\Bhat(z_0, r)} &\leq \frac{2^{m+1}}{v(x) \phi(D_x) + 2^{m+1} - 2} \phi(r) &\mbox{(by Lemma \ref{escapeTimePropSmallR}(b))}\\
    &\leq \frac{2^{m+1}}{2^{m+1} - 2} \phi(r) = 2\phi(r).
\end{align*}

Thus, $\Eleq$ holds for the auxiliary space.
\end{proof}

\begin{proof}[Proof of Proposition \ref{MasterProp}, implication \eqref{EgeqForwards}]
Suppose $\QEgeq$, $\VD$, and $\Jleq$ all hold for the original space. By $\QEgeq$, there exists a $c>0$ such that $\mathbb{E}^{z_0} \tau_{B(x, r)} \geq c\phi(r)$ for all $x \in M_0$ and all positive $r \geq D_x$. Since $\VD$ and $\Jleq$ hold, by \eqref{bowtie1}, there exists a $C_{\mathcal{J}}>0$ such that $v(x) \phi(D_x) \leq C_{\mathcal{J}}$ for all $x \in M_A$.

Fix $z_0 \in \Mohat$ and $r>0$. Let $x=\pi(z_0)$.
If $r \geq D_x$, then by \eqref{escapeTimePropBigR},
\begin{equation*}
    \mathbb{E}^{z_0} \widehat{\tau}_{\Bhat(z_0, r)} = \mathbb{E}^x \tau_{B(x, r)} \geq c \phi(r).
\end{equation*}
Suppose instead that $r \leq D_x$. Let $m$ be the non-negative integer such that $2^{-(m+1)} \phi(D_x) < \phi(r) \leq 2^{-m} \phi(D_x)$. Then
\begin{align} \label{EgeqForwardsSmallRCalculation}
\begin{split}
    \mathbb{E}^{z_0} \widehat{\tau}_{\Bhat(z_0, r)} &\geq \frac{2^m}{v(x) \phi(D_x) + 2^{m+1} - 2} \phi(r) \qquad\mbox{(by Lemma \ref{escapeTimePropSmallR}(b))}\\
    &= \frac{1}{2 + \left(v(x)\phi(D_x)-2\right) 2^{-m}} \phi(r)\\
    &\geq \frac{1}{2 + \left(v(x)\phi(D_x)\right) 2^{-m}} \phi(r)\\
    &\geq \frac{1}{2 + C_{\mathcal{J}} 2^{-m}} \phi(r) \geq \frac{1}{2+C_{\mathcal{J}}} \phi(r).
\end{split}
\end{align}
(We need not worry about any of the denominators from \eqref{EgeqForwardsSmallRCalculation} being non-positive, since $v(x) \phi(D_x) + 2^{m+1} - 2 > 2^{m+1} - 2 \geq 0$, and each denominator is replaced with something at least as large.)

Thus, $\Egeq$ holds for the auxiliary space.
\end{proof}

\section{The heat kernel of the auxiliary space} \label{HeatKernelSection1}

In this section, we derive the following explicit formula for the heat kernel of the jump process on the auxiliary space, in terms of the heat kernel of the jump process on the original space. Recall from \eqref{v(x)Formula} that for an $x \in M_A$, $v(x)$ denotes the rate at which $\{X_t\}$ leaves $x$. Recall from \eqref{AwmDef} that $A^w_m$ denotes the set of all $w' \in W$ such that $m$ is the first index of disagreement between $w$ and $w'$, and that $A^w_m$ is a sphere in the metric space $(W, \rho^D)$ for all $D$.

\begin{prop}\label{PhatIsTheHeatKernelProp}
Suppose $(M, d, \mu, \mathcal{E}, \mathcal{F})$ satisfies Assumption \ref{BasicAssumptions}, $\phi$ is of regular growth, $(\mathcal{E}, \mathcal{F})$ admits a jump kernel, and $M=M_A \cup M_C$. If $\{X_t\}$ has a heat kernel $p(t, x, y)$, then $\{\Xhat_t\}$ also has a heat kernel $\phat$, and it is given by the formula
\begin{equation}\label{PhatFormula}
    \phat(t, z, z') = \left\{
    \begin{matrix}
        p(t, \pi(z), \pi(z')) &:& \mbox{if $\pi(z) \neq \pi(z')$, or if $z = z' \in M_C$}\\
        \\
        p(t, x, x) + \frac{e^{-v(x) t}}{\mu(x)} \left( a(t, m, D_x) -1 \right) &:& \mbox{if $x \in M_A$, $z=(x, w)$, $z'=(x, w')$, $w' \in A^w_m$}\\
        \\
        p(t, x, x) + \frac{e^{-v(x) t}}{\mu(x)} \left( a(t, \infty, D_x) -1 \right) &:& \mbox{if $x \in M_A$ and $z=z' \in W_x$}
    \end{matrix}\right.
\end{equation}
where
\begin{equation*}
    a(t, m, D) := 1+\sum_{j=1}^{m-1} 2^{j-1} \exp \left( - \frac{(3 \cdot 2^j-2) t}{\phi(D)} \right) - 2^{m-1} \exp \left( - \frac{(3 \cdot 2^m-2) t}{\phi(D)} \right)
\end{equation*}
and
\begin{equation*}
    a(t, \infty, D) := \lim_{m \to \infty} a(t, m, D) = 1 + \sum_{j=1}^\infty 2^{j-1} \exp\left( - \frac{(3 \cdot 2^j -2) t)}{\phi(D)} \right).
\end{equation*}
\end{prop}

Proposition \ref{PhatIsTheHeatKernelProp} will be the most important tool used to prove \eqref{UhkBackwards}-\eqref{LhkBackwards} and \eqref{NdlBackwards}-\eqref{UhkdForwards} in Section \ref{HeatKernelSection2}.

In order to prove Proposition \ref{PhatIsTheHeatKernelProp}, it will help to first consider a jump process $\{Z^x_t\}$, which we construct below, that takes values on $W_x$ for some $x \in M_A$, and has a jump kernel equal to that of $\{\Xhat_t\}$ restricted to $W_x \times W_x$. One can think of $\{Z^x_t\}$ as having the same law as $\{\Xhat_t\}$, if $\{\Xhat_t\}$ never left $W_x$. However, in order to better analyze heat kernels, we define $\{Z^x_t\}$ not in terms of its jump kernel, but by a different probabilistic interpretation, involving ``refreshings" in which every so often, the process jumps to a uniform point on a ball that it is currently in. We define $\{Z^x_t\}$ in these terms, then show that it has the same jump kernel as $\{\Xhat_t\}$ (if $\{\Xhat_t\}$ never left $W_x$), and then use the ``refreshings" interpretation of $\{Z^x_t\}$ to study the heat kernels of both processes.

Recall that $W$ is the set of infinite binary strings of the form $w = (w(1), w(2), w(3), \dots)$, and that $E^w_m$ is defined as the set of $w' \in W$ such that the first $m$ indices of $w$ and $w'$ agree. For all $x \in M_A$ and $w \in W$, let us refer to $\{x\} \times E^w_m$ as the \textit{$m$-cell} of $(x, w)$.

Fix $x \in M_A$. Let $\{\lambda_m\}_{m=0}^\infty$ be the following sequence (which depends on $x$):
\begin{equation} \label{LambdamDef}
    \lambda_m := \left\{ \begin{matrix}
        \frac{4}{\phi(D_x)} &:& \mbox{if $m=0$}\\
        \\
        \frac{3 \cdot 2^m}{\phi(D_x)} &:& \mbox{if $m \in \mathbb{N}$}.
    \end{matrix}\right.
\end{equation}
Let $\{Z^x_t\}$ be a $W_x$-valued Markovian jump process with the following rules:

\begin{itemize}
    \item For each $m$, let $\{N^m(t)\}_{t \geq 0}$ be a Poisson point process on $(0, \infty)$ with intensity $\lambda_m$. Whenever an event in $N^m$ occurs, the process $\{Z^x_t\}$ jumps to a uniform point on its current $m$-cell. We call this event a ``type-$m$ refreshing."
    \item The Poisson point processes and the destinations of the jumps are all independent.
\end{itemize}
In other words, for each $m \in \mathbb{N}$ there is a Poisson clock which rings with rate $\lambda_m$. The clocks associated with larger $m$ ring more frequently than those associated with smaller $m$. Whenever clock $m$ rings, if the current location is $(x, (w_1, w_2, w_3, \dots))$, all of the bits $w_j$ for $j \geq m$ are reset, uniformly at random and independently.

Let us compute the jump kernel of $\{Z^x_t\}$, to check that it matches that of $\{\Xhat_t\}$.
Recall that $A^w_m$ is the set of $w' \in W$ such that $m$ is the first index of disagreement between $w$ and $w'$. In order for the process $\{Z^x_t\}$ to jump from $(x, w)$ to a point in $\{x\} \times A^w_m$, a type-$j$ refreshing must occur for some $j<m$, and the uniform point on $E^w_j$ that the process jumps to must happen to be in $\{x\} \times A^w_m$. Type-$j$ refreshings occur with rate $\lambda_j$, and they do, the probability that the process jumps from $(x, w)$ to a point in $\{x\} \times A^w_m$ is $\muhat(\{x\} \times A^w_m) / \muhat(\{x\} \times E^w_j)$. Therefore, jumps from $(x, w)$ to $\{x\} \times A^w_m$ occur with rate
\begin{align*}
    \sum_{j=0}^{m-1} \lambda_j \cdot \frac{\muhat(\{x\} \times A^w_m)}{\muhat(\{x\} \times E^w_j)} &= \sum_{j=0}^{m-1} \lambda_j \cdot \frac{\mu(x) \nu(A^w_m) }{\mu(x) \nu(E^w_j)}\\
    &= \sum_{j=0}^{m-1} \lambda_j \frac{2^{-m}}{2^{-j}} &\mbox{(by \eqref{NuDef} and \eqref{NuOnSpheres})}\\
    &= 2^{-m} \sum_{j=0}^{m-1} 2^j \lambda_j\\
    &= 2^{-m} \frac{\left(4 + \sum_{j=1}^{m-1} 2^j \cdot 3 \cdot 2^j \right)}{\phi(D_x)} &\mbox{(by \eqref{LambdamDef})}\\
    &= 2^{-m} \cdot \frac{4^m}{\phi(D_x)}.
\end{align*}
Let $\Jhat_{Z^x}$ be the jump kernel of $\{Z^x_t\}$. By symmetry, $\Jhat_{Z^x}((x, w), \cdot)$ is constant on $\{x\} \times A^w_m$ for all $m$. By our calculation of the rate of jumps from $(x, w)$ to $\{x\} \times A^w_m$, the value of this constant is
\begin{equation} \label{JzxJumpKernel}
    \Jhat_{Z^x}((x, w), (x, w')) = \frac{2^{-m} \cdot \frac{4^m}{\phi(D_x)}}{\muhat(W_x)} = \frac{2^{-m} \cdot \frac{4^m}{\phi(D_x)}}{2^{-m} \mu(x)} = \frac{4^m}{\mu(x) \phi(D_x)} \qquad\mbox{for $w' \in A^w_m$}.
\end{equation}
By \eqref{JhatFormulaIntermsofM} and \eqref{JzxJumpKernel}, $\Jhat_{Z^x}(z, z') = \Jhat(z, z')$ for all $z, z' \in W_x$.

\begin{remark}\label{XhatFromZx}
Suppose $\{Z^x_t\}$ and $\{\Xhat_t\}$ have the same initial value, in $W_x$. Since $\{Z^x_t\}$ is the $W_x$-valued process with the same transition rates within $W_x$ as $\{\Xhat_t\}$, the two processes can be coupled so that $\{Z^x_t\}$ is independent from $\widehat{\tau}_{W_x}$ (the time that $\{\Xhat_t\}$ first leaves $W_x$), and $\Xhat_t = Z^x_t$ for all $0 \leq t < \widehat{\tau}_{W_x}$.
\end{remark}

It may seem like the construction of $\{Z^x_t\}$ was needlessly convoluted. After all, we could have just defined $\{Z^x_t\}$ as the Hunt process associated with the Dirichlet form $\Ehat^x(f, g) = \int_{W_x \times W_x \setminus \diag_{W_x}} (f(z)-f(z')) (g(z)-g(z')) \, \Jhat(dz, dz')$. However, interpreting $\{Z^x_t\}$ in terms of type-$m$ refreshings makes it much easier to calculate the distribution of $Z^x_t$ (for a given $t>0$ and initial value $Z^x_0$). In turn, this helps us to calculate the distribution of $\Xhat_t$ for a given $t>0$ and an initial value $\Xhat_0$.

In the following lemma, we use the ``refeshings" interpretation of $\{Z^x_t\}$ to determine the distribution of $Z^x_t$, for a given $t>0$ and $Z^x_0 = (x, w)$. We find the distribution of the largest $m$ such that the process $\{Z^x_t\}$ has had a type-$m$ refreshing by time $t$, and then argue that the conditional distribution of $Z^x_t$ is uniform on the $m$-cell of $(x, w)$, for this $m$.

\begin{lemma} \label{ZxtKernel}
If $x \in M_A$, $w \in W$, $t>0$, and the initial value of $\{Z^x_t\}$ is $(x, w)$, then
\begin{equation*}
    \mathbb{P}\left(Z^x_t \in \{x\} \times A^w_m \right) = 2^{-m} a(t, m, D_x).
\end{equation*}
\end{lemma}

\begin{proof}
For all $m \in \mathbb{Z}_+$, let $T_m$ be the time that the first event in the Poisson point process $\{N^m(t)\}$ associated with $m$ occurs. Then $T_m$ is exponential with rate $\lambda_m$ (i.e., $\mathbb{E}(T_m) = 1/\lambda_m$) and the variables $\{T_m\}_{m=0}^\infty$ are independent.

For all $t>0$, let $J_t := \min\{m : T_m \leq t \}$.
Conditional on $J_t$, the location of $Z^x_t$ is uniform on $\{x\} \times E^w_{J_t}$. Therefore,
\begin{dmath} \label{CondOnJt}
    \mathbb{P}\left( Z^x_t \in \{x\} \times A^w_m \Big| J_t=j \right) = \frac{\muhat \left( \left( \{x\} \times A^w_m \right) \cap \left( \{x\} \times E^w_j \right) \right)}{\muhat \left( \{x\} \times E^w_j \right)}
    = \frac{\mu(x) \nu(A^w_m \cap E^w_j)}{\mu(x) \nu(E^w_j)}
    = \left\{ \begin{matrix}
        \frac{2^{-m}}{2^{-j}} &:& \mbox{if $j<m$}\\
        \\
        0 &:& \mbox{if $j \geq m$}.
    \end{matrix}\right.
\end{dmath}
Conditioning on $J_t$ gives us
\begin{align}\label{ProbZonsphere}
\begin{split}
    \mathbb{P}\left( Z^x_t \in \{x\} \times A^w_m \right) &= \sum_{j=0}^\infty \mathbb{P}\left( J_t=j  \right) \mathbb{P}\left( Z^x_t \in \{x\} \times A^w_m \Big|J_t=j \right)\\
    &= \sum_{j=0}^{m-1} \mathbb{P}(J_t=j) \cdot 2^{j-m} \qquad\mbox{(by \eqref{CondOnJt})}\\
    &= 2^{-m} \sum_{j=0}^{m-1} 2^j \mathbb{P}(J_t=j).
\end{split}
\end{align}

Since $J_t=j$ if and only if \{$T_j \leq t$, and $T_k > t$ for all $k<j$\}, we have
\begin{equation}  \label{PJtj0}
    \mathbb{P}(J_t = 0) = \mathbb{P}(T_0 \leq t) = 1 - e^{-\lambda_0 t} = 1 - e^{-4t/\phi(D_x)}
\end{equation}
and for all $j \geq 1$, we have
\begin{dmath}  \label{PJtj}
    {\mathbb{P}(J_t=j)} = {\mathbb{P}(T_j \leq t)} \prod_{k=0}^{j-1} {\mathbb{P}(T_k>t)}
    = (1-e^{-\lambda_j t}) \prod_{k=0}^{j-1} e^{-\lambda_k t}
    = (1-e^{-\lambda_j t}) \exp \left( - t \sum_{k=0}^{j-1} \lambda_k \right)
    = \left(1-\exp\left( - \frac{t}{\phi(D_x)}(3 \cdot 2^j) \right) \right) \exp \left( - \frac{t}{\phi(D_x)} \cdot \left( 4 + 3 \sum_{k=1}^{j-1} 2^k \right) \right)
    = \left(1-\exp\left( - \frac{t}{\phi(D_x)}(3 \cdot 2^j) \right) \right) \exp \left( - \frac{t}{\phi(D_x)} \cdot \left( 3 \cdot 2^j - 2 \right) \right)
    = \exp \left( - \frac{t}{\phi(D_x)} \left( 3 \cdot 2^j - 2 \right) \right) - \exp \left( - \frac{t}{\phi(D_x)} \left( 3 \cdot 2^{j+1} - 2 \right) \right).
\end{dmath}
By plugging \eqref{PJtj0} and \eqref{PJtj} into \eqref{ProbZonsphere} and grouping like terms, we obtain
\begin{dmath*}
     \mathbb{P}\left( Z^x_t \in \{x\} \times A^w_m \right) = 2^{-m} \left( {1 - e^{-4t/\phi(D_x)} + \sum_{j=1}^{m-1}} 2^j \exp \left( - \frac{t}{\phi(D_x)} \left( 3 \cdot 2^j - 2 \right) \right) - \sum_{j=1}^{m-1} 2^j \exp \left( - \frac{t}{\phi(D_x)} \left( 3 \cdot 2^{j+1} - 2 \right) \right)  \right)
     = 2^{-m} \left( 1+\sum_{j=1}^{m-1} 2^{j-1} \exp \left( - \frac{(3 \cdot 2^j-2) t}{\phi(D_x)} \right) - 2^{m-1} \exp \left( - \frac{(3 \cdot 2^m-2) t}{\phi(D_x)} \right) \right)
    = 2^{-m} a(t, m, D_x).
\end{dmath*}
\end{proof}

Next, we calculate the distribution of $\Xhat_t$ for a given $t>0$ and $\Xhat_0 = z$. This mostly consists of finding $\mathbb{P}^z (\Xhat_t \in S)$ for various sets $S \subseteq \Mhat$. We use Remark \ref{XhatFromZx} and Lemma \ref{ZxtKernel} when $z \in W_x$ and $S \subseteq W_x$ for some $x \in M_A$. In every case, we show that $\mathbb{P}^z(\Xhat_t \in S)$ is equal to exactly what it would be if $\phat$ was the heat kernel of $\{\Xhat_t\}$: namely $\mathbb{P}^z(\Xhat_t \in S) = \int_{\Mhat} \phat(t, z, z') \muhat(dz')$.

\begin{lemma}\label{DistributionHeatKernelPhat}
Suppose $(M, d, \mu, \mathcal{E}, \mathcal{F})$ satisfies Assumption \ref{BasicAssumptions}, $\phi$ is of regular growth, $(\mathcal{E}, \mathcal{F})$ admits a jump kernel, $M=M_A \cup M_C$, and $\{X_t\}$ has a heat kernel $p$.
Let $\phat$ be the right-hand side of \eqref{PhatFormula}. If $z \in \Mhat$, $t>0$, and $S \subseteq \Mhat$, then
\begin{equation*}
    \mathbb{P}^z(\Xhat_t \in S) = \int_S \phat(t, z, z') \muhat(dz').
\end{equation*}
\end{lemma}

\begin{proof}
Fix $x \in M$, $z \in W_x$, and $t>0$. (The point $x$ can be in either $M_A$ or $M_C$.) Suppose the process $\{\Xhat_t\}$ has initial value $z$. We will find the distribution of $\Xhat_t$ by finding $\mathbb{P}^z(\Xhat_t \in S)$ for various measurable sets $S \subseteq \Mhat$.

By \eqref{couple},
\begin{equation} \label{SSubsetMc}
    \mathbb{P}^z(\Xhat_t \in S) = \mathbb{P}^x(X_t \in S) = \int_S p(t, x, y) \mu(dy) \qquad\mbox{for $S \subseteq M_C$}.
\end{equation}
If $y \in M_A$, then by \eqref{couple},
\begin{equation}\label{ProbXhattInWy}
    \mathbb{P}^z(\Xhat \in W_y) = \mathbb{P}^x(X_t=y) = \mu(y) p(t, x, y).
\end{equation}
For $y \neq x$, whenever the process $\{\Xhat_t\}$ first reaches $W_y$, it jumps to a uniform point on $W_y$. Therefore, if $S \subseteq W_y$ for some $y \in M_A$ such that $y \neq x$,
\begin{equation}\label{XhattUniformOnWy}
    \mathbb{P}^z\left(\Xhat_t \in S \Big| \Xhat_t \in W_y\right) = \frac{\muhat(S)}{\muhat(W_y)} = \frac{\muhat(S)}{\mu(y)}.
\end{equation}

Now assume $x \in M_A$. If the process $\{\Xhat_t\}$ leaves but then returns to $W_x$, it returns to a uniform point on $W_x$. Therefore, if $x \in M_A$ and $S \subseteq W_x$,
\begin{equation} \label{XhattUniformOnWxCond}
    \mathbb{P}^z \left( \Xhat_t \in S \Big| \mbox{$\widehat{\tau}_{W_x} \leq t$ and $\Xhat_t \in W_x$} \right) = \frac{\muhat(S)}{\muhat(W_x)} = \frac{\muhat(S)}{\mu(x)}.
\end{equation}
(The event ``$\widehat{\tau}_{W_x} \leq t$ and $\Xhat_t \in W_x$," which we are conditioning on in \eqref{LeftAndReturned}, means that the process $\{\Xhat_t\}$ has left $W_x$, but then returned (at least once) and is currently back in $W_x$ at time $t$.) By \eqref{ProbXhattInWy} and \eqref{XhattUniformOnWy},
\begin{equation} \label{SSubsetWyYnotx}
    \mathbb{P}^z(\Xhat_t \in S) = \mu(y) p(t, x, y) \cdot \frac{\muhat(S)}{\mu(y)} = \muhat(S) p(t, x, y) \qquad\mbox{whenever $y \in M_A$, $y \neq x$, and $S \subseteq W_y$}.
\end{equation}
Note that if $x \in M_A$, $\tau_{\{x\}}$ is exponential($v(x)$), where $v(x)$ is as defined in \eqref{v(x)Formula}. Therefore, if $x \in M_A$, the probability that the process $\{X_t\}$ has never yet left $x$ by time $t$ is
\begin{equation} \label{HkSectionExponential}
    \mathbb{P}^x(\tau_{\{x\}}>t) = e^{-v(x) t}.
\end{equation}
By \eqref{HkSectionExponential} and \eqref{couple}, the probability that the process $\{\Xhat_t\}$ has never left $W_x$ by time $t$ is
\begin{equation}\label{WxHoldingExpo}
    \mathbb{P}^z(\widehat{\tau}_{W_x} >t) = e^{-v(x) t}.
\end{equation}
By \eqref{ProbXhattInWy} and \eqref{WxHoldingExpo}, the probability that the process $\{\Xhat_t\}$ has left $W_x$, then returned at least once, and is currently back in $W_x$ at time $t$ is
\begin{dmath}\label{LeftAndReturned}
    {\mathbb{P}^z\left( \mbox{$\widehat{\tau}_{W_x} \leq t$ and $\Xhat_t \in W_x$} \right)} = {\mathbb{P}^z(\Xhat_t \in W_x) - \mathbb{P}^z \left( \mbox{$\widehat{\tau}_{W_x} > t$ and $\Xhat_t \in W_x$} \right)}
    = \mu(x) p(t, x, x) - e^{-v(x) t}.
\end{dmath}
By and \eqref{LeftAndReturned} and \eqref{XhattUniformOnWxCond}, if $S \subseteq W_x$, then the probability that the process has left $W_x$, then returned at least once, and is currently in $S \subseteq W_x$ at time $t$ is
\begin{dmath} \label{LeftAndReturnedS}
    {\mathbb{P}^z \left( \mbox{$\widehat{\tau}_{W_x \leq t}$ and $\Xhat_t \in S$} \right)}={\mathbb{P}^z \left( \mbox{$\widehat{\tau}_{W_x \leq t}$ and $\Xhat_t \in W_y$} \right) \mathbb{P}^z \left( \Xhat_t \in S \Big| \mbox{$\widehat{\tau}_{W_x} \leq t$ and $\Xhat_t \in W_x$} \right)}
    = \left(\mu(x) p(t, x, x) - e^{-v(x) t}\right) \cdot \frac{\muhat(S)}{\mu(x)}
    = \muhat(S) \left( p(t, x, x) - \frac{e^{-v(x) t}}{\mu(x)} \right).
\end{dmath}
What about the probability that $\Xhat_t \in S$ (for $S \subseteq W_x$) but the process $\{\Xhat_t\}$ has \textit{never} left $W_x$ by time $t$? If $S \subseteq W_x$, then
\begin{dmath}\label{NeverLeavesS}
    {\mathbb{P}^z \left( \mbox{$\widehat{\tau}_{W_x}>t$ and $\Xhat_t \in S$} \right) } = {\mathbb{P}^z(\widehat{\tau}_{W_x}>t) \mathbb{P}\left( Z^x_t \in S \Big| Z^x_0=z \right) \qquad \mbox{(by Remark \ref{XhatFromZx})}}
    = {e^{-v(x) t} \mathbb{P}\left( Z^x_t \in S \Big| Z^x_0=z \right) \qquad \mbox{(by \eqref{WxHoldingExpo})}}.
\end{dmath}
Let us continue to assume $x \in M_A$. By applying \eqref{LeftAndReturnedS} and \eqref{NeverLeavesS} to $S = \{x\} \times A^w_m$, for some $m \in \mathbb{N}$, the probability that $\Xhat_t \in \{x\} \times A^w_m$ is
\begin{align} \label{ApplyingPartitionBasedOnTauhatWx}
\begin{split}
    \mathbb{P}^z \left(\Xhat_t \in \{x\} \times A^w_m \right) &= \mathbb{P}^z \left( \mbox{$\widehat{\tau}_{W_x \leq t}$ and $\Xhat_t \in \{x\} \times A^w_m$} \right)+\mathbb{P}^z \left( \mbox{$\widehat{\tau}_{W_x > t}$ and $\Xhat_t \in \{x\} \times A^w_m$} \right)\\
    &= \mu(x) 2^{-m} \left( p(t, x, x) - \frac{e^{-v(x) t}}{\mu(x)} \right) + e^{-v(x) t} \mathbb{P} \left( Z^x_t \in \{x\} \times A^w_m \Big| Z^x_0=z \right).
\end{split}
\end{align}
By Lemma \ref{ZxtKernel}, $\mathbb{P} \left( Z^x_t \in \{x\} \times A^w_m \Big| Z^x_0=z \right) = 2^{-m} a(t, m, D_x)$. Therefore, \eqref{ApplyingPartitionBasedOnTauhatWx} becomes
\begin{align}\label{ProbXhattInXTimesAwm}
\begin{split}
    \mathbb{P}^z(\Xhat_t \in \{x\} \times A^w_m) &= \mu(x) 2^{-m} \left( p(t, x, x) - \frac{e^{-v(x) t}}{\mu(x)} \right) + e^{-v(x) t} \, 2^{-m} \, a(t, m, D_x)\\
    &= \mu(x) 2^{-m} \left( p(t, x, x) + \frac{e^{-v(x) t}}{\mu(x)} \left( a(t, m, D_x)-1 \right) \right)\\
    &= \muhat\left(\{x\} \times A^w_m\right) \left( p(t, x, x) + \frac{e^{-v(x) t}}{\mu(x)} \left( a(t, m, D_x)-1 \right) \right).
\end{split}
\end{align}
By symmetry, the distribution of $\Xhat_t$ conditioned on $\Xhat_t \in \{x\} \times A^w_m$ is uniform on $\{x\} \times A^w_m$. Therefore, if $S$ is a subset of $\{x\} \times A^w_m$, we can obtain the probability that $\Xhat_t \in S$ by multiplying the probability from \eqref{ProbXhattInXTimesAwm} by $\muhat(S) / \muhat(\{x\} \times A^w_m)$. The result is
\begin{equation}\label{SSubsetAwm}
    \mathbb{P}^z(\Xhat_t \in S) = \muhat(S) \left( p(t, x, x) + \frac{e^{-v(x) t}}{\mu(x)} \left( a(t, m, D_x)-1 \right) \right) \qquad\mbox{for $S \subseteq \{x\} \times A^w_m$}.
\end{equation}

We have now shown how to calculate $\mathbb{P}^z(\Xhat_t \in S)$ for any initial value $z$ and any measurable $S \subseteq \Mhat$. We can complete the proof of the lemma. Fix $z \in \Mhat$, $t>0$, and $S \subseteq \Mhat$. Let $x=\pi(z)$. We will have two cases for our final calculation: $x \in M_C$ and $x \in M_A$.

If $z=x \in M_C$, then
\begin{align} \label{ProofOfPhatXinMc}
\begin{split}
    \mathbb{P}^z(\Xhat_t \in S) &= \mathbb{P}^z(\Xhat_t \in S \cap M_C) + \sum_{y \in M_A} \mathbb{P}^z(\Xhat_t \in S \cap W_y)\\
    &=\int_{S \cap M_C} p(t, x, z') \mu(dz') + \sum_{y \in M_A} \muhat(S \cap W_y) p(t, x, y) \qquad\mbox{(by \eqref{SSubsetMc} and \eqref{SSubsetWyYnotx})}\\
    &= \int_{S \cap M_C} p(t, x, z') \mu(dz') + \sum_{y \in M_A} \int_{S \cap W_y} p(t, x, y) \muhat(dz')\\
    &= \int_S p(t, z, z') \muhat(dz) \qquad\mbox{(by the definition of $\phat$)}.
\end{split}
\end{align}
Now suppose $x \in M_A$. By partitioning $S$,
\begin{equation}\label{622}
    \mathbb{P}^z(\Xhat_t \in S) = \mathbb{P}^z(\Xhat_t \in S \cap M_C) + \sum_{y \in M_A \setminus \{x\}} \mathbb{P}^z(\Xhat_t \in S \cap W_y) + \sum_{m=1}^\infty \mathbb{P}^z\left(\Xhat_t \in S \cap \left( \{x\} \times A^w_m\right)\right).
\end{equation}
By the same calculations as in \eqref{ProofOfPhatXinMc},
\begin{equation}\label{623}
    \mathbb{P}^z(\Xhat_t \in S \cap M_C) + \sum_{y \in M_A \setminus \{x\}} \mathbb{P}^z(\Xhat_t \in S \cap W_y) = \int_{S \setminus W_x} \phat(t, z, z') \muhat(dz')
\end{equation}
For all $m \in \mathbb{N}$, by \eqref{SSubsetAwm},
\begin{align}\label{624}
\begin{split}
    \mathbb{P}^z\left(\Xhat_t \in S \cap \left( \{x\} \times A^w_m\right)\right) &= \muhat\left( S \cap \left( \{x\} \times A^w_m\right) \right) \left( p(t, x, x) + \frac{e^{-v(x) t}}{\mu(x)} \left( a(t, m, D_x)-1 \right) \right)\\
    &= \int_{S \cap \left( \{x\} \times A^w_m\right)} \phat(t, z, z') \muhat(dz') \qquad\mbox{by the definition of $\phat$}.
\end{split}
\end{align}
By plugging \eqref{623} and \eqref{624} into \eqref{622},
\begin{align*}
\begin{split}
    \mathbb{P}^z(\Xhat_t \in S) &=\int_{S \setminus W_x} \phat(t, z, z') \muhat(dz') + \sum_{m=1}^\infty \int_{S \cap \left( \{x\} \times A^w_m\right)} \phat(t, z, z') \muhat(dz')\\
    &=\int_S \phat(t, z, z') \muhat(dz').
\end{split}
\end{align*}
In either case (whether $x \in M_A$ or $x \in M_C$), we have $\mathbb{P}^z(\Xhat_t \in S) = \int_S \phat(t, z, z') \muhat(dz')$.
\end{proof}

Lemma \ref{DistributionHeatKernelPhat} is not quite enough to guarantee that $\phat$ is the heat kernel of $\{\Xhat_t\}$. We still need to verify that $\phat$ satisfies \eqref{HeatKernelDefinitionDistribution}-\eqref{HeatKernelDefinitionChapmanKolmogorov}. Expressed in terms of the auxiliary space, we must show the following:
\begin{align}
    \mathbb{E}^{z_0} f(\Xhat_t) = \int_\Mhat \phat(t, z_0, z) f(z) \muhat(dz) \qquad&\mbox{for all $f \in L^\infty(\Mhat, \muhat)$, $z_0 \in \Mohat$, and $t>0$}. \label{AuxiliaryHeatKernelDefinitionDistribution}\\
    \phat(t, z_0, z_1) = \phat(t, z_1, z_0) \qquad&\mbox{for all $z_0, z_1 \in \Mohat$ and $t>0$}. \label{AuxiliaryHeatKernelDefinitionSymmetry}\\
    \phat(s+t, z_0, z_1) = \int_\Mhat \phat(t, z_0, z) \phat(t, z, z_1) \muhat(dz) \qquad&\mbox{for all $z_0, z_1 \in \Mohat$ and $s,t > 0$}. \label{AuxiliaryHeatKernelDefinitionChapmanKolmogorov}
\end{align}
It is clear that \eqref{AuxiliaryHeatKernelDefinitionDistribution} follows from Lemma \ref{DistributionHeatKernelPhat}. The definition of $\phat(t, z_0, z_1)$ is symmetric in $z_0$ and $z_1$, so \eqref{AuxiliaryHeatKernelDefinitionSymmetry} holds. However, we do not yet have Chapman-Kolmogorov \eqref{AuxiliaryHeatKernelDefinitionChapmanKolmogorov}. In fact, the $a(t, \infty, D_x)$ term that occurs in the formula for $\phat$ was never used in Lemma \ref{DistributionHeatKernelPhat}. Our proof of \eqref{AuxiliaryHeatKernelDefinitionDistribution} would still hold if we changed the value of $\phat(t, (x, w), (x, w))$ for every $t>0$, $x \in M_A$, and $w \in W$. This is no trivial matter, since we need to understand the on-diagonal heat kernel of $\{\Xhat_t\}$ in order to prove \eqref{UhkdForwards} (which says that $\UHKD$ for the original space implies $\UHKD$ for the auxiliary space).

To prove $\phat$ satisfies Chapman-Kolmogorov, we break down \eqref{AuxiliaryHeatKernelDefinitionChapmanKolmogorov} into three cases:
\begin{itemize}
    \item $z_0 \in W_x$ and $z_1 \in W_y$ for some $x \neq y$.
    \item $z_0 = z_1 = x$ for some $x \in M_C \setminus \mathcal{N}$.
    \item $z_0, z_1 \in W_x$ for some $x \in M_A$.
\end{itemize}
We prove the first two cases by direct calculation. For the third case, choose $w_0, w_1 \in W$ such that $z_0 = (x, w_0)$ and $z_1 = (x, w_1)$. Let $w_1$ vary, and treat the left-hand-side and right-hand-side of \eqref{AuxiliaryHeatKernelDefinitionChapmanKolmogorov} as functions of $w_1$. We use Proposition \ref{AeChapmanKolmogorov} to show that these functions of $w_1$ agree almost-everywhere, and then show that they are both continuous so they agree everywhere.

\begin{lemma} \label{PhatChapmamKolmogorov}
Suppose $(M, d, \mu, \mathcal{E}, \mathcal{F})$ satisfies Assumption \ref{BasicAssumptions}, $\phi$ is of regular growth, $(\mathcal{E}, \mathcal{F})$ admits a jump kernel, $M=M_A \cup M_C$, and $\{X_t\}$ has a heat kernel $p$.
Then $\phat$ (as defined in \eqref{PhatFormula}) satisfies \eqref{AuxiliaryHeatKernelDefinitionChapmanKolmogorov}.
\end{lemma}

\begin{proof}
We will prove this lemma by showing each of the following smaller claims:

(a) If $z_0 \in W_x$ and $z_1 \in W_y$ for some distinct $x, y \in M_0$ (i.e. $x \neq y$), and $s, t > 0$, then
\begin{equation*}
    \phat(s+t, z_0, z_1) = \int_{\Mhat} \phat(s, z_0, z) \phat(t, z, z_1) \, \muhat(dz).
\end{equation*}

(b) If $x \in M_C \setminus \mathcal{N}$ and $s, t > 0$, then
\begin{equation*}
    \phat(s+t, x, x) = \int_{\Mhat} \phat(s, x, z) \phat(t, z, x) \, \muhat(dz).
\end{equation*}

(c) If $x \in M_A$, $w_0, w_1 \in W$, and $s, t > 0$, then
\begin{equation}\label{ChapKolmOnWxEquation}
    \phat(s+t, (x, w_0), (x, w_1)) = \int_\Mhat \phat(s, (x, w_0), z) \phat(t, z, (x, w_1)) \muhat(dz).
\end{equation}
If (a)-(c) are all proved, then we have \eqref{AuxiliaryHeatKernelDefinitionChapmanKolmogorov} for all $z_0, z_1 \in \Mhat$.
The proofs of (a)-(c) are as follows:

(a) By the definition of $\phat$, and since the heat kernel on the original space satisfies Chapman-Kolmogorov,
\begin{equation}\label{629}
    \phat(s+t, z_0, z_1) = p(s+t, x, y) = \int_M p(s, x, x') p(t, x', y) \mu(dy).
\end{equation}
By breaking $M$ into $\{x\}$, $\{y\}$, and $M \setminus (\{x\} \cup \{y\})$, \eqref{629} becomes
\begin{align}\label{630}
\begin{split}
    \phat(s+t, z_0, z_1) =& \mu(x) p(s, x, y) p(t, x, y)\\
    &+ \mu(y) p(s, x, y)p(t, y, y)\\
    &+ \int_{M \setminus (\{x\} \cup \{y\})} p(s, x, x') p(t, x', y) \mu(dy).
\end{split}
\end{align}
We will handle the three terms of the right-hand side of \eqref{630} separately. First,
\begin{align}\label{631}
\begin{split}
    \mu(x) p(s, x, x) p(t, x, y) &= \mathbb{P}^x(X_s = x) \cdot p(t, x, y)
    = \mathbb{P}^{z_0}(\Xhat_s \in W_x) \cdot p(t, x, y) \qquad\mbox{(by \eqref{couple})}\\
    &= \int_{W_x} \phat(s, z_0, z) \muhat(dz) \cdot p(t, x, y)\\
    &= \int_{W_x} \phat(s, z_0, z) \phat(t, z, z_1) \muhat(dz).
\end{split}
\end{align}
By the same argument,
\begin{equation}\label{632}
    \mu(y) p(s, x, y) p(t, y, y) = \int_{W_y} \phat(s, z_0, z) \phat(t, z, z_1) \muhat(dz). 
\end{equation}
For all $x' \in M\setminus (\{x\} \cup \{y\})$, for all $z \in W_{x'}$, we have $\phat(s, z_0, z) \phat(t, z, z_1) = p(s, x, x')p(t, x', y)$. Therefore,
\begin{equation}\label{633}
    \int_{M \setminus (\{x\} \cup \{y\})} p(s, x, x') p(t, x', y) \mu(dy) = \int_{\Mhat \setminus (W_x \cup W_y)} \phat(s, z_0, z) \phat(t, z, z_1) \muhat(dz).
\end{equation}
By using \eqref{631}-\eqref{633} to evaluate the terms on the right-hand side of \eqref{630},
\begin{align*}
    \phat(s+t, z_0, z_1) &= \int_{\Mhat} \phat(s, z_0, z) \phat(t, z, z_1) \muhat(dz).
\end{align*}

(b) Fix $x \in M_C \setminus \mathcal{N}$ and $s, t >0$. By the definition of $\phat$,
\begin{align*}
    \phat(s+t, z_0, z_1) &= p(s+t, x, x)\\
    &= \int_M p(s, x, x') p(t, x', x) \mu(dx')\\
    &= \int_{M \setminus \{x\}} p(s, x, x') p(t, x', x) \mu(dx') & \mbox{(since $\mu(\{x\}) = 0$)}\\
    &= \int_{\Mhat \setminus W_x} \phat(s, z_0, z) \phat(t, z, z_1) \muhat(dz)\\
    &= \int_{\Mhat} \phat(s, z_0, z) \phat(t, z, z_1) \muhat(dz) &\mbox{(since $\muhat(W_x)=0$)}.
\end{align*}

(c) Fix $x \in M_A$, $w_0 \in W$, and $s, t>0$ (but do not fix $w_1$). Let $F(w_1)$ be the left-hand side of \eqref{ChapKolmOnWxEquation} and let $G(w_1)$ be the right-hand side of \eqref{ChapKolmOnWxEquation}. Recall that we have already established \eqref{AuxiliaryHeatKernelDefinitionDistribution}, as a consequence of Lemma \ref{DistributionHeatKernelPhat}. Therefore, by Proposition \ref{AeChapmanKolmogorov}, $F : W \to \mathbb{R}$ and $G : W \to \mathbb{R}$ agree $\nu$-almost everywhere. In order to show that $F$ and $G$ are the same function, it is enough to show that they are both continuous.

Let us prove that $F$ is continuous first. For all $m \in \mathbb{N}$, $F$ is locally constant (and therefore continuous) on $A^{w_0}_m$. The only point in $W$ that does not belong to any $A^{w_0}_m$ is $w_0$ itself. By the formula for $\phat$,
\begin{align*}
    F(w_0) &= p(s+t, x, x) + \frac{e^{-v(x) t}}{\mu(x)} \left( a(s+t, \infty, D_x) - 1 \right)\\
    &= p(s+t, x, x) + \frac{e^{-v(x) t}}{\mu(x)} \left( \lim_{m \to \infty} a(s+t, m, D_x) - 1 \right) &\mbox{(by the definition of $a(s+t, \infty, D_x)$)}\\
    &= \lim_{m \to \infty} \left( p(s+t, x, x) + \frac{e^{-v(x) t}}{\mu(x)} \left( a(s+t, m, D_x) - 1 \right) \right)\\
    &= \lim_{w \to w_0} F(w).
\end{align*}
Therefore, $F$ is continuous on all of $W$.

For all $w_1 \in W$,
\begin{align*}
    G(w_1) &= \int_{\Mhat} \phat(s, (x, w_0), z) \phat(t, z, (x, w_1)) \muhat(dz)\\
    &= \int_{M \setminus \{x\}} p(s, x, y) p(t, y, x) \mu(dy) + \mu(x) \int_W \phat(s, (x, w_0), (x, w)) \phat(t, (x, w), (x, w_1)) \nu(dw)\\
    &= \int_{M \setminus \{x\}} p(s, x, y) p(t, y, x) \mu(dy) + \mu(x) \int_W \lim_{w' \to w_1} \phat(s, (x, w_0), (x, w)) \phat(t, (x, w), (x, w')) \nu(dw)\\
    & \qquad\mbox{(by the continuity of $\phat(t, (x, w_0), (x, \cdot))$)}\\
    &= \int_{M \setminus \{x\}} p(s, x, y) p(t, y, x) \mu(dy) +  \lim_{w' \to w_1} \mu(x) \int_W \phat(s, (x, w_0), (x, w)) \phat(t, (x, w), (x, w')) \nu(dw)\\
    &\qquad\mbox{(by the Dominated convergence theorem)}\\
    &= \lim_{w' \to w_1} G(w')
\end{align*}
so $G$ is continuous too.
\end{proof}

At last, we have proved Proposition \ref{PhatIsTheHeatKernelProp}.

\begin{proof}[Proof of Proposition \ref{PhatIsTheHeatKernelProp}]
By Lemma \ref{DistributionHeatKernelPhat}, $\phat$ satisfies \eqref{AuxiliaryHeatKernelDefinitionDistribution}. Since the formula for $\phat$ is symmetric in $z$ and $z'$, $\phat$ satisfies \eqref{AuxiliaryHeatKernelDefinitionSymmetry}. By Lemma \ref{PhatChapmamKolmogorov}, $\phat$ satisfies \eqref{AuxiliaryHeatKernelDefinitionChapmanKolmogorov}.
\end{proof}

It would be understandable for the reader to ask where in Lemma \ref{PhatChapmamKolmogorov} we used the $a(t, \infty, D_x)$ that appears in the formula for $\phat(t, (x, w_0), (x, w_0))$. The $a(t, \infty, D_x)$ appears in the formula for $\phat(t, (x, w_0), (x, w_0))$ because it is the value that makes $\phat(t, (x, w_0), (x, \cdot))$ continuous at $w_0$. In the proof of part (c) of Lemma \ref{PhatChapmamKolmogorov}, we use the fact that $\phat(t, (x, w_0), (x, \cdot))$ is continuous.

\section{Comparing the original heat kernel and auxiliary heat kernel} \label{HeatKernelSection2}

The purpose of this section is to prove \eqref{UhkBackwards}-\eqref{LhkBackwards} and \eqref{NdlBackwards}-\eqref{UhkdForwards}, the implications in Proposition \ref{MasterProp} involving heat kernels. Each of these is proved by using Proposition \ref{PhatIsTheHeatKernelProp} to compare the heat kernels of $\{X_t\}$ and $\{\Xhat_t\}$.

Note that if $\{\Xhat_t\}$ contains a heat kernel, it is easy to see from \eqref{couple} that $\{X_t\}$ contains a heat kernel too, in which case Proposition \ref{PhatIsTheHeatKernelProp} tells us that the two heat kernels are related by \eqref{PhatFormula}. Therefore, there are only two possibilities:
\begin{itemize}
    \item Both processes $\{X_t\}$ and $\{\Xhat_t\}$ have heat kernels, and they are related by \eqref{PhatFormula}.
    \item Neither $\{X_t\}$ nor $\{\Xhat_t\}$ has a heat kernel.
\end{itemize}
Thus, whenever a heat kernel-related condition (say $\HK$, $\NDL$, or $\UHKD$) holds on \textit{either} the original space or the auxiliary space, we know that \textit{both} spaces have heat kernels, and they are related by \eqref{PhatFormula}.

\subsection*{\texorpdfstring{$\UHK$ and $\LHK$: proof of \eqref{UhkBackwards}-\eqref{LhkBackwards}}{TEXT}}

In this subsection, we show that $\UHK$ (resp. $\LHK$) for the auxiliary space implies $\UHK$ (resp. $\LHK$) for the original space.
Recall the function $q(t, x, y)$ defined in \eqref{QtxyFormula}, which $p(t, x,y)$ is on the order of if $\HK$ holds. The function defined the same way on the auxiliary space is
\begin{equation*}
    \qhat(t, z_0, z) = \frac{1}{\Vhat(z_0, \phi^{-1}(t))} \wedge \frac{t}{\Vhat(z_0, \dhat(z, z_0)) \phi(\dhat(z,z_0))} \qquad\mbox{for $z_0, z \in \Mhat$, $t>0$}
\end{equation*}
so we must show that if $\phat$ is at most (resp. at least) on the order of $\qhat$, then $p$ is at most (resp. at least) on the order of $q$.

Our strategy is the following: use Proposition \ref{PhatIsTheHeatKernelProp} to compare $p$ to $\phat$, use the hypothesis ($\UHK$ or $\LHK$ for the auxiliary space) to compare $\phat$ to $\qhat$, and then compare $\qhat$ to $q$. If we are able to do this, then we successfully compare $p$ to $q$. The only missing ingredient is a way to compare $\qhat$ to $q$. The following lemma shows how this comparison is made.

\begin{lemma} \label{QhatQComparison}
Suppose $(M, d, \mu, \mathcal{E}, \mathcal{F})$ satisfies Assumption \ref{BasicAssumptions}, $\phi$ is of regular growth, $(\mathcal{E}, \mathcal{F})$ admits a jump kernel, and $M=M_A \cup M_C$.

(a) If $x$ and $y$ are distinct elements of $M$, $z_1 \in W_x$, $z_2 \in W_y$, and $t>0$, then
\begin{equation*}
    \qhat(t, z_1, z_2) = q(t, x, y).
\end{equation*}

(b) If $x \in M_A$, $z_0 \in W_x$, $t>0$, and $t \geq \phi(D_x)$, then
\begin{equation*} \label{WxIntStatement1}
    \int_{W_x} \qhat(t, z_0, z) \muhat(dz) = \mu(x) q(t, x, x).
\end{equation*}

(c) If $x \in M_A$, $z_0 \in W_x$, $t>0$, and $t \leq \phi(D_x)$, then
\begin{equation} \label{WxIntStatement2}
    \int_{W_x} \qhat(t, z_0, z) \muhat(dz) \in [1, 3].
\end{equation}

\end{lemma}

Before we prove Lemma \ref{QhatQComparison}, let us establish a more explicit formula for $q$ and $\qhat$. Observe that $V(x, \phi^{-1}(t))$ is strictly increasing in $t$, and $V(x, d(x, y))\phi(d(x, y)) / t$ is strictly decreasing in $t$. The two are equal for $t = \phi(d(x, y))$. Therefore, \eqref{QtxyFormula} can be rewritten as
\begin{equation} \label{whichMin}
    q(t, x, y) = \left\{ \begin{matrix}
        \frac{1}{V(x, \phi^{-1}(t))} &:& \mbox{if $t \geq \phi(d(x, y))$}\\
        \\
        \frac{t}{V(x, d(x, y)) \phi(d(x, y))} &:& \mbox{if $t \leq \phi(d(x, y))$}.
    \end{matrix}\right.
\end{equation}
By applying \eqref{whichMin} to the auxiliary space,
\begin{equation} \label{QhatWhichMin}
    \qhat(t, z_0, z) = \left\{ \begin{matrix}
        \frac{1}{\Vhat(z_0, \phi^{-1}(t))} &:& \mbox{if $t \geq \phi(\dhat(z, z_0))$}\\
        \\
        \frac{t}{\Vhat(z_0, \dhat(z, z_0)) \phi(\dhat(z, z_0))} &:& \mbox{if $t \leq \phi(\dhat(z, z_0))$}.
    \end{matrix}\right.
\end{equation}

\begin{proof}[Proof of Lemma \ref{QhatQComparison}]

(a) Let $x$ and $y$ be distinct elements of $M$, $z_1 \in W_x$, and $z_2 \in W_y$.
If $t \geq \phi(d(x, y)) =  \phi(\dhat(z_1, z_2))$, then by \eqref{whichMin} and \eqref{QhatWhichMin},
\begin{dmath*}
    \qhat(t, z_1, z_2) = \frac{1}{\Vhat(z_1, \phi^{-1}(t))} = \frac{1}{V(x, \phi^{-1}(t)} \qquad\mbox{(since $\phi^{-1}(t) \geq d(x, y) \geq D_x$)}
    = q(t, x, y).
\end{dmath*}
If $t \leq \phi(d(x, y)) = \phi(\dhat(z_1, z_2))$, then by \eqref{whichMin} and \eqref{QhatWhichMin},
\begin{dmath*}
    \qhat(t, z_1, z_2) = \frac{t}{\Vhat(z_1, \dhat(z_1, z_2)) \phi(\dhat(z_1, z_2))}
    = \frac{t}{\Vhat(z_1, d(x, y)) \phi(d(x, y))}
    = \frac{t}{V(x, d(x, y)) \phi(d(x, y))} \qquad \mbox{(since $d(x, y) \geq D_x$)}
    = q(t, x, y).
\end{dmath*}

(b) Fix $x \in M_A$, $z_0 \in W_x$, and a positive $t \geq D_x$. 
For all $z \in W_x$, \eqref{whichMin} and \eqref{QhatWhichMin} give
\begin{equation*}
    \qhat(t, z_0, z) = \frac{1}{\Vhat(z, \phi^{-1}(t))} = \frac{1}{V(x, \phi^{-1}(t))} = q(t, x, x)
\end{equation*}
so
\begin{equation*}
    \int_{W_x} \qhat(t, z_0, z) \muhat(dz) = \muhat(W_x) q(t, x, x) = \mu(x) q(t, x, x).
\end{equation*}

(c) Fix $x \in M_A$, $z_0 \in W_x$, and $t \in (0, D_x]$. 
This will be the hardest part of Lemma \ref{QhatQComparison} to prove. Essentially, we decompose the integral $\int_{W_x} \qhat(t, z_0, z) \muhat(dz)$ into a sum of integrals over the spheres $\{x\} \times A^{w_0}_m$ (where $z_0 = (x, w_0)$). Equation \eqref{QhatWhichMin} gives a different formula for $\qhat(t, (x, w_0), \cdot)$ on $A^{w_0}_m$ for large $m$ and small $m$, so our sum gets broken into two sums, which we evaluate separately.

Let $m_0$ be the unique non-negative integer such that $\phi^{-1}(t) \in \left( d^{D_x}_{m_0+1}, d^{D_x}_{m_0} \right]$, or
\begin{equation}
    2^{-(m_0+1)}\phi(D_x) < t \leq 2^{-m_0} \phi(D_x).
\end{equation}
Let $w_0$ be the element of $W$ such that $z_0=(x, w_0)$. For all $m \in \mathbb{N}$, let $w_m$ be a representative of $A^{w_0}_m$ (recall how the sphere $A^{w_0}_m$ is defined in \eqref{AwmDef}) and let $z_m=(x, w_m)$. If $m > m_0$, then $t \geq \phi(\dhat(z_0, z_m))$, so
\begin{dmath} \label{WxIntStatementA}
    \qhat(t, z_0, z_m) = \frac{1}{\Vhat(z_0, \phi^{-1}(t))} = \frac{1}{\Vhat(z_0, d^{D_x}_{m_0})} \qquad\mbox{(by our choice of $m_0$)}
    = {\frac{1}{\mu(x) \cdot 2^{-m_0}} = \frac{2^{m_0}}{\mu(x)}}.
\end{dmath}
If $m \leq m_0$, then $t \leq \phi(\dhat(z_0, z_m))$, so
\begin{dmath} \label{WxIntStatementB}
    \qhat(t, z_0, z_m) = \frac{t}{\Vhat(z_0, \dhat(z_0, z_m)) \phi(\dhat(z_0, z_m))}
    = \frac{t}{\Vhat(z_0, d^{D_x}_m) \phi(d^{D_x}_m)}
    = \frac{t}{2^{-m} \mu(x) \cdot 2^{-m} \phi(D_x)} = \frac{4^m t}{\mu(x) \phi(D_x)}.
\end{dmath}
By \eqref{WxIntStatementA} and \eqref{WxIntStatementB},
\begin{dmath} \label{WxIntStatementC}
    \int_{W_x} \qhat(t, z_0, z) \muhat(dz) = \sum_{m=1}^\infty \mu(x) \nu(A^{w_0}_m) \qhat(t, z_0, z_m) 
    = \sum_{m=1}^{m_0} 2^{-m} \cdot \frac{4^m t}{\phi(D_x)} + \sum_{m=m_0+1}^\infty 2^{-m} \cdot 2^{m_0}
    = \frac{t}{\phi(D_x)} \sum_{m=1}^{m_0} 2^m + 2^{m_0} \sum_{m=m_0 + 1} 2^{-m}
    = \frac{t}{\phi(D_x)}(2^{m_0 + 1}-2) + 1
    = \frac{2t}{\phi(D_x)}(2^{m_0}-1) + 1.
\end{dmath}
Obviously, \eqref{WxIntStatementC} implies
\begin{equation*}
    \int_{W_x} \qhat(t, z_0, z) \muhat(dz) \geq 1.
\end{equation*}
Recall that $m_0$ was chosen so that $t / \phi(D_x) \leq 2^{-m_0}$. Thus, \eqref{WxIntStatementC} also implies
\begin{equation*}
    \int_{W_x} \qhat(t, z_0, z)\muhat(dz) \leq 2 \cdot 2^{-m_0} \cdot (2^{m_0} - 1) + 1 \leq 2 + 1 = 3.
\end{equation*}

\end{proof}

Now that we have Lemma \ref{QhatQComparison}, we have all the necessary ingredients to prove \eqref{UhkBackwards}-\eqref{LhkBackwards}. We compare $p$ to $\phat$ using Proposition \ref{PhatIsTheHeatKernelProp}, compare $\phat$ to $\qhat$ using the hypothesis that the condition holds on the auxiliary space, and compare $\qhat$ to $q$ using Lemma \ref{QhatQComparison}.
The proof that $\UHK$ for the auxiliary space implies $\UHK$ for the original space is as follows.

\begin{proof}[Proof of Proposition \ref{MasterProp}, implication \eqref{UhkBackwards}]
Suppose $\UHK$ holds for $(\Mhat, \dhat, \muhat, \Ehat, \Fhat)$; there exists a $C>0$ such that $\phat(t, z_0, z) \leq C\qhat(t, z_0, z)$ for all $t, z_0, z$. Fix $x, y \in M_0$ and $t>0$.

Suppose $x \neq y$. Let $z_0$ be a representative of $W_x$, and let $z$ be a representative of $W_y$. Then
\begin{align*}
    p(t, x, y) &= \phat(t, z_0, z) & \mbox{(by Proposition \ref{PhatIsTheHeatKernelProp})}\\
    &\leq C\qhat(t, z_0, z)\\
    &= Cq(t, x, y) &\mbox{(by Lemma \ref{QhatQComparison}(a))}.
\end{align*}
If on the other hand, $x=y \in M_C \setminus \mathcal{N}$, then
\begin{align*}
    p(t, x, x) &= \phat(t, x, x) &\mbox{(by Proposition \ref{PhatIsTheHeatKernelProp})}\\
    &\leq C\qhat(t, x, x)\\
    &= \frac{C}{\Vhat(x, \phi^{-1}(t))} &\mbox{(by \eqref{QhatWhichMin})}\\
    &= \frac{C}{V(x, \phi^{-1}(t))} &\mbox{(since $\phi^{-1}(t) > 0 = D_x$)}\\
    &= Cq(t, x, x) &\mbox{(by \eqref{whichMin})}.
\end{align*}
Finally, if $x=y \in M_A$,
\begin{dmath} \label{PtxxIntermsOfQhat}
    p(t, x, x) = {\frac{1}{\mu(x)}\mathbb{P}^x(X_t=x)}
    ={\frac{1}{\mu(x)} \mathbb{P}^{z_0}(\Xhat_t \in W_x)}
    = \frac{1}{\mu(x)} \int_{W_x} \phat(t, z_0, z) \muhat(dz)
    \leq \frac{C}{\mu(x)} \int_{W_x} \qhat(t, z_0, z) \muhat(dz).
\end{dmath}

If $t \geq \phi(D_x)$, then by \eqref{PtxxIntermsOfQhat} and Lemma \ref{QhatQComparison}(b),
\begin{equation*}
    p(t, x, x) \leq \frac{C}{\mu(x)} \cdot \mu(x) q(t, x, x) = Cq(t, x, x).
\end{equation*}
Suppose $t \leq \phi(D_x)$. By \eqref{onDiagonalQt},
\begin{equation}\label{HkLemmaLastCaseQ}
    q(t, x, x) = \frac{1}{V(x, \phi^{-1}(t))} = \frac{1}{\mu(x)}.
\end{equation}
By \eqref{PtxxIntermsOfQhat} and Lemma \ref{QhatQComparison}(c),
\begin{equation} \label{HkLemmaLastCase3C}
    p(t, x, x) \leq \frac{3C}{\mu(x)}.
\end{equation}
By \eqref{HkLemmaLastCaseQ} and \eqref{HkLemmaLastCase3C},
\begin{equation*}
    p(t, x, x) \leq 3C q(t, x,x).
\end{equation*}
Therefore, $\UHK$ holds for $(M, d, \mu, \mathcal{E}, \mathcal{F})$.
\end{proof}

The proof of \eqref{LhkBackwards} ($\LHK$ for the auxiliary space implies $\LHK$ for the original space) is almost exactly the same. Simply change the direction of the inequalities and use the lower bound instead of the upper bound from Lemma \ref{QhatQComparison}(c).

\subsection*{\texorpdfstring{$\NDL$: proof of \eqref{NdlBackwards}}{TEXT}}

In this subsection, we prove that if $\VD$ holds, then $\NDL$ for the auxiliary space implies $\NDL$ for the original space. We remind the reader that $\NDL$ says that
\begin{equation} \label{NdlOriginal}
    p^B(t, x, y) \geq \frac{c_1}{V(x_0, \phi^{-1}(t))} \quad\mbox{for all $B=B(x_0, r)$, $0<t\leq \phi(\varepsilon r)$, and $\,x, y \in B(x_0, \varepsilon\phi^{-1}(t)) \cap M_0$}
\end{equation}
where $c_1>0$ and $\varepsilon \in (0, 1)$ are constants, and $p^B$ is the Dirichlet heat kernel. For the auxiliary space, the statement of $\NDL$ becomes
\begin{equation}\label{NdlAuxiliary}
    \phat^{\Bhat}(t, z_1, z_2) \geq \frac{c_1}{\Vhat(z_0, \phi^{-1}(t))} \quad\mbox{for all $\Bhat=\Bhat(z_0, r)$, $0<t\leq \phi(\varepsilon r)$, and $\,z_1, z_2 \in \Bhat(z_0, \varepsilon\phi^{-1}(t)) \cap \Mohat$}.
\end{equation}
(The constants $c_1$ and $\varepsilon$ may not be the same in \eqref{NdlOriginal} and \eqref{NdlAuxiliary}.

We must show that \eqref{NdlAuxiliary} implies \eqref{NdlOriginal}.
In our proof, we consider three cases:
\begin{itemize}
    \item $\varepsilon \phi^{-1}(t) < D_{x_0}$.
    \item $\varepsilon \phi^{-1}(t) \geq D_{x_0}$, and either $x \neq y$ or $x=y \in M_C \setminus \mathcal{N}$.
    \item $\varepsilon \phi^{-1}(t) \geq D_{x_0}$ and $x=y \in M_A$.
\end{itemize}
In all three cases, we use \eqref{couple} to compare $p^B$ to some quantity involving the near-diagonal Dirichlet heat kernel on the auxiliary space, then use the hypothesis that $\NDL$ holds on the auxiliary space to put a lower bound on this quantity. In the first case, the lower bound we obtain involves the ratio $\Vhat(z_0, \phi^{-t}(t)) / \Vhat(z_0, \varepsilon \phi^{-t}(t))$, which we handle using $\VD$. (This is where the assumption of $\VD$ is needed.)

\begin{proof}[Proof of Proposition \ref{MasterProp}, implication \eqref{NdlBackwards}]
Assume $\VD$ holds for the original space and $\NDL$ holds for the auxiliary space. By $\NDL$, there exist constants $c_1>0$ and $\varepsilon \in (0, 1)$ for which \eqref{NdlAuxiliary} holds.

Fix $x_0 \in M$ and $r>0$. Let $B=B(x_0, r)$. Suppose $0<t \leq \phi(\varepsilon r)$ and $x, y \in B(x_0, \varepsilon\phi^{-1}(t)) \cap M_0$. We will show that $p^B(t, x, y)$ is at least some positive constant times $\frac{1}{V(x_0, \phi^{-1}(t))}$. We will treat the three following cases separately.

(a) $\varepsilon \phi^{-1}(t) < D_{x_0}$.

(b) $\varepsilon \phi^{-1}(t) \geq D_{x_0}$, and either $x \neq y$ or $x=y \in M_C \setminus \mathcal{N}$.

(c) $\varepsilon \phi^{-1}(t) \geq D_{x_0}$ and $x=y \in M_A$.

Note that these three cases cover all possibilities. The proofs for each case are as follows:

(a) In this case, $x_0 \in M_A$ (since $D_{x_0} > \varepsilon \phi^{-1}(t) > 0$). Also, $x$ and $y$ must both be $x_0$, since there are no other points in $B(x_0, \varepsilon\phi^{-1}(t))$.
By the definition of Dirichlet heat kernel,
\begin{equation} \label{NdlProofIntermed1}
    p^B(t, x, y) = p^B(t, x_0, x_0) = \frac{1}{\mu(x_0)} \mathbb{P}^{x_0} \left( \mbox{$X_t=x_0$, and $X_s \in B$ for all $0 \leq s \leq t$} \right).
\end{equation}
Since $\VD$ holds for the original space, $\VD$ also holds for the auxiliary space by Proposition \ref{volumeGrowthProp}(b). By applying \eqref{VdAlternate} to the auxiliary space, there exist constants $C>0$ and $d>0$ such that
\begin{equation}\label{VdAlternateAuxiliary}
    \frac{\Vhat(z, R)}{\Vhat(z, r)} \leq C \left(\frac{R}{r}\right)^d \qquad\mbox{for all $z \in \Mhat$, $R \geq r>0$}.
\end{equation}
If $z_0$ is a representative of $W_{x_0}$,
\begin{align*}
    p^B(t, x, y) &= \frac{1}{\mu(x_0)} \mathbb{P}^{x_0} \left( \mbox{$X_t=x_0$, and $X_s \in B$ for all $0 \leq s \leq t$} \right) \qquad\mbox{(by \eqref{NdlProofIntermed1})}\\
    &= \frac{1}{\mu(x_0)} \mathbb{P}^{z_0}\left(\mbox{$\Xhat_t \in W_{x_0}$, and $\Xhat_s \in \pi^{-1}(B)$ for all $0 \leq s \leq t$}\right) \qquad\mbox{(by \eqref{couple})}\\
    &= \frac{1}{\mu(x_0)} \int_{W_x} \phat^{\left(\pi^{-1}(B)\right)}(t, z_0, z) \muhat(dz)\\
    &\geq \frac{1}{\mu(x_0)} \int_{\Bhat(z_0, \varepsilon\phi^{-1}(t))} \frac{c_1}{\Vhat(z_0, \phi^{-1}(t))} \muhat{dz} \qquad\mbox{(since $\NDL$ holds on $(\Mhat, \dhat, \muhat, \Ehat, \Fhat)$)}\\
    &= \frac{c_1}{\mu(x_0)} \cdot \frac{\Vhat(z_0, \varepsilon\phi^{-1}(t))}{\Vhat(z_0, \phi^{-1}(t))}\\
    &\geq \frac{c_2}{\mu(x_0)} \varepsilon^d \qquad\mbox{(by \eqref{VdAlternateAuxiliary}, where $c_2$ is another positive constant)}\\
    &\geq \frac{c_2 \varepsilon^d}{V(x_0, \phi^{-1}(t))} \qquad\mbox{(since $\mu(x_0) \leq V(x_0, \phi^{-1}(t)) $)}.
\end{align*}

(b) In this case, since $\varepsilon \phi^{-1}(t) \geq D_{x_0}$, we have
\begin{equation} \label{NdlBothBalsBig}
    \Bhat(z_0, \varepsilon \phi^{-1}(t)) = \pi^{-1}(B(x_0, \varepsilon\phi^{-1}(t)) \qquad\mbox{and}\qquad \Bhat(z_0, \phi^{-1}(t)) = \pi^{-1}(B)).
\end{equation}
Let $z_1$ be a representative of $W_x$ and let $z_2$ be a representative of $W_y$. Then $z_1, z_2 \in \Bhat(z_0, \varepsilon\phi^{-1}(t)) \cap \Mohat$ by \eqref{NdlBothBalsBig}.
\begin{dmath*}
    p^B(t, x, y) = \phat^{\left( \pi^{-1}(B) \right)}(t, z_1, z_2) \qquad\mbox{(by \eqref{couple}, since $x \neq y$ or $x=y \in M_C$)}
    \geq \frac{c_1}{\Vhat(z_0, \phi^{-1}(t))} \qquad\mbox{(since $z_1, z_2 \in \Bhat(z_0, \varepsilon\phi^{-1}(t)) \cap \Mohat$ and $\NDL$ holds on $(\Mhat, \dhat, \muhat, \Ehat, \Fhat)$)}
    = \frac{c_1}{V(x_0, \phi^{-1}(t))} \qquad\mbox{(by \eqref{NdlBothBalsBig})}.
\end{dmath*}

(c) In this case, we still have $\varepsilon \phi^{-1}(t)\geq D_{x_0}$, so we once again have \eqref{NdlBothBalsBig}.
Note that all of $W_x$ belongs to $\Bhat(z_0, \varepsilon\phi^{-1}(t))$, so
\begin{align*}
    p^B(t, x, x) &= \frac{1}{\mu(x)} \int_{W_x} \phat^{\left( \pi^{-1}(B) \right)}(t, z_0, z) \muhat(dz) \qquad\mbox{(by \eqref{couple})}\\
    &\geq \frac{c_1}{\mu(x)} \frac{\mu(x)}{\Vhat(z_0, \phi^{-1}(t))}= \frac{c_1}{\Vhat(z_0, \phi^{-1}(t))} = \frac{c_1}{V(x_0, \phi^{-1}(t))}.
\end{align*}
\end{proof}

\subsection*{\texorpdfstring{$\UHKD$: proof of \eqref{UhkdForwards}}{TEXT}}

In this subsection we prove that if $\VD$ holds, then $\UHKD$ for the original space implies $\UHKD$ for the auxiliary space.

In other words, we assume an upper bound on the on-diagonal heat kernel $p(t, x, x)$, and prove an upper bound on the on-diagonal heat kernel $\phat(t, z, z)$. Recall from \eqref{PhatFormula} that for $x \in M_C$,
\begin{equation*} \label{PhatOndiagMc}
    \phat(t, x, x) = p(t, x, x)
\end{equation*}
and for $x \in M_A$, $w \in W$,
\begin{equation}\label{PhatOndiagMa}
    \phat(t, (x, w), (x, w)) = p(t, x, x) + \frac{e^{-v(x) t}}{\mu(x)} \left( a(t, \infty, D_x) - 1 \right)
\end{equation}
where $a(t, \infty, D_x)$ is as defined in Proposition \ref{PhatIsTheHeatKernelProp}.
Since we are already assuming an upper bound on $p(t, x, x)$, the difficulty lies in proving an upper bound on the second term in the right-hand-side of \eqref{PhatOndiagMa}.

The first step is to approximate $a(t, \infty, D_x)-1$ from above by a nice function of $t$ and $D_x$. Recall that $a(t, \infty, D_x)$ is defined as an infinite sum. In the following lemma, we replace this sum with an integral (that overestimates each term), and obtain a nice upper bound on $a(t, \infty, D)-1$ for a general $t>0$ and $D>0$.

\begin{lemma} \label{AtdxUpperbound}
For all $t>0$ and $D>0$, if we let $s:=t/\phi(D)$, then
\begin{equation*}
    a(t, \infty, D) -1 \leq \frac{e^{-s}}{s}.
\end{equation*}
\end{lemma}

\begin{proof}
Recall that $a(t, \infty, D) := 1 + \sum_{j=0}^\infty 2^{j-1} \exp \left( -(3 \cdot 2^j - 2) \frac{t}{\phi(D)} \right)$.
Let $h : (0, \infty) \to (0, \infty)$ be the function defined by
\begin{equation*}
    h(s) := \sum_{j=1}^\infty 2^{j-1} \exp \left( -(3 \cdot 2^j - 2) s \right)
\end{equation*}
so that $a(t, \infty, D) - 1 = h(s)$.
If $x$ belongs to the interval $[j-1, j]$, then $2^x \geq 2^{j-1}$ and $\exp(-(3 \cdot 2^x - 2)s) \geq \exp(-(3 \cdot 2^j - 2)s)$. Therefore,
\begin{dmath*}
    h(s) = \sum_{j=1}^\infty 2^{j-1} \exp(-(3 \cdot 2^j -2)s) \leq \sum_{j=1}^\infty \int_{j-1}^j 2^x \exp(-(3 \cdot 2^x -2)s) \, dx = \int_0^\infty 2^x \exp(-(3 \cdot 2^x -2)s) \, dx
    = \frac{1}{(3 \log 2)s} \int_s^\infty e^{-u} du \qquad\mbox{(by the substitution $u = (3 \cdot 2^x - 2) s$)}
    = \frac{e^{-s}}{s(3 \log 2)}
    \leq \frac{e^{-s}}{s} \qquad\mbox{(since $3 \log 2 \approx 2.079 > 1$)}.
\end{dmath*}
\end{proof}

By applying Lemma \ref{AtdxUpperbound}, equation \eqref{PhatOndiagMa} becomes
\begin{equation}\label{PhatOndiagMaApprox}
    \phat(t, (x, w), (x, w)) \leq p(t, x, x) + \frac{e^{-v(x)t}}{\mu(x)} \cdot \frac{e^{-s}}{s} \qquad\mbox{where} \, s = \frac{t}{\phi(D_x)}.
\end{equation}
We would like to show that the second term in the right-hand-side of \eqref{PhatOndiagMaApprox} is at most some constant times $1/\Vhat((x, w), \phi^{-1}(t))$. This is equivalent to showing that $e^{-v(x) t} \cdot \frac{\Vhat((x, w), \phi^{-1}(t))}{\mu(x)} \cdot \frac{e^{-s}}{s}$ is at most some constant. In the following lemma, we show how to put an upper bound on the quantity $\frac{\Vhat((x, w), \phi^{-1}(t))}{\mu(x)} \cdot \frac{e^{-s}}{s}$.
The argument is very different for $t \geq \phi(D_x)$ than for $t \leq \phi(D_x)$. For $t \geq \phi(D_x)$, we observe that $\mu(x) = \Vhat((x, w), D_x)$ and use $\VD$ and the regular growth of $\phi$ to control the ratio $\frac{\Vhat((x, w), \phi^{-1}(t))}{\mu(x)}$. For $t \leq \phi(D_x)$, we see that $\Vhat((x, w), \phi^{-1}(t)) = 2^{-m} \mu(x)$ for some $m$, and so $\frac{\Vhat((x, w), \phi^{-1}(t))}{\mu(x)} = 2^{-m}$.

\begin{lemma} \label{UhkdOndiagMaLemma}
Suppose $(M, d, \mu, \mathcal{E}, \mathcal{F})$ satisfies Assumption \ref{BasicAssumptions} and $\VD$, $\phi$ is of regular growth, $(\mathcal{E}, \mathcal{F})$ admits a jump kernel, and $M=M_A \cup M_C$.
There exists a $C>0$ such that
\begin{equation*}
    \frac{\Vhat(z_0, \phi^{-1}(t))}{\mu(x)} \cdot \frac{e^{-s}}{s} \leq C \qquad\mbox{for all $x \in M_A$, $z_0 \in W_x$, and $t>0$,} \qquad\mbox{where} \quad s := \frac{t}{\phi(D_x)}.
\end{equation*}
\end{lemma}

\begin{proof}
Fix $x$, $z_0$, and $t$. Let $s:=t/\phi(D_x)$. We will treat the following cases separately:

(a) $t \geq \phi(D_x)$.

(b) $t \leq \phi(D_x)$.

The proofs for cases (a) and (b) are:

(a) Since $\VD$ holds, let $d$ be the exponent from \eqref{VdAlternate}.
Since $\phi$ is of regular growth, let $\beta_1$ be the exponent from the lower bound in \eqref{regularGrowth}.
We can assume that $d \geq \beta_1$ because replacing $d$ with a larger value (or $\beta_1$ with a smaller value) does not make \eqref{VdAlternate} (or \eqref{regularGrowth}) cease to be true. By \eqref{VhatFormula}, $\VD$, and the regular growth of $\phi$,
\begin{equation*}
    \frac{\Vhat(z_0, \phi^{-1}(t))}{\mu(x)} = \frac{V(x, \phi^{-1}(t))}{V(x, D_x)} \leq C_1 \left(\frac{\phi^{-1}(t)}{D_x}\right)^d \leq C_2 \left( \frac{t}{\phi(D_x)} \right)^{d/\beta_1} = C_2 s^{d/\beta_1}.
\end{equation*}
Thus,
\begin{equation} \label{UhkdRatioBigT}
    \frac{\Vhat(z_0, \phi^{-1}(t))}{\mu(x)} \cdot \frac{e^{-s}}{s} \leq C_2 s^{\left(d/\beta_1 - 1\right)} e^{-s}.
\end{equation}
Let $\alpha := \frac{d}{\beta_1} - 1$. Recall that we assumed $d \geq \beta_1$, which means that $\alpha \geq 0$. Let $g$ be the function $g(u) = u^\alpha e^{-u}$.
Let $G := \sup_{u \in [1, \infty)} g(u)$. Note that $G$ is finite, since $g$ is continuous on $[1, \infty)$ and $\lim_{u \to \infty} g(u) = 0$.
By \eqref{UhkdRatioBigT},
\begin{equation*}
    \frac{\Vhat(z_0, \phi^{-1}(t))}{\mu(x)} \cdot \frac{e^{-s}}{s} \leq C_2 g(s) \leq C_2 G.
\end{equation*}

(b) Let $m$ be the integer such that $2^{-(m+1)} \phi(D_x) < t \leq 2^{-m} \phi(D_x)$.
Then
\begin{equation} \label{UhkdSmalltCase1}
    \frac{\Vhat(z_0, \phi^{-1}(t))}{\Vhat(z_0, D_x)} = \frac{2^{-m} \mu(x)}{\mu(x)} = 2^{-m}
\end{equation}
and $s \in \left(2^{-(m+1)}, 2^{-m} \right]$ so
\begin{equation} \label{UhkdSmalltCase2}
    \frac{e^{-s}}{s} \leq \frac{1}{s} \leq 2^{m+1}.
\end{equation}
By \eqref{UhkdSmalltCase1} and \eqref{UhkdSmalltCase2},
\begin{equation*}
    \frac{\Vhat(z_0, \phi^{-1}(t))}{\Vhat(z_0, D_x)} \cdot \frac{e^{-s}}{s} \leq 2^{-m} \cdot 2^{m+1} = 2.
\end{equation*}
\end{proof}

Now we prove the desired result, that $\VD$ and $\UHKD$ for the original space implies $\UHKD$ for the auxiliary space. This entails putting an upper bound on the on-diagonal heat kernel $\phat(t, z_0, z_0)$. There are two cases to consider: $z_0 \in M_C$ and $z_0 \in W_x$ for some $x \in M_A$. If $z_0 \in M_C$, the upper bound follows directly from the hypothesis that $\UHKD$ holds on the original space. If $z_0 \in W_x$ for some $x \in M_A$, we use Lemma \ref{AtdxUpperbound} to show put an upper bound on $\phat(t, z_0, z_0)$ of the form \eqref{PhatOndiagMaApprox}. Then we use Lemma \ref{UhkdOndiagMaLemma} to put an upper bound on the right-hand-side of \eqref{PhatOndiagMaApprox}.

\begin{proof}[Proof of Proposition \ref{MasterProp}, implication \eqref{UhkdForwards}]
By the assumption that $\UHKD$ holds on the original space, there exists a $C_2>0$ such that
\begin{equation} \label{UhkdHypothesis}
    p(t, x, x) \leq \frac{C_2}{V(x, \phi^{-1}(t))} \qquad\mbox{for all $x \in M$, $t>0$}.
\end{equation}

Fix $z_0 \in M$ and $t>0$. We would like to show that $\phat(t, z_0, z_0)$ is at most some constant times $1/\Vhat(z_0, \phi^{-1}(t))$. If $z_0 = x$ for some $x \in M_C$, then this is easy:
\begin{equation*}
    \phat(t, z_0, z_0) = p(t, x, x) \leq \frac{C_2}{V(x, \phi^{-1}(t))}= \frac{C_2}{\Vhat(z_0, \phi^{-1}(t))}.
\end{equation*}
Therefore, we can assume that $z_0 \in W_x$ for some $x \in M_A$

Instead of showing that $\phat(t, z_0, z_0)$ is at most a constant times $1/\Vhat(z_0, \phi^{-1}(t))$, it will be slightly easier to show the equivalent claim that $\phat(t, z_0, z_0) \Vhat(z_0, \phi^{-1}(t))$ is at most some constant. Let $s = t/\phi(D_x)$. By \eqref{PhatOndiagMaApprox},
\begin{equation} \label{PhatTimesVhat}
    \phat(t, z_0, z_0) \Vhat(z_0, \phi^{-1}(t)) \leq p(t, x, x) \Vhat(z_0, \phi^{-1}(t)) + e^{-v(x) t} \frac{\Vhat(z_0, \phi^{-1}(t))}{\mu(x)} \cdot \frac{e^{-s}}{s}.
\end{equation}
By \eqref{VhatFormula} and \eqref{UhkdHypothesis},
\begin{equation} \label{PhatTimesVhatFirstterm}
    p(t, x, x) \Vhat(z_0, \phi^{-1}(t)) \,\leq\, p(t, x, x) V(x, \phi^{-1}(t)) \,\leq \,C_2.
\end{equation}
By Lemma \ref{UhkdOndiagMaLemma},
\begin{equation} \label{PhatTimesVhatSecondterm}
    \frac{\Vhat(z_0, \phi^{-1}(t))}{\mu(x)} \cdot \frac{e^{-s}}{s} \leq C_3.
\end{equation}
By \eqref{PhatTimesVhat}, \eqref{PhatTimesVhatFirstterm}, and \eqref{PhatTimesVhatSecondterm},
\begin{equation*}
     \phat(t, z_0, z_0) \Vhat(z_0, \phi^{-1}(t)) \leq C_2 + e^{-v(x) t} C_3 \leq C_2 + C_3.
\end{equation*}
\end{proof}

\section{Cut-off Sobolev and Poincar\'e inequalties} \label{CsPoincSection}

In this section, we prove \eqref{CsjForwards} and \eqref{PoincareForward}: we show that if either $\CSJ$ or $\PI$ holds on the original space, then it also holds on the auxiliary space.

The Poincar\'e and cut-off Sobolev inequalities depend on $\mathcal{F}$, the domain of the regular Dirichlet form. The statements of $\CSJ$ and $\PI$ both contain the phrase ``for all $f \in \mathcal{F}$" (or ``for all $f \in \mathcal{F}_b$" where $\mathcal{F}_b := \mathcal{F} \cap L^\infty(M, \mu)$). (See Definitions \ref{CsjDefinition} and \ref{PoincareDefinition}.) When applied to the auxiliary space, these will involve $\Fhat$. Therefore, in order to prove \eqref{CsjForwards} and \eqref{PoincareForward}, it will be necessary to know which functions on $\Mhat$ are in $\Fhat$.

Throughout this section, let us continue to assume all of the conditions that are necessary for the construction of $(\Mhat, \dhat, \muhat, \Ehat, \Fhat)$. In other words, $(M, d, \mu, \mathcal{E}, \mathcal{F})$ satisfies the basic assumptions, $(\mathcal{E}, \mathcal{F})$ admits a jump kernel, and $M=M_A \cup M_C$.
For a function $f \in L^2(\Mhat, \mu)$, we define the following functions on $M$ that approximate the behavior of $f$ in different ways.

\begin{definition}

For all $f \in L^2(\Mhat, \mu)$, define $f_\mean : M \to \mathbb{R}$ by
\begin{equation}\label{fmeanDef}
    f_\mean(x) = \left\{ \begin{matrix}
        \frac{1}{\mu(x)} \int_{W_x} f \, d\mu &:& \mbox{if $x \in M_A$}\\
        \\
        f(x) &:& \mbox{if $x \in M_C$}.
    \end{matrix}
    \right.
\end{equation}
Similarly, define $f_\rms : M \to \mathbb{R}$ by
\begin{equation}\label{frmsDef}
    f_\rms(x) = \left\{ \begin{matrix}
        \left(\frac{1}{\mu(x)} \int_{W_x} f^2 \, d\mu\right)^{\sfrac12} &:& \mbox{if $x \in M_A$}\\
        \\
        |f(x)| &:& \mbox{if $x \in M_C$}.
    \end{matrix}
    \right.
\end{equation}
\end{definition}

In other words, for all $x$, $f_\mean(x)$ is the mean value of $f$ on $W_x$, and $f_\rms(x)$ is the root-mean-square of $|f|$ on $W_x$. Note that $f_\rms$ only takes non-negative values.

\begin{remark}\label{MaCountable}
We should verify that $f_\mean$ and $f_\rms$ are measurable, for $f \in L^2(\Mhat, \muhat)$. This is quickly verified by the following argument:

(a) The set $M_A$ is at most countable. If it was uncountable, then $(M, d)$ would not be separable, since there would not be a dense countable subset.

(b) Let $G$ be a measurable subset of $\mathbb{R}$. Clearly, $M_C$ is measurable (since its complement is countable). Since $f$ is a measurable function, this means that $\{x \in M_C : f(x) \in G\}$ is a measurable set. The pre-image $f_\mean^{-1}(G)$ only differs from $\{x \in M_C : f(x) \in G\}$ by a countable set (a subset of $M_A$), so the pre-image is itself a measurable set. The same argument, using $|f|$ instead of $f$, applies to $f_\rms$.

\end{remark}

It is clear from the definition of $f_\mean$ and $f_\rms$ that for all measurable $E \subseteq M$,
\begin{equation} \label{FmeanFrmsIntegralsAgree}
    \int_{\pi^{-1}(E)} f \, d\muhat = \int_E f_\mean \, d\mu \qquad\mbox{and}\qquad \int_{\pi^{-1}(E)} f^2 \, d\muhat = \int_E f_\rms^2 \, d\mu.
\end{equation}
Recall that $\pi : \Mhat \to M$ is the projection that maps each $z \in W_x$ to $x$, for all $x \in M$. For any $g : M \to \mathbb{R}$, $(g \circ \pi) : \Mhat \to \mathbb{R}$ is the function that maps each $z \in W_x$ to $g(x)$, for all $x \in M$. In Appendix \ref{DomainsAppendix}, we prove the following two results, each of
which will be used in the proofs of \eqref{CsjForwards} and \eqref{PoincareForward}.

\begin{prop} \label{DomainProp}
Suppose $(M, d, \mu, \mathcal{E}, \mathcal{F})$ satisfies Assumption \ref{BasicAssumptions}, $\phi$ is of regular growth, $(\mathcal{E}, \mathcal{F})$ admits a jump kernel, and $M=M_A \cup M_C$. Then

(a) For all $g \in \mathcal{F}$, the function $g \circ \pi$ belongs to $\Fhat$.

(b) For all $f \in \Fhat$, the function $f_\mean$ belongs to $\mathcal{F}$.

(c) For all $f \in \Fhat$, the function $f_\rms$ belongs to $\mathcal{F}$.
\end{prop}

\begin{lemma} \label{differencesSquaredLem}
Suppose $(M, d, \mu, \mathcal{E}, \mathcal{F})$ satisfies Assumption \ref{BasicAssumptions}, $(\mathcal{E}, \mathcal{F})$ admits a jump kernel, and $M=M_A \cup M_C$.
If $E_1$ and $E_2$ are Borel subsets of $M$, $f \in L^2(\Mhat, \muhat)$, and $g : M \times M \rightarrow [0, \infty)$ is measurable, then both
\begin{equation*}
    \int_{E_1} \int_{E_2} (f_\mean(x)-f_\mean(y))^2 g(x, y) \mu(dy) \mu(dx)
\end{equation*}
and
\begin{equation*}
    \int_{E_1} \int_{E_2} (f_\rms(x)-f_\rms(y))^2 g(x, y) \mu(dy) \mu(dx)
\end{equation*}
are less than or equal to
\begin{equation*}
    \int_{\pi^{-1}(E_1)} \int_{\pi^{-1}(E_2)} (f(z)-f(z'))^2 g(\pi(z), \pi(z')) \muhat(dz') \muhat(dz).
\end{equation*}
\end{lemma}

\subsection{Cut-off Sobolev} \label{cutoffSobolevSection}

In this subsection, we show that if $\CSJ$ holds for the original space, then $\CSJ$ holds for the auxiliary space.

We start by mentioning how to compute integrals with respect to the carr\'e du Champ operator. 
Let $\varphi \in \mathcal{F}$, let $h : M \rightarrow [0, \infty)$ be a measurable non-negative function, and let $C$ be a Borel subset of $M$.
By the definition of the carr\'e du Champ operator (defined in \eqref{carreDef}),
\begin{equation} \label{energyIntegral}
    \int_C h \, d\Gamma(\varphi) = \int_{x \in C} h(x) \int_{y \in M} (\varphi(x)-\varphi(y))^2 \, J(dx, dy).
\end{equation}
We will be particularly interested in the case where $\varphi = 1_A$, the indicator of some Borel $A \subseteq C$. In this case,
\begin{equation} \label{energyIntegralIndicator}
    \int_C h \, d\Gamma(1_A) = \int_{x \in A} h(x) \int_{y \in M \setminus A} J(dx,dy) + \int_{x \in C\setminus A} h(x) \int_{y \in A} J(dx,dy).
\end{equation}

In many of the results of this subsection, we will assume that $\VD$ and $\Jleq$ hold for the original space.
If so, let $C_D$ and $C_J$ be the constants such that
\begin{equation}\label{FirstVdJphi}
    V(x, 2r) \leq C_D V(x, r) \qquad \mbox{for all $x \in M$, $r>0$}
\end{equation}
and
\begin{equation}
    J(x, y) \leq \frac{C_J}{V(x, d(x, y)) \phi(d(x, y))} \qquad\mbox{for all $x, y \in M$}.
\end{equation}
Recall the definition of $v(x)$, from \eqref{v(x)Formula}.
If $\VD$ and $\Jleq$ hold, then by Lemma \ref{JumpsAndAtomsPreliminaryFactsLemma}(c) and \eqref{bowtie1}, there exists a constant $C_{\mathcal{J}}$ such that
\begin{equation}\label{CmathcaljforCSJ}
    \mathcal{J}(x, \rho) \leq \frac{C_{\mathcal{J}}}{\phi(\rho)} \quad\mbox{for all $x \in M$, $\rho>0$}, \qquad\mbox{and} \quad
    v(x) \leq \frac{C_{\mathcal{J}}}{\phi(D_x)} \quad\mbox{for all $x \in M_A$}
\end{equation}
(where $\mathcal{J}(x, \rho)$ is as defined in \eqref{MathcaljFormula}).

For all $z_0 \in \Mhat, \rho>0$, let
\begin{equation}\label{MathcalJhatFormula}
    \widehat{\mathcal{J}}(z_0, \rho) := \int_\Mhat \Jhat(z_0, z) \muhat(dz).
\end{equation}
(the analog of $\mathcal{J}(x, \rho)$ for the auxiliary space).

If $\VD$ and $\Jleq$ hold on the original space, they also hold on the auxiliary space by Proposition \ref{volumeGrowthProp} and \eqref{JleqBothways}, so there exist constants $C'_D$, $C'_J$, and $C'_{\mathcal{J}}$ such that
\begin{equation}
    \Vhat(z, 2r) \leq C'_D \Vhat(z, r) \qquad\mbox{for all $z\in\Mhat$, $r>0$},
\end{equation}
\begin{equation}
    \Jhat(z_1, z_2) \leq \frac{C'_J}{\Vhat( z_1, \dhat(z_1, z_2)) \phi(\dhat(z_1, z_2))} \qquad\mbox{for all $z_1, z_2 \in \Mhat$},
\end{equation}
and
\begin{equation}\label{LastVdJphi}
    \widehat{\mathcal{J}}(z, \rho) \leq \frac{C'_{\mathcal{J}}}{\phi(\rho)} \qquad\mbox{for all $z \in \Mhat$, $\rho>0$}.
\end{equation}
(where $\widehat{\mathcal{J}}(z, \rho)$ is as defined in \eqref{MathcalJhatFormula}).
Throughout the proofs in this section, if $\VD$ and $\Jleq$ is an assumption, we will refer to the constants $C_D$, $C_J$, $C_{\mathcal{J}}$, $C'_D$, $C'_J$, and $C'_{\mathcal{J}}$ that satisfy \eqref{FirstVdJphi}-\eqref{LastVdJphi}.

Recall that our goal is to show that $\CSJ$ holds for the auxiliary space. We seek to show that for $\muhat$-almost all $z_0 \in \Mhat$ and all $R \geq r>0$, if we let
\begin{align} \label{S1S2S3}
\begin{split}
    S_1 &:= \Bhat(z_0, R),\\
    S_2 &:= \Bhat(z_0, R+r),\\
    S_3 &:= \Bhat(z_0, R+(1+C_0)r),\\
    V &:= S_2 \setminus S_1,\\
    \mbox{and}\qquad V^* &:= S_3 \setminus \Bhat(z_0, R-C_0 r)
\end{split}
\end{align}
then for each $f \in \Fhat$, there exists a $\widehat{\varphi} \in \cutoff(S_1, S_2)$ such that
\begin{equation} \label{CsjAuxiliaryGoal}
    \int_{S_3} f^2 \, d\Gamma(\widehat{\varphi}) \leq C_1 \int_{V \times V^*} (f(z)-f(z'))^2 \Jhat(dz, dz') + \frac{C_2}{\phi(r)} \int_{S_3} f^2 \, d\muhat.
\end{equation}
Let $x_0 = \pi(z_0)$. The following lemma handles the case when $R \leq D_{x_0}$.

\begin{lemma} \label{CsjForwardsSmallR}

Suppose $(M, d, \mu, \mathcal{E}, \mathcal{F})$ satisfies Assumption \ref{BasicAssumptions}, $\phi$ is of regular growth, $(\mathcal{E}, \mathcal{F})$ admits a jump kernel, and $M=M_A \cup M_C$.
Fix $z_0 \in \Mhat$, $R \geq r>0$, and $f \in \Fhat$. Let $S_1$, $S_2$, $S_3$, $V$, and $V^*$ be as in \eqref{S1S2S3}, and let $x_0 = \pi(z_0)$. If $\VD$ and $\Jleq$ hold on the original space, and $R \leq D_{x_0}$, then
\begin{equation*}
    \int_{S_3} f^2 \, d\Gamma(1_{S_1}) \leq \frac{C}{\phi(R)} \int_{S_3} f^2 \, d\muhat
\end{equation*}
where $C$ is a constant that does not depend on any of $z_0, R, r, f$.

\end{lemma}

\begin{proof}
By \eqref{energyIntegralIndicator},
\begin{dmath}\label{711}
    \int_{S_3} f^2 \, d\Gamma(1_{S_1}) = \int_{z \in S_1} f^2(z) \left( \int_{z' \in \Mhat\setminus S_1} \Jhat(z, z') \muhat(dz') \right) \muhat(dz)\\
    + \int_{z \in S_3\setminus S_1} f^2(z) \left( \int_{z' \in S_1} \Jhat(z,z') \muhat(dz') \right)\muhat(dz).
\end{dmath}
If $z \in S_1$ and $z' \notin S_1$, then $\Jhat(z, z') = \Jhat(z_0, z')$. Therefore, for all $z \in S_1$,
\begin{equation}\label{712}
    \int_{\Mhat\setminus S_1} \Jhat(z,z') \muhat(dz') = \int_{\Mhat\setminus S_1} \Jhat(z_0, z') \muhat(dz') = \widehat{\mathcal{J}}(z_0, R) \leq \frac{C'_{\mathcal{J}}}{\phi(R)}.
\end{equation}
If $z \in S_3 \setminus S_1$ and $z' \in S_1$, then $\dhat(z, z') \geq R$. Therefore, for all $z \in S_3 \setminus S_1$, by $\Jleq$,
\begin{equation} \label{713}
    \int_{z' \in S_1} \Jhat(z, z') \muhat(dz') \leq \int_{z' \in S_1} \frac{C_J'}{\Vhat(z_0, R) \phi(R)} \muhat(dz') = \frac{C_J'}{\phi(R)}.
\end{equation}
By \eqref{712} and \eqref{713}, \eqref{711} becomes
\begin{dmath*}
    \int_{S_3} f^2 \, d\Gamma(1_{S_1}) \leq \frac{C'_{\mathcal{J}}}{\phi(R)} \int_{S_1} f^2 \, d\muhat + \int_{S_3} f^2 \, d\Gamma(1_{S_1}) + \frac{C'_J}{\phi(R)} \int_{S_3 \setminus S_1} f^2 \, d\muhat \leq \frac{\max\{ C_{\mathcal{J}}, C_J \}}{\phi(R)} \int_{S_3} f^2 \, d\muhat.
\end{dmath*}
\end{proof}

For the case when $R \geq D_{x_0}$, it will help to consider the corresponding situation on the original space. Let
\begin{align} \label{B1B2B3}
\begin{split}
    B_1 &:= B(x_0, R),\\
    B_2 &:= B(x_0, R+r),\\
    B_3 &:= B(x_0, R+(1+C_0)r),\\
    U &:= B_2 \setminus B_1,\\
    \mbox{and}\qquad U^* &:= B_3 \setminus B(x_0, R-C_0 r).
\end{split}
\end{align}
The following lemma tells us how every term in \eqref{CsjAuxiliaryGoal} compares to its corresponding term in \eqref{CsjFormula}, with $f_\rms$ playing the role of $f$ in \eqref{CsjFormula}. 

\begin{lemma} \label{CsjComparingTerms}

Suppose $(M, d, \mu, \mathcal{E}, \mathcal{F})$ satisfies Assumption \ref{BasicAssumptions}, $\phi$ is of regular growth, $(\mathcal{E}, \mathcal{F})$ admits a jump kernel, and $M=M_A \cup M_C$.
Fix $z_0 \in \Mhat$, $R \geq r>0$, and $f \in \Fhat$. Let $S_1$, $S_2$, $S_3$, $V$, and $V^*$ be as in \eqref{S1S2S3}, let $x_0 = \pi(z_0)$, and let $B_1$, $B_2$, $B_3$, $U$, and $U^*$ be as in \eqref{B1B2B3}. Let $\varphi$ be a function in $\cutoff(B_1, B_2)$.

If $R \geq D_{x_0}$, then

(a)

\begin{equation*}
    \int_{S_3} f^2 \, d\Gamma(\varphi \circ \pi) = \int_{B_3} f_\rms^2 \, d\Gamma(\varphi).
\end{equation*}

(b)

\begin{equation*}
    \int_{V \times V^*} (f(z)-f(z'))^2 \Jhat(dz, dz') \geq \int_{U \times U^*} (f_\rms(x)-f_\rms(y))^2 J(dx, dy).
\end{equation*}

(c)

\begin{equation*}
    \int_{S_3} f^2 \, d\muhat= \int_{B_3} f_\rms^2 \, d\mu.
\end{equation*}

\end{lemma}

\begin{proof}

(a) For all $x \in M$, let $$\alpha(x) := \int_{y\in M} (\varphi(x)-\varphi(y))^2 J(x, y) \mu(dy).$$ For all $z \in \Mhat$, let $$g(z) := \sqrt{\alpha(\pi(z))} f(z).$$ We claim that $g_\rms = \sqrt{\alpha} \cdot f_\rms.$ Indeed, for all $x \in M_A$,
\begin{equation*}
    g_\rms(x) = \left( \frac{1}{\mu(x)} \int_{W_x} \alpha(x) f^2 \, d\muhat \right)^{\sfrac12} = \sqrt{\alpha(x)} f_\rms(x)
\end{equation*}
and for all $x \in M_C$,
\begin{equation*}
    g_\rms(x) = g(x) = \sqrt{\alpha(x)} f(x) = \sqrt{\alpha(x)} f_\rms(x).
\end{equation*}
By \eqref{energyIntegral},
\begin{align*}
    \int_{S_3} f^2 \, d\Gamma(\varphi \circ \pi) &= \int_{z \in S_3} f^2(z) \left( \int_{z' \in \Mhat} (\varphi(\pi(z))-\varphi(\pi(z')))^2 \Jhat(z, z') \muhat(dz') \right) \muhat(dz)\\
    &=\int_{z \in S_3} f^2(z) \left( \int_{y \in M} (\varphi(\pi(z))-\varphi(y))^2 J(\pi(z), y) \mu(dy) \right) \muhat(dz)\\
    &= \int_{z \in S_3} \alpha(\pi(z)) f^2(z) \muhat(dz) &\mbox{(by the definition of $\alpha$)}\\
    &= \int_{S_3} g^2 \, d\muhat &\mbox{(by the definition of $g$)}\\
    &= \int_{B_3} g_\rms^2 \, d\mu &\mbox{(by \eqref{FmeanFrmsIntegralsAgree})}\\
    &= \int_{B_3} \alpha(x) \left( f_\rms(x) \right)^2 \mu(dx) &\mbox{(since $g_\rms = \sqrt{\alpha} \cdot f_\rms$)}\\
    &= \int_{B_3} \left( f_\rms(x) \right)^2 \int_M (\varphi(x)-\varphi(y))^2 J(x, y) \mu(dy) \mu(dx)\\
    &= \int_{B_3} f_\rms^2 \, d\Gamma(\varphi).
\end{align*}

(b) Since $R+r \geq R \geq D_{x_0}$, by \eqref{BhatFormula} we have $V = \pi^{-1}(U)$.
It is however possible for $V^*$ not to be equal to $\pi^{-1}(U^*)$, since $R-C_0 r$ might be less than $D_{x_0}$. The exact relation between $U^*$ and $V^*$ is
\begin{equation*}
    V^* = \pi^{-1}(U^*) \cup \left\{ z \in W_x : \dhat(z_0, z) \geq R-C_0 r \right\}.
\end{equation*}
By Lemma \ref{differencesSquaredLem},
\begin{equation} \label{vv*uu*differencesSquared}
    \int_{U \times U^*} (f_\rms(x)-f_\rms(y))^2 \, J(dx, dy) \leq \int_{\pi^{-1}(U) \times \pi^{-1}(U^*)} (f(z)-f(z'))^2 \, \Jhat(dz, dz').
\end{equation}
Since $\pi^{-1}(U) = V$ and $\pi^{-1}(U^*) \subseteq V^*$, \eqref{vv*uu*differencesSquared} implies
\begin{equation*}
    \int_{U \times U^*} (f_\rms(x)-f_\rms(y))^2 \, J(dx, dy) \leq \int_{V \times V^*} (f(z)-f(z'))^2 \, \Jhat(dz, dz').
\end{equation*}

(c) follows from \eqref{FmeanFrmsIntegralsAgree}.

\end{proof}

We now prove the desired result that $\VD$, $\Jleq$, and $\CSJ$ for the original space implies $\CSJ$ for the auxiliary space. Lemma \ref{CsjForwardsSmallR} tells us that \eqref{CsjAuxiliaryGoal} holds when $R \leq D_{x_0}$. For $R \geq D_{x_0}$, we use the assumption of $\CSJ$ on the original space to get \eqref{CsjFormula} with ($f_\rms$ playing the role of $f$) and then use Lemma \ref{CsjComparingTerms} to compare the terms of \eqref{CsjAuxiliaryGoal} with those of \eqref{CsjFormula}. We find that every comparison goes in the direction we need.

\begin{proof}[Proof of Proposition \ref{MasterProp}, implication \eqref{CsjForwards}]

Since $\CSJ$ holds on the original space, there exists an $E \subseteq M$ of full measure, and constants $C_0, C_1, C_2$ such that \eqref{CsjFormula} holds whenever $x_0 \in E$.

Fix $z_0 \in \pi^{-1}(E)$, $R \geq r>0$, and $f \in \Fhat$. Let $S_1$, $S_2$, $S_3$, $V$, and $V^*$ be as in \eqref{S1S2S3}, let $x_0 = \pi(z_0)$, and let $B_1$, $B_2$, $B_3$, $U$, and $U^*$ be as in \eqref{B1B2B3}. We must show that there exists a $\widehat{\varphi} \in \cutoff(S_1, S_2)$ (which may depend on $z_0$, $R$, $r$, and $f$) such that \eqref{CsjAuxiliaryGoal} holds.

If $R \leq D_{x_0}$, let $\widehat{\varphi}$ be the indicator $1_{S_1}$. Recall how $\Dhat$ and $\Fhat$ are defined (see Definition \ref{DhatDef}). It is clear that $1_{S_1} = H_{x_0, h}$ for a locally constant function $h : W \to \mathbb{R}$, so $1_{S_1} \in \Dhat \subseteq \Fhat$. It is also clear that $0 \leq 1_{S_1} \leq 1$ everywhere in $\Mhat$, $1_{S_1}=1$ in $S_1$, and $1_{S_1} = 0$ on $S_2^c$. Therefore, $\widehat{\varphi} = 1_{S_1} \in \cutoff(S_1, S_2)$. By Lemma \ref{CsjForwardsSmallR},
\begin{equation*}
    \int_{S_3} f^2 \, d\Gamma(\widehat{\varphi}) \leq \frac{C}{\phi(R)} \int_{S_3} f^2 \, d\muhat \leq \frac{C}{\phi(r)} \int_{S_3} f^2 \, d\muhat
\end{equation*}
(where $C$ is the constant from Lemma \ref{CsjForwardsSmallR} that does not depend on any of $z_0, R, r, f$).

Suppose $R \geq D_{x_0}$.
By Proposition \ref{DomainProp}(c), $f_\rms \in \mathcal{F}$. Since $x_0 \in E$, by our assumption of $\CSJ$ on the original space, there exists a $\varphi \in \cutoff(B_1, B_2)$ such that
\begin{equation} \label{CsjOnFL2}
    \int_{B_3} f_\rms^2 \, d\Gamma(\varphi) \leq C_1 \int_{U \times U^*} (f_\rms(x)-f_\rms(y))^2 J(dx, dy) + \frac{C_2}{\phi(r)} \int_{B_3} f_\rms^2 \, d\mu.
\end{equation}
Let $\widehat{\varphi} = \varphi \circ \pi$. By Proposition \ref{DomainProp}(a), $\varphi \circ \pi \in \Fhat$. It is clear from the fact that $\varphi \in \cutoff(B_1, B_2)$ that $0 \leq \varphi \circ \pi \leq 1$ everywhere in $\Mhat$, $\varphi \circ \pi = 1$ in $S_1$, and $\varphi \circ \pi = 0$ on $S_2^c$. Therefore, $\widehat{\varphi} = \varphi \circ \pi \in \cutoff(S_1, S_2)$.
We then have
\begin{align*}
    \int_{S_3} f^2 \, d\Gamma(\widehat{\varphi}) &= \int_{B_3} f_\rms^2 \, d\Gamma(\varphi) \qquad\mbox{(by Lemma \ref{CsjComparingTerms}(a))}\\
    &\leq C_1 \int_{U \times U^*} (f_\rms(x)-f_\rms(y))^2 J(dx, dy) + \frac{C_2}{\phi(r)} \int_{B_3} f_\rms^2 \, d\mu \qquad\mbox{(by \eqref{CsjOnFL2})}\\
    &\leq C_1 \int_{V \times V^*} (f(z)-f(z'))^2 \Jhat(dz, dz') + \frac{C_2}{\phi(r)} \int_{S_3} f^2 \, d\muhat \qquad\mbox{(by Lemma \ref{CsjComparingTerms} parts (b) and (c))}.
\end{align*}

\end{proof}

If we repeat the same argument, but let $\varphi$ be a cut-off function such that
\begin{equation*}
    \int_{B_3} g^2 \, d\Gamma(\varphi) \leq C_1 \int_{U \times U^*} (g(x)-g(y))^2 J(dx, dy) + \frac{C_2}{\phi(r)} \int_{B_3} g^2 \, d\mu
\end{equation*}
for \textit{all} $g \in \mathcal{F}$ (and then define $\widehat{\varphi}$ the same way, but before $f \in \Fhat$ has been fixed), we obtain \eqref{CsjAuxiliaryGoal} for \textit{all} $f \in \mathcal{F}$. Therefore, we have the following result.

\begin{prop} \label{ScsjForwards}
Suppose $(M, d, \mu, \mathcal{E}, \mathcal{F})$ satisfies Assumption \ref{BasicAssumptions}, $\phi$ is of regular growth, $(\mathcal{E}, \mathcal{F})$ admits a jump kernel, and $M=M_A \cup M_C$. If $\VD$ and $\Jleq$ hold for $(M, d, \mu, \mathcal{E}, \mathcal{F})$, then
\begin{equation*}
    \boxed{\mbox{$\SCSJ$ for $(M, d, \mu, \mathcal{E}, \mathcal{F})$}} \Longrightarrow \boxed{\mbox{$\SCSJ$ for $(\Mhat, \dhat, \muhat, \Ehat, \Fhat)$}}.
\end{equation*}
\end{prop}
(We do not use Proposition \ref{ScsjForwards} in this paper, but point it out for its own sake.)

\subsection{Poincar\'e inequality}

In this subsection, we show that $\PI$ for the original space implies $\PI$ for the auxiliary space.

Recall from Definition \ref{PoincareDefinition} that for $f \in \mathcal{F}$ and $B=B(x, r)$, we use $f_B$ to denote the mean value of $f$ on $B$. Let us do the same for $f \in \Fhat$ and $S = \Bhat(z, r)$:

\begin{equation*}
    f_S := \frac{1}{\muhat(S)} \int_S f \, d\muhat.
\end{equation*}
In the following lemma, we see that for $x_0 \in M_A$, a form of the Poincar\'e inequality on $\Mhat$ holds at small scales (within $W_{x_0}$).

\begin{lemma} \label{PiHelperProp}
Suppose $(M, d, \mu, \mathcal{E}, \mathcal{F})$ satisfies Assumption \ref{BasicAssumptions}, $(\mathcal{E}, \mathcal{F})$ admits a jump kernel, and $M=M_A \cup M_C$.
If $x_0 \in M_A$, $z_0 \in W_{x_0}$, $0<r \leq D_{x_0}$, and $S = \Bhat(z_0, r)$, then
\begin{equation*}
    \int_S (f-f_S)^2 \, d\muhat \leq \frac{\phi(r)}{2} \int_{S\times S} (f(z)-f(z'))^2 \Jhat(dz, dz')
\end{equation*}
for all $f \in\Fhat$.
\end{lemma}

\begin{proof}

Let $w_0$ be the element of $W_{x_0}$ such that $z_0=(x_0, w_0)$. If $\xi$ and $\eta$ are independent random variables, each uniform on $S$, then for all $f \in \Fhat$,
\begin{dmath} \label{PiHelper1}
    \int_S (f-f_S)^2 \, d\muhat = \muhat(S) \Var(f(\xi))\\
    = \frac{\muhat(S)}{2} \mathbb{E} \left[ (f(\xi)-f(\eta))^2 \right]\\
    = \frac{1}{2\muhat(S)} \int_{S} \int_{S} (f(z)-f(z'))^2 \muhat(dz') \muhat(dz).
\end{dmath}
If $z$ and $z'$ belong to $S=\Bhat(z_0, r)$, then $z=(x, w)$ and $z'=(x, w')$ for some $w, w'$. By the ultrametric property \eqref{WxUltrametric} on $W_{x_0}$,
\begin{equation*}
    \dhat(z, z') = \rho^{D_x}(w, w') \leq \max \left\{ \rho^{D_x}(z, z_0), \rho^{D_x}(z_0, z') \right\} = \max \left\{ \dhat(z, z_0), \dhat(z_0, z') \right\} < r.
\end{equation*}
By multiplying the integrand of \eqref{PiHelper1} by a quantity greater than or equal to $1$ (namely, $\frac{\Vhat(z, r)}{\Vhat\left(z, \dhat(z, z')\right)} \cdot \frac{\phi(r)}{\phi\left(\dhat(z,z')\right)}$),
\begin{equation}\label{PiHelper2}
    \int_S (f-f_S)^2 \, d\muhat \leq \frac{1}{2 \muhat(S)} \int_S \int_S (f(z)-f(z'))^2 \cdot \frac{\Vhat(z, r)}{\Vhat\left(z, \dhat(z, z')\right)} \cdot \frac{\phi(r)}{\phi\left(\dhat(z,z')\right)} \muhat(dz') \muhat(dz).
\end{equation}
Note that for all $z \in \Bhat(z_0, r)$, by the ultrametric property \eqref{WxUltrametricBalls} on $W_{x_0}$, we have $\Bhat(z, r) = \Bhat(z_0, r)=S$ (and thus $\Vhat(z, r) = \muhat(S)$). Thus, the $\Vhat(z, r)$ and $\muhat(S)$ terms in \eqref{PiHelper2} cancel:

\begin{align*} \label{PiHelper3}
    \int_S (f-f_S)^2 \, d\muhat &\leq \frac{\phi(r)}{2} \int_S \int_S (f(z)-f(z'))^2 \cdot \frac{1}{\Vhat\left(z, \dhat(z, z')\right) \phi\left(\dhat(z, z')\right)} \muhat(dz') \muhat(dz)\\
    &= \frac{\phi(r)}{2} \int_S \int_S (f(z)-f(z'))^2 \Jhat(dz, dz') \qquad\mbox{(by \eqref{JhatFormula})}.
\end{align*}
\end{proof}

Now to prove \eqref{PoincareForward} in full generality. Let us briefly explain our proof technique. Recall how $f_\mean$ was defined (in \eqref{fmeanDef}). Suppose $f \in \Fhat$ and $S = \Bhat(z_0, r)$. Let $x_0 = \pi(z_0)$ and $B=B(x_0, r)$. If $r \leq D_{x_0}$, we can simply use Lemma \ref{PiHelperProp} to put the necessary upper bound on $\int_S (f-f_S)^2 \, d\muhat$. If $r \geq D_{x_0}$, we first show that $\int_S (f-f_S)^2 \, d\muhat$ is at most the integral $\int_B \left[ f_\mean - (f_\mean)_B \right]^2 \, d\mu$, plus the sum (over all atoms $x \in B$) of $\int_{W_x} (f-f_{W_x})^2\, d\muhat$. We use the hypothesis that the Poincar\'e inequality holds on the original space to handle $\int_B \left[ f_\mean - (f_\mean)_B \right]^2 \, d\mu$, and we use Lemma \ref{PiHelperProp} to handle $\int_{W_x} (f-f_{W_x})^2\, d\muhat$ for each $x$. We also need to invoke Lemma \ref{DomainProp}(b) to guarantee that $f_\mean \in \mathcal{F}$.

\begin{proof}[Proof of Proposition \ref{MasterProp}, implication \eqref{PoincareForward}]

Let $\kappa$ and $C$ be the constants from $\PI$ for $(M, d, \mu, \mathcal{E}, \mathcal{F})$. Fix $z_0 \in \Mhat$, $r>0$, and $f \in \Fhat_b$ (where $\Fhat_b = \{ f \in \Fhat : \norm{f}_\infty < \infty\}$). Let $S = \Bhat(z_0, r)$ and $x_0 = \pi(z_0)$.

If $r \leq D_{x_0}$, then by Lemma \ref{PiHelperProp},
\begin{equation*}
    \int_S (f-f_S)^2 \, d\muhat \leq \frac{\phi(r)}{2} \int_{S \times S} (f(z)-f(z'))^2 \Jhat(dz, dz').
\end{equation*}

Suppose $r \geq D_{x_0}$. Let $B=B(x_0, r)$ and note that $S = \pi^{-1}(B)$, by \eqref{BhatFormula}. Recall the function $f_\mean$, defined in \eqref{fmeanDef}. We will apply the Poincar\'e inequality to $f_\mean$ and $B$.

By Proposition \ref{DomainProp}(b), $f_\mean \in \mathcal{F}$. Clearly, $\norm{f_\mean}_\infty \leq \norm{f}_\infty < \infty$, so $f_\mean \in \mathcal{F}_b$. By \eqref{FmeanFrmsIntegralsAgree}, $(f_\mean)_B = f_S$. Using the inequality $(a+b)^2 \leq 2a^2 + 2b^2$,
\begin{dmath}\label{A^2B^2inequality}
    \int_S (f-f_S)^2 \, d\muhat = \int_S \left( \left[ f-(f_\mean \circ \pi) \right] - \left[ (f_\mean \circ \pi) - (f_\mean)_B \right] \right)^2 \, d\muhat
    \leq 2 \int_S \left[ f-(f_\mean \circ \pi) \right]^2 \, d\muhat + 2\int_S \left[ (f_\mean \circ \pi) - (f_\mean)_B \right]^2 \, d\muhat 
    = 2 \sum_{x \in B \cap M_A} \int_{W_x} (f-f_{W_x})^2 \, d\muhat + 2 \int_B \left[ f_\mean - (f_\mean)_B \right]^2 \, d\mu.
\end{dmath}
By Lemma \ref{PiHelperProp},
\begin{equation} \label{Prop54ForMa}
    \int_{W_x} (f-f_{W_x})^2 \, d\muhat \leq \frac{\phi(D_x)}{2} \int_{W_x \times W_x} (f(z)-f(z'))^2 \Jhat(dz, dz') \qquad\mbox{for all $x \in M_A$}.
\end{equation}
For all $x \in B \cap M_A$, either $x=x_0$ (in which case, $D_x = D_{x_0}\leq r$) or $x \neq x_0$ (in which case, $D_x \leq d(x_0, x) < r$). In either case, $D_x \leq r$.
Therefore, we can replace the $\phi(D_x)$ in \eqref{Prop54ForMa} with $\phi(r)$, to obtain
\begin{equation}\label{Prop54ForB}
    \int_{W_x} (f-f_{W_x})^2 \, d\muhat \leq \frac{\phi(r)}{2} \int_{W_x \times W_x} (f(z)-f(z'))^2 \Jhat(dz, dz') \qquad\mbox{for all $x \in B \cap M_A$}.
\end{equation}
By $\PI$ for $(M, d, \mu, \mathcal{E}, \mathcal{F})$,
\begin{equation} \label{PoincSubcall}
    \int_{B} (f_\mean - (f_\mean)_B )^2 \, d\mu \leq C\phi(r) \int_{\kappa B \times \kappa B} (f_\mean(x)-f_\mean(y))^2 J(dx, dy).
\end{equation}
Since $\kappa r \geq r \geq D_{x_0}$, we have $\pi^{-1}(\kappa B) = \kappa S$. By \eqref{PoincSubcall} and Lemma \ref{differencesSquaredLem},
\begin{equation}\label{PoincSubcall'}
    \int_{B} (f_\mean - (f_\mean)_B )^2 \, d\mu \leq C\phi(r) \int_{\kappa S \times \kappa S} (f(z)-f(z'))^2 \Jhat(dz, dz').
\end{equation}
By \eqref{A^2B^2inequality}, \eqref{Prop54ForB}, and \eqref{PoincSubcall'},
\begin{dmath*}
    \int_{S} (f-f_S)^2 \, d\muhat \leq \phi(r) \sum_{x \in B \cap M_A} \int_{W_x \times W_x} (f(z)-f(z'))^2 \Jhat(dz, dz') + 2C \phi(r) \int_{\kappa S \times \kappa S} (f(z)-f(z'))^2 \Jhat(dz, dz')
    = \phi(r) \int_{\kappa S\times \kappa S} \left(2C + 1_{\left\{ \pi(z)=\pi(z') \in B \cap M_A  \right\}}\right) (f(z)-f(z'))^2 \Jhat(dz, dz')
    \leq (2C+1) \phi(r) \int_{\kappa S \times \kappa S} (f(z)-f(z'))^2 \Jhat(dz, dz').
\end{dmath*}
\end{proof}

\appendix

\section{Domains of the regular Dirichlet forms} \label{DomainsAppendix}

In this appendix we prove Proposition \ref{DomainProp} and Lemma \ref{differencesSquaredLem}.
\LongVersionShortVersion{}{In some cases, we include more detailed version of these proofs in \cite[Appendix A]{longversion}.}

We use the notation $\norm{f}_p$ for the $L^p$-norm of a function $f$, regardless of whether $f$ is a function on $M$ or $\Mhat$.
It follows from \eqref{integralsAgree} that for any $g \in \mathcal{F}$,
\begin{equation} \label{GcircpiAgresswithG}
    \norm{g \circ \pi}_2 = \norm{g}_2, \qquad \Ehat(g \circ \pi) = \mathcal{E}(g), \qquad\mbox{and} \qquad \Ehat_1(g \circ \pi) = \mathcal{E}_1(g).
\end{equation}
We use \eqref{GcircpiAgresswithG} to easily prove Proposition \ref{DomainProp}(a), but first let us briefly recall how $\Fhat$ was defined. In Section \ref{auxiliaryEFconstruction}, we defined the bilinear form $\Ehat$ (given by \eqref{EhatFormula}) on a subset $\Fhatmax$ of $L^2(\Mhat, \muhat)$. We defined the $\Ehat_1$-norm on $\Fhatmax$. by \eqref{Ehat1normFormula}. We then constructed a set $\Dhat \subseteq \Fhatmax$ in Definition \ref{DhatDef}, defined $\Ehat$ as the $\Ehat_1$-closure of $\Dhat$, and used Lemma \ref{constructRegularDirichletForms} to show that $(\Ehat, \Fhat)$ was a regular Dirichlet form.

\begin{proof}[Proof of Proposition \ref{DomainProp}(a)]

Recall that $\mathcal{F} \cap C_c(M)$ is a core of $(\mathcal{E}, \mathcal{F}$), and is therefore dense in $\mathcal{F}$ under the $\mathcal{E}_1$-norm. Let $\{g_n\}$ be a sequence in $\mathcal{F} \cap C_c(M)$ such that $\norm{g_n - g}_{\mathcal{E}_1} \to 0$. By Definition \ref{DhatDef}, $g_n \circ \pi \in \Dhat$ for all $n$. By \eqref{GcircpiAgresswithG}, 
\begin{equation} \label{gnCircPiToGCircpi}
    \norm{(g_n \circ \pi) - (g \circ \pi)}_{\Ehat_1} = \norm{(g_n - g) \circ \pi}_{\Ehat_1} = \norm{g_n - g}_{\mathcal{E}_1} \to 0.
\end{equation}
Recall that $\Fhat$ is defined as the $\Ehat_1$-closure of $\Dhat$. Since $g_n \circ \pi \in \Dhat$ for all $n$, \eqref{gnCircPiToGCircpi} means that $g \circ \pi \in \Fhat$.
\end{proof}

\LongVersionShortVersion
{ 
Now let us prove Lemma \ref{differencesSquaredLem}.
It will help to consider the following collection of random variables.
Let $\{A_x\}_{x \in M}$ be an independent collection of random variables such that $A_x$ is uniform on $W_x$ for all $x$. (If $x \in M_C$, this means $A_x=x$ with probability $1$.) Let $\{B_x\}_{x \in M}$ be an independent copy of $\{A_x : x \in M\}$.
Another way to express the definitions of $f_\mean$ and $f_\rms$ is
\begin{equation} \label{FmeanIntermsofRvs}
    f_\mean(x) = \mathbb{E}[f(A_x)] = \mathbb{E}[f(B_x)]
\end{equation}
and
\begin{equation} \label{FrmsIntermsofRvs}
    f_\rms(x) = \norm{f(A_x)}_2 = \norm{f(B_x)}_2.
\end{equation}
Note that if $X$ and $Y$ are random variables with finite second moment, the reverse triangle inequality gives
\begin{equation} \label{ReverseTriangleInequality}
    \left| \norm{X}_2 - \norm{Y}_2 \right|  \leq \norm{X-Y}_2.
\end{equation}
We will use \eqref{FmeanIntermsofRvs} and \eqref{FrmsIntermsofRvs} to prove results involving $f_\mean$ and $f_\rms$. Usually this will entail applying Jensen's inequality or \eqref{ReverseTriangleInequality} to the variables from $\{A_x\} \cup \{B_x\}$.
}
{ 
Now let us prove Lemma \ref{differencesSquaredLem}.
Note that if $\xi$ and $\eta$ are random variables with finite second moment, the reverse triangle inequality gives
\begin{equation} \label{ReverseTriangleInequality}
    \left| \norm{\xi}_2 - \norm{\eta}_2 \right|  \leq \norm{\xi-\eta}_2.
\end{equation}
} 

\begin{proof}[Proof of Lemma \ref{differencesSquaredLem}]

\LongVersionShortVersion{
Let
\begin{align*}
    \alpha_1 &:= \int_{E_1} \int_{E_2} (f_\mean(x)-f_\mean(y))^2 g(x, y) \mu(dy) \mu(dx),\\
    \alpha_2 &:= \int_{E_1} \int_{E_2} (f_\rms(x)-f_\rms(y))^2 g(x, y) \mu(dy) \mu(dx),\\
    \mbox{and} \qquad \beta &:=\int_{\pi^{-1}(E_1)} \int_{\pi^{-1}(E_2)} (f(z)-f(z'))^2 g(\pi(z), \pi(z')) \muhat(dz') \muhat(dz).
\end{align*}
We must show that both $\alpha_1$ and $\alpha_2$ are less than or equal to $\beta$.
By \eqref{FmeanIntermsofRvs} and Jensen's inequality,
\begin{align*}
    \alpha_1 = & \sum_{x \in E_1 \cap M_A} \sum_{y \in E_2 M_A} \mu(x)\mu(y) \left( \mathbb{E}[f(A_x)-f(B_y)]\right)^2 g(x, y)\\
    &+ \sum_{x \in E_1 \cap M_A} \mu(x) \int_{y\in E_2 \cap M_C} \left( \mathbb{E}[f(A_x)-f(y)]\right)^2 g(x, y) \mu(dy)\\
    &+ \sum_{y \in E_2 \cap M_A} \mu(y) \int_{x\in E_1 \cap M_C} \left( \mathbb{E}[f(x)-f(B_y)]\right)^2 g(x, y) \mu(dx)\\
    &+ \int_{(E_1 \cap M_C) \times (E_2 \cap M_C)} (f(x)-f(y))^2 g(x, y) \mu(dy) \mu(dx)\\
    \leq & \sum_{x \in E_1 \cap M_A} \sum_{y \in E_2 \cap M_A} \mu(x)\mu(y) \mathbb{E} \left[\left( f(A_x)-f(B_y) \right)^2\right] g(x, y)\\
    &+ \sum_{x \in E_1 \cap M_A} \mu(x) \int_{y\in E_2 \cap M_C} \mathbb{E} \left[\left( f(A_x)-f(y) \right)^2\right] g(x, y) \mu(dy)\\
    &+ \sum_{y \in E_2 \cap M_A} \mu(y) \int_{x\in E_1 \cap M_C} \mathbb{E} \left[\left( f(x)-f(B_y) \right)^2\right] g(x, y) \mu(dx)\\
    &+ \int_{(E_1 \cap M_C) \times (E_2 \cap M_C)} (f(x)-f(y))^2 g(x, y) \mu(dy) \mu(dx)\\
    =& \beta.
\end{align*}
By \eqref{FrmsIntermsofRvs} and \eqref{ReverseTriangleInequality},
\begin{align*}
    \alpha_2 =& \sum_{x \in E_1 \cap M_A} \sum_{y \in E_2 \cap M_A} \mu(x) \mu(y) \left| \norm{f(A_x)}_2 - \norm{f(B_y)}_2 \right|^2 g(x, y)\\
    &+ \sum_{x \in E_1 \cap M_A} \mu(x) \int_{y \in E_2\cap M_C} \left| \norm{f(A_x)}_2 - |f(y)| \right|^2 g(x, y) \mu(dy)\\
    &+ \sum_{y \in E_2 \cap M_A} \mu(y) \int_{x \in E_1 \cap M_C} \left||f(x)|-\norm{f(B_y)}_2 \right|^2 g(x, y) \mu(dx)\\
    &+ \int_{(E_1 \cap M_C) \times (E_2 \cap M_C)} (f(x)-f(y))^2 g(x, y) \mu(dy) \mu(dx)\\
    \leq& \sum_{x \in E_1 \cap M_A} \sum_{y \in E_2 \cap M_A} \mu(x) \mu(y) \norm{f(A_x)-f(B_y)}_2^2 g(x, y)\\
    &+ \sum_{x \in E_1 \cap M_A} \mu(x) \int_{y \in E_2\cap M_C} \norm{f(A_x)-f(y)}_2^2 g(x, y) \mu(dy)\\
    &+ \sum_{y \in E_2 \cap M_A} \mu(y) \int_{x \in E_1 \cap M_C} \norm{f(x)-f(B_y)}_2^2 g(x, y) \mu(dx)\\
    &+ \int_{(E_1 \cap M_C) \times (E_2 \cap M_C)} (f(x)-f(y))^2 g(x, y) \mu(dy) \mu(dx)\\
    =& \beta.
\end{align*}
}
{If
\begin{align*}
    \alpha_1 &:= \int_{E_1} \int_{E_2} (f_\mean(x)-f_\mean(y))^2 g(x, y) \mu(dy) \mu(dx),\\
    \alpha_2 &:= \int_{E_1} \int_{E_2} (f_\rms(x)-f_\rms(y))^2 g(x, y) \mu(dy) \mu(dx),\\
    \mbox{and} \qquad \beta &:=\int_{\pi^{-1}(E_1)} \int_{\pi^{-1}(E_2)} (f(z)-f(z'))^2 g(\pi(z), \pi(z')) \muhat(dz') \muhat(dz),
\end{align*}
then $\alpha_1 \leq \beta$ by Jensen's inequality and $\alpha_2 \leq \beta$ by \eqref{ReverseTriangleInequality}.}
\end{proof}

We still need to prove that for a function $f \in \Fhat$, both $f_\mean$ and $f_\rms$ belong to $\mathcal{F}$. For the proof of $f_\mean \in \mathcal{F}$, we will construct a sequence $\{f_n\} \subseteq \Dhat$ such that $f_n \to f$ in the $\Ehat_1$-norm, and show that $\{ (f_n)_\mean \} \subseteq \mathcal{F}$ and $(f_n)_\mean \to f_\mean$ in the $\mathcal{E}_1$-norm. By comparing $\mathcal{E}((f_\mean)_n - f_\mean)$ to $\Ehat_1(f_n-f)$, we conclude that $f_\mean \in \mathcal{F}$.

\LongVersionShortVersion{ 
The following lemma is used to make the comparison between $\mathcal{E}((f_\mean)_n - f_\mean)$ and $\Ehat_1(f_n-f)$.

\begin{lemma}\label{FhatSmallThanFInEhat1Norm}
Suppose $(M, d, \mu, \mathcal{E}, \mathcal{F})$ satisfies Assumption \ref{BasicAssumptions}, $(\mathcal{E}, \mathcal{F})$ admits a jump kernel, and $M=M_A \cup M_C$.
For all $f \in L^2(\Mhat, \muhat)$, $\norm{f_\mean}_2 \leq \norm{f}_2$ and $\mathcal{E}(f_\mean) \leq \Ehat(f)$. Consequently, $\mathcal{E}_1(f_\mean) \leq \Ehat_1(f)$.
\end{lemma}

\begin{proof}
For the $L^2$-norms,
\begin{align*}
    \norm{f_\mean}_2^2 &= \sum_{x \in M_A} \mu(x) (f_\mean(x))^2 + \int_{M_C} f^2 \, d\mu\\
    &= \sum_{x \in M_A} \mu(x) \left( \mathbb{E} \left[ f(A_x) \right] \right)^2 + \int_{M_C} f^2 \, d\mu \qquad\mbox{(by \eqref{FmeanIntermsofRvs})} \\
    &\leq \sum_{x \in M_A} \mu(x) \mathbb{E} \left[ f(A_x)^2 \right] + \int_{M_C} f^2 \, d\mu \qquad\mbox{(by Jensen's inequality)}\\
    &= \sum_{x \in M_A} \int_{W_x} f^2 \, d\muhat + \int_{M_C} f^2 \, d\mu\\
    &= \norm{f}_2.
\end{align*}
For the energies,
\begin{align*}
    \mathcal{E}(f_\mean) &= \int_{M\times M \setminus \diag_M} (f_\mean(x)-f_\mean(y))^2 J(x, y) \mu(dy) \mu(dx)\\
    &\leq \int_{\Mhat\times\Mhat \setminus \diag_{\Mhat}} (f(z)-f(z'))^2 J(\pi(z), \pi(z')) 1_{\left\{ \pi(z) \neq \pi(z') \right\}} \muhat(dz')\muhat(dz) \qquad\mbox{(by Lemma \ref{differencesSquaredLem})}\\
    &\leq \int_{\Mhat\times\Mhat \setminus \diag_{\Mhat}} (f(z)-f(z'))^2 \Jhat(z, z') \muhat(dz')\muhat(dz)\\ 
    &= \Ehat(f).
\end{align*}
\end{proof}
}
{ 
\begin{remark} \label{FmeanLowEnergy}
Suppose $f \in L^2(\Mhat, \muhat)$. By Jensen's inequality, $\norm{f_\mean}_2 \leq \norm{f}_2$. By Lemma \ref{differencesSquaredLem}, $\mathcal{E}(f_\mean) \leq \mathcal{E}(f)$. Consequently, $\mathcal{E}_1(f_\mean) \leq \Ehat(f)$.
\end{remark}
} 

\begin{proof}[Proof of Proposition \ref{DomainProp}(b)]
Recall that $\Fhat$ is defined as the $\Ehat_1$-closure of $\Dhat$. Since $f \in \Fhat$, there exists a sequence $\{f_n\} \subseteq \Dhat$ such that $f_n \to f$ in the $\Ehat_1$-norm.

For all $n$, by the definition of $\Dhat$, $f_n$ has a representation of the form
\begin{equation}\label{FmeanInF1}
    f_n = (g_n \circ \pi) + \sum_{j=1}^{N_n} H_{x^n_j, h^n_j}
\end{equation}
where $g_n \in \mathcal{F} \cap C_c(M)$, $N_n$ is a non-negative integer, and $x^n_j \in M_A$ and $h^n_j \in \Dtilde$ for all $1 \leq j\leq N_n$.

By taking averages over each $W_x$, \eqref{FmeanInF1} gives
\begin{equation}\label{FmeanInF2}
    (f_n)_\mean = g_n + \sum_{j=1}^{N_n} \left(\int_W h^n_j d\nu\right) \delta_{x_j}.
\end{equation}
For all $n$, $g_n$ belongs to $\mathcal{F}$ by construction, and each $\delta_{x^n_j}$ belongs to $\mathcal{F} \cap C_c(M)$ by Lemma \ref{JumpsAndAtomsPreliminaryFactsLemma}(a).
Thus, by \eqref{FmeanInF2}, $(f_n)_\mean \in \mathcal{F}$.

\LongVersionShortVersion{By Lemma \ref{FhatSmallThanFInEhat1Norm},}{By Remark \ref{FmeanLowEnergy},}
\begin{equation}\label{FmeanInF4}
    \norm{(f_n)_\mean - f_\mean}_{\mathcal{E}_1} = \norm{(f_n-f)_\mean}_{\mathcal{E}_1} \leq \norm{f_n-f}_{\Ehat_1} \to 0.
\end{equation}
Since $(f_n)_\mean \in \mathcal{F}$ for all $n$ and $\mathcal{F}$ is complete with respect to the $\mathcal{E}_1$-norm, \eqref{FmeanInF4} implies $f_\mean \in \mathcal{F}$.
\end{proof}

All that remains now is to prove Proposition \ref{DomainProp}(c). This is by far the hardest result in this section to prove.
Note that in the proof of Proposition \ref{DomainProp}(b), we used the fact that $(f_n-f)_\mean = (f_n)_\mean - f_\mean$. Unfortunately, the same distributive property does not hold for root-mean-square averages, so we can not use the same argument to show that $f_\rms \in \mathcal{F}$. Instead, we approximate $f_\rms$ by a series $\{g_n\} \subseteq \mathcal{F}$ such that $g_n = f_\mean$ except for on a finite set $I_n$, on which $g_n = f_\rms$, where the increasing union $\bigcup_n I_n$ is all of $M_A$.
We show that $\Ehat_1(f_\rms - g_n) \to 0$ (at least along a subsequence), but the argument is calculation-heavy and difficult to summarize, because so many terms arise from the quantity $\Ehat(f_\rms-g_n)$.

We start with the following lemma. The reader should think of $u$ in this lemma as $f-(f_\mean \circ \pi)$, where $f$ is a function in $\Fhat$, and $f_\rms$ is the function we are trying to show belongs to $\mathcal{F}$.
By Remark \ref{MaCountable}(a), $M_A$ is countable.

\begin{lemma} \label{FrmsInFHelper}
Suppose $(M, d, \mu, \mathcal{E}, \mathcal{F})$ satisfies Assumption \ref{BasicAssumptions}, $(\mathcal{E}, \mathcal{F})$ admits a jump kernel, and $M=M_A \cup M_C$.
Suppose $u \in \Fhat$ and $u_\mean$ is identically $0$. (In other words, $\int_{W_x} u \, d\muhat = 0$ for all $x \in M_A$, and $u(x)=0$ for all $x \in M_C$.) Let $\{x_j\}$ be an enumeration of $M_A$. (Since $M_A$ is countable, such an enumeration exists.) For all $n$, let
\begin{equation*}
    I_n = \{x_j : 1\leq j \leq n\},\qquad J_n = \{x_j : j>n \},
\end{equation*}
and
\begin{equation} \label{unuDefun}
    u_n = u \cdot 1_{(\pi^{-1}(I_n))}.
\end{equation}
Then there exists a sunsequence $\{u_{n_k}\}$ such that $\Ehat_1(u_{n_k}-u) \to 0$.
\end{lemma}

\begin{proof}
Clearly, $\norm{u-u_n}_2^2 = \sum_{j>n} \norm{u \cdot 1_{W_{x_j}}}_2^2 \to 0$ (since this is the tail of $\sum_j \norm{u \cdot 1_{W_{x_j}}}_2^2 = \norm{u}_2^2$, which is finite. Therefore, $u_n \to u$ in $L^2$-norm. All that remains to show is that $\Ehat(u_{n_k} - u) \to 0$ for some subsequence $\{n_k\}$.

Recall that $\Fhat$ is the $\Ehat_1$-closure of $\Dhat$. Since $u \in \Fhat$, there exists a sequence $\{v_k\} \subseteq \Dhat$ such that $\Ehat_1(v_k-u) \to 0$.
For all $k$, since $v_k \in \Dhat$, we can represent $v_k$ as
\begin{equation*}
    v_k = (g_k \circ \pi) + \sum_{j=1}^{n_k} H_{x_j, h^k_j}
\end{equation*}
where $g_k$ belongs to $\mathcal{F} \cap C_c(M)$, $n_k$ is a non-negative integer, and $h^k_j \in \Dtilde$ for all $j \leq n$.
We will show that $\Ehat(u-u_{n_k}) \leq \Ehat(u-v_k)$.

For all $z \in\Mhat$,
\begin{equation}\label{unu1}
    (u-v_k)(z) = \left\{ \begin{matrix}
        u(z) - g_k(x_j) - H_{x_j, h^k_j}(z) &:& \mbox{if $z \in W_{x_j}$ for some $j \leq n_k$}\\
        \\
        u(z)-g_k(x_j) &:& \mbox{if $z \in W_{x_j}$ for some $j > n_k$}\\
        \\
        -g_k(z) &:& \mbox{if $z\in M_C$}
    \end{matrix}\right.
\end{equation}
and
\begin{equation}\label{unu2}
    (u-u_{n_k})(z) = \left\{
    \begin{matrix}
        u(z) &:& \mbox{if $z \in \pi^{-1}(J_{n_k})$}\\
        \\
        0 &:& \mbox{otherwise}.
    \end{matrix}\right.
\end{equation}
For all $x \in M_A$ and $C \in \mathbb{R}$, since the mean of $u$ on $W_x$ is $0$,
\begin{equation}\label{unu3}
    \int_{W_x} (u-C)^2 \, d\muhat \geq \int_{W_x} u^2 \, d\muhat.
\end{equation}
Similarly, if $x$ and $y$ are distinct elements of $M_A$, then
\begin{equation}\label{unu4}
    \int_{W_x \times W_y} (u(z)-u(z')-C)^2 \, d\muhat\, d\muhat \geq \int_{W_x \times W_y} (u(z)-u(z'))^2 \, d\muhat\, d\muhat \qquad\mbox{for all $C \in \mathbb{R}$}.
\end{equation}

For the sake of brevity, for any $S \subseteq \Mhat \times \Mhat$, let
\begin{equation*}
    \alpha(S) := \int_{(z, z') \in E} \left((u-u_{n_k})(z) - (u-u_{n_k})(z')\right)^2 \Jhat(dz, dz')
\end{equation*}
and
\begin{equation*}
    \beta(S) := \int_{(z, z') \in E} \left((u-v_k)(z) - (u-v_k)(z')\right)^2 \Jhat(dz, dz')
\end{equation*}
We would like to show that $\alpha(\Mhat\times\Mhat) = \Ehat(u_{n_k}-u)$ is less than or equal to $\beta(\Mhat\times\Mhat) = \Ehat(u-v_k)$. We will do so by breaking $\Mhat \times \Mhat$ into various sets $S$, and showing that $\alpha(S) \leq \beta(S)$ for each one. In each case, we use \eqref{unu1} and \eqref{unu2} to express $\alpha(S)$ and $\beta(S)$, and use \eqref{unu3} and \eqref{unu4} to conclude that $\beta(S)$ is the larger one.

Note that $\Mhat = M_C \cup \pi^{-1}(I_{n_k}) \cup \pi^{-1}(J_{n_k})$.
By \eqref{unu2}, whenever $z$ and $z'$ both belong to $M_C \cup \pi^{-1}(I_{n_k})$, we have $(u-u_{n_k})(z)-(u-u_{n_k})(z') = 0$. Therefore,
\begin{equation} \label{DomainAlphaBetaFirst}
    \alpha \left( \left( M_C \cup \pi^{-1}(I_{n_k}) \right) \times \left( M_C \cup \pi^{-1}(I_{n_k}) \right) \right) = 0 \leq \beta \left( \left( M_C \cup \pi^{-1}(I_{n_k}) \right) \times \left( M_C \cup \pi^{-1}(I_{n_k}) \right) \right).
\end{equation}
If $x \in J_{n_k}$ and $z' \in M_C \cup \pi^{-1}(I_n)$, then
\begin{align*}
    \int_{W_x} ((u-v_k)(z) - (u-v_k)(z'))^2 \, \muhat(dz) &= \int_{W_x} (u(z)-g_k(x)-(u-v_k)(z'))^2 \, \muhat(dz) \qquad\mbox{(by \eqref{unu1})}\\
    &\geq \int_{W_x} u^2(z) \, \muhat(dz) \qquad\mbox{(by \eqref{unu3}, with $C = g_k(x) + (u-v_k)(z')$)}\\
    &= \int_{W_x} ((u-u_{n_k})(z) - (u-u_{n_k})(z'))^2 \, \muhat(dz) \qquad\mbox{(by \eqref{unu2})}.
\end{align*}
Therefore,
\begin{align}
\begin{split}
    \beta(\pi^{-1}(J_{n_k}) \times (M_C \cup \pi^{-1}(I_n))) &= \sum_{x \in J_{n_k}} \int_{z' \in M_C \cup \pi^{-1}(I_{n_k})} J(x, \pi(z')) \int_{W_x} ((u-v_k)(z) - (u-v_k)(z'))^2 \, \muhat(dz) \, \muhat(dz') \\
    &\geq \sum_{x \in J_{n_k}} \int_{z' \in M_C \cup \pi^{-1}(I_{n_k})} J(x, \pi(z')) \int_{W_x} ((u-u_{n_k})(z) - (u-u_{n_k})(z'))^2 \, \muhat(dz) \, \muhat(dz') \\
    &= \alpha(\pi^{-1}(J_{n_k}) \times (M_C \cup \pi^{-1}(I_n))).
\end{split}
\end{align}
If $x$ and $y$ are distinct elements of $J_{n_k}$, then
\begin{align}
\begin{split}
    \beta(W_x \times W_y) &= J(x, y) \int_{z \in W_x} \int_{z' \in W_y} \left(u(z) - g_k(x) - u(z') + g_k(y)\right)^2 \, \muhat(dz') \, \muhat(dz) \qquad\mbox{(by \eqref{unu1})}\\
    &\geq J(x, y) \int_{W_x \times W_y} (u(z)-u(z'))^2 \, \muhat(dz') \muhat(dz) \qquad\mbox{(by \eqref{unu4} with $C=g_k(x)-g_k(y)$)}\\
    &= \alpha(W_x \times W_y) \qquad\mbox{(by \eqref{unu2})}.
\end{split}
\end{align}
If $x \in J_{n_k}$, then by \eqref{unu1} and \eqref{unu2},
\begin{equation} \label{DomainAlphaBetaLast}
    \alpha(W_x \times W_x) = \int_{W_x \times W_x} (u(z)-u(z'))^2 \, \Jhat(dz, dz') = \beta(W_x \times W_x).
\end{equation}
By \eqref{DomainAlphaBetaFirst}-\eqref{DomainAlphaBetaLast}, we see that $\alpha(\Mhat \times \Mhat) \leq \beta(\Mhat \times \Mhat)$, or $\Ehat(u-u_{n_k}) \leq \Ehat(u-v_k)$. Since $\Ehat(u-v_k) \to 0$, this means that $\Ehat(u-u_{n_k}) \to 0$ as $k \to \infty$.
\end{proof}

Note that the key idea in our proof of Lemma \ref{FrmsInFHelper} (using \eqref{unu3} and \eqref{unu4} to compare energies) would not work if it weren't for the assumption that $u_\mean$ is identically $0$. Let us now prove that $f_\rms \in \mathcal{F}$ for all $f \in \Fhat$. Recall that our strategy is to approximate $f_\rms$ by a series $\{g_n\} \subseteq \mathcal{F}$ such that $g_n = f_\mean$ except for on a finite set $I_n$, on which $g_n = f_\rms$ (where the increasing union $\bigcup_n I_n$ is all of $M_A$), and show that $\Ehat_1(f_\rms - g_n) \to 0$ along a subsequence. We decompose the quantity $\Ehat_1(f_\rms - g_n)$ into a sum of several terms, each of which we show goes to $0$ (along a subsequence). One of these terms is bounded above by $\Ehat(u-u_{n_k})$, where $u = f - (f_\mean \circ \pi)$. We handle this term using Lemma \ref{FrmsInFHelper}.

\begin{proof}[Proof of Proposition \ref{DomainProp}(c)]
Let us assume without loss of generality that $f \geq 0$. We can do this because $|f| \in \Fhat$ and $|f|_\rms = f_\rms$, so proving the result for $|f|$ proves it for $f$. 

As in Lemma \ref{FrmsInFHelper}, let $\{x_j\}$ be an enumeration of $M_A$. Let $u := f-(f_\mean \circ \pi)$. Note that $u \in \Fhat$ and $u_\mean$ is identically $0$, just like in the setting of Lemma \ref{FrmsInFHelper}. Let $\{u_n\}$ be as in \eqref{unuDefun}.

Even though we have not yet shown that $f_\rms \in \mathcal{F}$, we do know that by Lemma \ref{differencesSquaredLem} that $\mathcal{E}(f_\rms) \leq \Ehat(f) < \infty$, and that
$\norm{f_\rms}_2^2 = \norm{f}_2^2 < \infty$, so $f_\rms \in \mathcal{F}_\max$.

Let $h:= f_\rms - f_\mean$. By Proposition \ref{DomainProp}(b), $f_\mean \in \mathcal{F} \subseteq \Fmax$. Since $\Fmax$ is a vector space, $h \in \Fmax$. Thus,
\begin{align*}
    \infty &> \mathcal{E}(h) = \int_{M\times M \setminus \diag_M} (h(x)-h(y))^2 J(dx, dy)\\
    &= \int_{M_A \times M_A \setminus \diag_{M_A}} (h(x)-h(y))^2 J(dx, dy) + 2 \int_{x \in M_A} \int_{y \in M_C} h^2(x) J(dx, dy) \qquad\mbox{(since $h$ is $0$ on $M_C$)}.
\end{align*}
In particular, we have both
\begin{equation} \label{FrmsInf1}
    \int_{M_A \times M_A \setminus \diag_{M_A}} (h(x)-h(y))^2 J(dx, dy) < \infty
\end{equation}
and
\begin{equation}\label{FrmsInf2}
    \int_{x \in M_A} \int_{y \in M_C} h^2(x) J(dx, dy) < \infty.
\end{equation}

For all $n$, let
\begin{equation*}
    g_n := f_\mean + \sum_{j \leq n} h(x_j) \delta_{x_j}
\end{equation*}
so that
\begin{equation} \label{CloseapproxtoFl2}
    f_\rms - g_n = \sum_{j>n} h(x_j) \delta_{x_j}.
\end{equation}
By Proposition \ref{DomainProp}(b) and Lemma \ref{JumpsAndAtomsPreliminaryFactsLemma}(a), $g_n \in \mathcal{F}$.
We will show that some subsequence of $\{g_n\}$ converges to $f_\rms$ under the $\mathcal{E}_1$-norm.

First,
\begin{dmath} \label{FrmsInf5}
    \norm{f_\rms-g_n}_2^2 = \sum_{j>n} \mu(x_j) h^2(x_j) = \sum_{j>n} \mu(x_j) \left[ f_\rms(x_j) - f_\mean(x_j) \right]^2 \leq \sum_{j>n}\mu(x_j) f_\rms^2(x_j) \to 0
\end{dmath}
since $\sum_{j>n}\mu(x_j) f_\rms^2(x_j)$ is the tail of $\sum_j \mu(x_j) f_\rms(x_j) = \int_{M_A} f_\rms^2 \, d\mu = \int_{\pi^{-1}(M_A)} f^2 \, d\muhat$, which is finite.

By \eqref{CloseapproxtoFl2},
\begin{equation} \label{Efl2-gn}
    \mathcal{E}(f_\rms-g_n) = \int_{J_n \times J_n \setminus \diag_{J_n}} (h(x)-h(y))^2 J(dx, dy) + 2 \int_{x \in J_n} \int_{y \in M\setminus J_n} h^2(x) J(dx, dy).
\end{equation}
Note that the decreasing intersection $\bigcap_{n \in \mathbb{N}} J_n \times J_n \setminus \diag_{J_n}$ is empty. Therefore, by \eqref{FrmsInf1} and continuity from above,
\begin{equation}\label{FrmsInf7}
    \lim_{n\to\infty} \int_{J_n\times J_n \setminus \diag_{J_n}} (h(x)-h(y))^2 J(dx, dy) = 0.
\end{equation}
\LongVersionShortVersion
{ 
For all $x \in M_A$, by \eqref{FrmsIntermsofRvs} and \eqref{ReverseTriangleInequality},
\begin{equation*}
    h(x) = f_\rms(x) - f_\mean(x) = \norm{f(A_x)}_2 - \norm{f_\mean(x)}_2 \leq \norm{f(A_x)-f_\mean(x)}_2 = u_\rms(x).
\end{equation*}
}
{ 
For all $x \in M_A$, f $\xi$ is uniformly distributed on $W_x$, then by \eqref{ReverseTriangleInequality},
\begin{equation*}
    h(x) = f_\rms(x) - f_\mean(x) = \norm{f(\xi)}_2 - \norm{f_\mean(x)}_2 \leq \norm{f(\xi)-f_\mean(x)}_2 = u_\rms(x).
\end{equation*}
} 
Therefore,
\begin{align}\label{FrmsInf9}
\begin{split}
    \int_{x \in J_n} \int_{y \in M\setminus J_n} h^2(x) J(dx,dy) &\leq \int_{x \in J_n} \int_{y \in M\setminus J_n} u_\rms^2(x) J(dx,dy)\\
    &\leq \int_{z \in \widehat{J_n}} \int_{z' \in \widehat{M}\setminus \widehat{J_n}} u^2(z) \Jhat(dz, dz') \qquad\mbox{(by Lemma \ref{differencesSquaredLem})}\\
    &= \int_{z \in \widehat{J_n}} \int_{z' \in \widehat{M}\setminus \widehat{J_n}} \left( (u-u_n)(z) - (u-u_n)(z')\right) \Jhat(dz, dz')\\
    &\leq \Ehat(u-u_n).
\end{split}
\end{align}
Let $\{n_k\}$ be the subsequence given by Lemma \ref{FrmsInFHelper} such that $\Ehat_1(u-u_{n_k}) \to 0$.
By \eqref{Efl2-gn}, \eqref{FrmsInf7}, and \eqref{FrmsInf9},
\begin{equation}\label{FrmsInf10}
    \Ehat(f_\rms - g_{n_k}) \to 0.
\end{equation}
By \eqref{FrmsInf5} and \eqref{FrmsInf10},
\begin{equation}\label{FrmsInf11}
    \Ehat_1(f_\rms - g_{n_k}) = \norm{f_\rms - g_n}_2^2 + \Ehat(f_\rms - g_{n_k}) \to 0.
\end{equation}
Since $\Fhat$ is closed and $g_{n_k} \in \Fhat$ for all $k$, \eqref{FrmsInf11} implies that $f_\rms \in \mathcal{F}$.
\end{proof}

\LongVersionShortVersion{

\numberwithin{equation}{subsection}
\numberwithin{theorem}{subsection}

\section{Additional details}

\subsection{\texorpdfstring{$\VD + \mbox{quasi-uniform perfectness} \Longrightarrow \QRVD$}{TEXT}} \label{UniformPerfectnessAppendix}

\begin{proof}[Proof of Lemma \ref{VdPlusQufImpliesQrvd}]
By quasi-uniform perfectness, there exists a $C_U > 1$ such that $B(x, C_U r) \setminus B(x, r)$ is non-empty for all $x \in M$, $r \geq D_x$.
By $\VD$, there exists a constant $C_D$ such that
\begin{equation*}
    V(x, (2C_U+1)r) \leq C_D V(x, r) \qquad\mbox{for all $x \in M$, $r>0$}.
\end{equation*}
We will show that
\begin{equation}\label{GraphQrvdGoal}
    V(x, (2C_U+1)r) \geq \left(1+\frac{1}{C_D-1}\right) V(x, r) \qquad\mbox{for all $x \in M$, $r \geq D_x$}.
\end{equation}
Obviously, $C_D$ must be greater than $1$, or else $B(x, r)^c$ would have measure $0$ for all $x$ and $r$. Thus, \eqref{GraphQrvdGoal} is enough to establish $\QRVD$.

Fix $x \in M$ and $r \geq D_x$. By the uniform perfectness, there exists a $y \in B(x, 2C_U r) \setminus B(x, 2r)$. The balls $B(x, r)$ and $B(y, r)$ are disjoint, because if there was some $z \in B(x, r) \cap B(y, r)$, we would have
\begin{equation*}
    2r \leq d(x, y) \leq d(x, z)+d(z,y) < r + r = 2r.
\end{equation*}
Both $B(x, r)$ and $B(y, r)$ are contained in $B(x, (2C_U+1)r)$, because for all $z \in B(y, r)$,
\begin{equation*}
    d(x, z) \leq d(x, y)+d(y, z) < 2C_U r+r = (2C_U+1)r.
\end{equation*}
By the same argument, both are contained in $B(y, (2C_U+1)r)$. Therefore,
\begin{equation}\label{2ballsInbigXball}
    V(x, r) + V(y, r) \leq V(x, (2C_U+1)r)
\end{equation}
and
\begin{equation} \label{2ballsInbigYball}
    V(x, r) + V(y, r) \leq V(x, (2C_U+1)r) \leq C_D V(y, r).
\end{equation}
Simplifying \eqref{2ballsInbigYball}, we obtain
\begin{equation}\label{Simplified2ballsInbigYball}
    \frac{1}{C_D-1} V(x, r) \leq V(y, r).
\end{equation}
By \eqref{2ballsInbigXball} and \eqref{Simplified2ballsInbigYball},
\begin{equation*}
    \left(1+\frac{1}{C_D-1}\right) V(x, r) \leq V(x, r) + V(y, r) \leq V(x, (2C_U+1)r),
\end{equation*}
completing the proof of \eqref{GraphQrvdGoal}.
\end{proof}

\subsection{Section \ref{preliminariesSection} proofs} \label{preliminariesSectionProofs}

Let us prove all the facts stated in Section \ref{preliminariesSection}.
We first prove Proposition \ref{AeChapmanKolmogorov} (if $p(t, x, y)$ satisfies \eqref{HeatKernelDefinitionDistribution}, then it satisfies \eqref{HeatKernelDefinitionChapmanKolmogorov} for almost every $z$). 

\begin{proof}[Proof of Proposition \ref{AeChapmanKolmogorov}]
Fix $x \in M_0$ and $s, t > 0$. Let $F(z)$ be the left-hand-side of \eqref{HeatKernelDefinitionChapmanKolmogorov}, and $G(z)$ be the right-hand-side of \eqref{HeatKernelDefinitionChapmanKolmogorov}. We will show that $F=G$ $\mu$-a.e by showing that $\int_M F(z) f(z) \, \mu(dz) = \int_M G(z) f(z) \, \mu(dz)$ for all $f \in L^\infty(M, \mu)$. Fix such an $f$. By \eqref{HeatKernelDefinitionDistribution} the Markov property, and Fubini's theorem,
\begin{align*}
    \int_M F(z) f(z) \, \mu(dz) &= \int_M p(s+t, x, z) \, \mu(dz)\\
    &= \mathbb{E}^x f(X_{s+t})\\
    &=\mathbb{E}^x \left( \mathbb{E}^{X_s} f(X_t) \right)\\
    &=\mathbb{E}^x \left( \int_M p(t, X_s, z) f(z)\, \mu(dz) \right)\\
    &= \int_M p(s, x, y) \int_M  p(t, y, z) f(z) \, \mu(dz) \, \mu(dy)\\
    &= \int_M \int_M p(s, x, y) p(t, y, z) \mu(dz) \, f(z) \, \mu(dz)\\
    &= \int_M G(z) f(z) \, \mu(dz).
\end{align*}
\end{proof}

In order to prove Lemma \ref{constructRegularDirichletForms} (the recipe with which to construct regular Dirichlet forms on a general measure space $(\mathcal{X}, m)$), it will help to review some of the basic theory and terminology of Dirichlet forms. Given a $\sigma$-finite measure space $(\mathcal{X}, \mathcal{A}, m)$, recall that a \textit{Dirichlet form} is a pair $(\mathscr{E}, \mathscr{F})$ satisfying the following axioms:

\begin{itemize}
    \item $\mathscr{F}$ is a dense linear subspace of $L^2(\mathcal{X}, m)$
    \item $\mathscr{E} : \mathscr{F} \times \mathscr{F} \to \mathbb{R}$ is bilinear.
    \item If
    \begin{equation*}
        \mathscr{E}_\alpha(f, g) := \mathscr{E}(f, g) + \alpha \int_{\mathcal{X}} fg \,dm \qquad\mbox{for all $f, g \in \mathscr{F}, \alpha>0$}
    \end{equation*}
    then $\mathscr{F}$ is closed with respect to the norm $\norm{f}_{\mathscr{E}_1} = \sqrt{\mathscr{E}_1(f, f)}$. (We will henceforth refer to this norm as ``the $\mathscr{E}_1$-norm.")
    \item For all $f \in \mathscr{F}$, the function 
    \begin{equation*}
        f^*(x) := \max\left\{ 0, \min\left\{f(x), 1\right\} \right\}
    \end{equation*}
    also belongs to $\mathscr{F}$, and
    \begin{equation*}
        \mathscr{E}(f^*, f^*) \leq \mathscr{E}(f,f).
    \end{equation*}
\end{itemize}
The fourth and last of the above axioms is called the Markovian property, and a space $\mathcal{D}$ of measurable functions of $\mathcal{F}$ is called \textit{Markovian} if $f^* \in \mathcal{D}$ for all $f \in \mathcal{D}$. The operation $f \mapsto f^*$ is a \textit{normal contraction}: in other words, $|f^*| \leq |f|$ and $|f^*-g^*| \leq |f-g|$ for all $f, g$.

Now suppose $\mathcal{X}$ is a locally compact, separable metric space, and $m$ is a positive Radon measure on $\mathcal{X}$ with full support.
A subspace $\mathcal{C}$ of $\mathscr{F} \cap C_c(\mathcal{X})$ is called a \textit{core} if the following are also satisfied:

\begin{itemize}
    \item $\mathcal{C}$ is dense in $\mathscr{F}$ under the $\mathscr{E}_1$-norm.
    \item $\mathcal{C}$ is dense in $C_c(\mathcal{X})$ under the uniform norm.
\end{itemize}
A Dirichlet form with a core is called \textit{regular}. One can see from this definition that a Dirichlet form $(\mathscr{E}, \mathscr{F})$ is regular iff $\mathscr{F} \cap C_c(\mathcal{X})$ is a core.

Let us now prove Lemma \ref{constructRegularDirichletForms}. We must show that if $\mathscr{E}$, $\Fscrmax$, $\mathscr{E}$, $\mathscr{E}_1$, $\mathcal{D}$, and $\mathscr{F}$ are all as in Lemma \ref{constructRegularDirichletForms}, then $(\mathscr{E}, \mathscr{F})$ is a regular Dirichlet form on $L^2(\mathcal{X}, m)$, with $\mathcal{D}$ as a core. This is a simple matter of checking each of the axioms listed above.

\begin{proof}[Proof of Lemma \ref{constructRegularDirichletForms}]
Since $\mathcal{D}$ is dense in $C_c(\mathcal{X})$ under the uniform norm, it is also dense in $C_c(\mathcal{X})$ under the $L^2$-norm. By \cite[Proposition 7.9]{fol}, $C_c(\mathcal{X})$ is dense in $L^2(\mathcal{X}, m)$, under the $L^2$-norm, and therefore so is $\mathcal{D}$. Since $\mathcal{D}$ is dense in $L^2(\mathcal{X}, m)$, so is the larger $\mathscr{F}$.

Let us show that $\mathscr{F} \subseteq \Fscrmax$. Note that $\Fscrmax$ is complete under the $\mathscr{E}_1$-norm. This is a standard result. (See \cite[Example 1.2.4]{fukush}. In fact, $(\mathscr{E}, \Fscrmax)$ is a Dirichlet form, although it is not necessarily regular.)
Since $\mathcal{D}$ is a subset of $\Fscrmax$,
\begin{equation*}
    \mathscr{F} = \overline{\mathcal{D}}^{\mathscr{E}_1} \subseteq \overline{\Fscrmax}^{\mathscr{E}_1} = \Fscrmax.
\end{equation*}
For all $f, g \in \mathscr{F}$, since $\mathscr{F} \subseteq \Fscrmax$, by Cauchy-Schwarz,
\begin{equation*}
    \left|\mathscr{E}(f, g) \right| \leq \mathscr{E}(f)^{\sfrac12} \mathscr{E}(g)^{\sfrac12} <\infty.
\end{equation*}
Therefore, $\mathscr{E} : \mathscr{F} \times\mathscr{F}\to\mathbb{R}$. It is clear from \eqref{EscrFormula} that $\mathscr{E}$ is bilinear.

Because $\mathscr{F}$ was defined as the $\mathscr{E}_1$-closure of $\mathcal{D}$, $\mathscr{F}$ is closed under the $\mathscr{E}_1$-norm.

In order to show that $\mathscr{F}$ is Markovian, fix $f \in \mathscr{F}$, so that we can show that $f^* \in \mathscr{F}$ (where $f^* := 0 \vee (f \wedge 1)$, as in \eqref{Markovian}). Let $\{f_n\}$ be a sequence of functions in $\mathcal{D}$ such that $\norm{f_n-f}_{\mathscr{E}_1} \rightarrow 0$. For all $n$, because $\mathcal{D}$ is Markovian, $f_n^* \in\mathcal{D}$. For all $x, y \in \mathcal{X}$, $$|f_n^*(x)-f^*(y)| \leq |f_n(x)-f(y)|.$$
Therefore, $$\mathscr{E}(f_n^*-f^*) \leq \mathscr{E}(f_n-f) \rightarrow 0 \qquad\mbox{and}\qquad\norm{f_n^* - f^*}_{L^2(\mathcal{X}, m)} \leq \norm{f_n - f}_{L^2(\mathcal{X}, m)} \rightarrow 0$$
so $\norm{f_n^* - f^*}_{\mathscr{E}_1} \rightarrow 0$. Since $\mathscr{F}$ is the $\mathscr{E}_1$-closure of $\mathcal{D}$, and $f^*$ is the $\mathscr{E}_1$-limit of $\{f_n^*\} \subseteq \mathcal{D}$, $f^* \in \mathscr{F}$. Since $|f^*(x)-f^*(y)| \leq |f(x)-f(y)|$ for all $x, y \in \mathcal{X}$, $\mathscr{E}(f^*, f^*) \leq \mathscr{E}(f,f)$.

We have shown that $(\mathscr{E}, \mathscr{F})$ is a Dirichlet form. All that remains is to show that $\mathcal{D}$ is a core. By the definition of $\mathscr{F}$, $\mathcal{D}$ is dense in $\mathscr{F}$ in $\mathscr{E}_1$-norm. By hypothesis, $\mathcal{D}$ is dense in $C_c(\mathcal{X})$ under the uniform norm.
\end{proof}

Before we prove Lemma \ref{JumpsAndAtomsPreliminaryFactsLemma}, let us review one more standard fact from the theory of Dirichlet forms. A core $\mathcal{C}$ of a regular Dirichlet form $(\mathcal{E}, \mathcal{F})$ is called \textit{standard} if $\mathcal{C}$ is a dense linear subspace of $C_c(\mathcal{X})$ and for all $\varepsilon>0$, there exists a function $\phi_\varepsilon :\mathbb{R}\to [-\varepsilon, 1+\varepsilon]$ such that
\begin{itemize}
    \item $\phi_\varepsilon(t) = t$ for all $t \in [0, 1]$.
    \item $0 \leq \phi_\varepsilon(t) - \phi_\varepsilon(s) \leq t-s$ for all $s<t$.
    \item $\phi_\varepsilon \circ f \in \mathcal{C}$ for all $f \in \mathcal{C}$.
\end{itemize}
We call $\mathcal{C}$ a \textit{standard special core} if $\mathcal{C}$ is a standard core and also satisfies the additional properties
\begin{itemize}
    \item $\mathcal{C}$ is a dense subalgebra of $C_c(\mathcal{X})$
    \item Whenever $K \subseteq U$, $K$ is compact, $U$ is open, and $\overline{U}$ is compact, there exists a non-negative $f \in \mathcal{F} \cap C_c(M)$ such that $f=0$ outside of $U$ and $f=1$ on $K$.
\end{itemize}
It is known (see \cite[Exercise 1.4.1]{fukush}) that every regular Dirichlet form has a standard special core:

\begin{lemma}\label{standardSpecial}
Let $\mathcal{X}$ be a locally compact separable metric space. Let $m$ be a positive Radon measure on $X$ with full support.
If $(\mathscr{E}, \mathscr{F})$ is a regular Dirichlet form on $L^2(\mathcal{X}, m)$, then $\mathscr{F}\cap C_c(M)$ is a standard special core.
\end{lemma}

Let us return our setting of $(M, d, \mu, \mathcal{E}, \mathcal{F})$.
Recall that $M_A$ is the set of isolated atoms in $M$, as defined in \eqref{MaMcDef}.
We use Lemma \ref{standardSpecial} to show that the indicator $\delta_x$ belongs to $\mathcal{F}$ for all $x \in M_A$.

\begin{proof}[Proof of Lemma \ref{JumpsAndAtomsPreliminaryFactsLemma}(a)]
Fix $x \in M_A$.
By Lemma \ref{standardSpecial}, since $(\mathcal{E}, \mathcal{F})$ is a regular Dirichlet form, $\mathcal{F} \cap C_c(M)$ is a standard special core. Since $x$ is isolated, $\{x\}$ is open, closed, and compact.
By the definition of a standard special core, there exists an $f \in \mathcal{F} \cap C_c(M)$ such that $f=0$ outside of $\{x\}$ and $f=1$ on $\{x\}$. Therefore, $\delta_x \in \mathcal{F}$.
\end{proof}

We use the fact that $\delta_x \in \mathcal{F}$ to show that $v(x)$ is finite for all $x \in M_A$.

\begin{proof}[Proof of Lemma \ref{JumpsAndAtomsPreliminaryFactsLemma}(b)]
Fix $x \in M_A$.
By Lemma \ref{JumpsAndAtomsPreliminaryFactsLemma}(a), $\delta_x \in \mathcal{F}$. Therefore,
\begin{equation*}
    \infty > \mathcal{E}(\delta_x, \delta_x) = \int_{M \times M} (\delta_x(y)-\delta_x(z))^2 J(dy, dz) = 2 \int_{\{x\} \times (M \setminus \{x\})} J(y, z) \mu(dz) \mu(dy) = 2\mu(x) v(x).
\end{equation*}
\end{proof}

Lemma \ref{JumpsAndAtomsPreliminaryFactsLemma}(c) states that if $\VD$ and $\Jleq$ hold, then $\mathcal{J}(x, \rho) \leq C_{\mathcal{J}}/\phi(\rho)$ for all $x \in M$ and $\rho>0$, where $C_{\mathcal{J}}$ is some constant. We show this by expressing the integral $\mathcal{J}(x, \rho) = \int_{M \setminus \{x\}} J(x, y) \, \mu(dy)$ as a sum of integrals of $J(x, \cdot)$ over growing annuli. Each annulus's volume is bounded above by $\VD$, and the value of $J(x, \cdot)$ in each annulus is bounded above by $\Jleq$.

\begin{proof}[Proof of Lemma \ref{JumpsAndAtomsPreliminaryFactsLemma}(c)]
Assume $\VD$ and $J_{\phi, \leq}$. Let $C$ be the constant from \eqref{VdFormula}. By $\Jleq$, there exists a constant $C_0>0$ such that
\begin{equation*}
    J(x, y) \leq \frac{C_0}{V(x, d(x, y)) \phi(d(x, y))}
\end{equation*}
for all distinct $x, y$.

Fix $x \in M$ and fix $\rho>0$. For all $ k \in \mathbb{N}$, let $E_k = \{y \in M : 2^{k-1}\rho \leq d(x, y) < 2^k\rho\}$. The $E_k$'s form a partition of $M \setminus B(x, \rho)$. Therefore,
\begin{equation}\label{bigJumpsPartition1}
    \mathcal{J}(x, \rho) = \sum_{k=1}^\infty \int_{E_k} J(x, y) \mu(dy).
\end{equation}
For all $y \in E_k$,
\begin{equation}\label{BigJumpsyEk1}
    J(x, y) \leq \frac{C_0}{V(x, d(x, y)) \phi(d(x, y))} \leq \frac{C_0}{V(x, 2^{k-1} \rho) \phi(2^{k-1} \rho)}.
\end{equation}
By \eqref{bigJumpsPartition1} and \eqref{BigJumpsyEk1},
\begin{align*}
    \mathcal{J}(x, \rho) &\leq \sum_{k=1}^\infty \frac{C_0 \mu(E_k)}{V(x, 2^{k-1} \rho) \phi(2^{k-1} \rho)} = C_0 \sum_{k=1}^\infty \frac{V(x, 2^k \rho)-V(x, 2^{k-1} \rho)}{V(x, 2^{k-1} \rho)} \cdot \frac{1}{\phi(2^{k-1} \rho)}\\
    &\leq C_0 (C-1) \sum_{k=1}^\infty \frac{1}{\phi(2^{k-1} \rho)} \qquad\mbox{(by $\VD$)}\\
    &\leq \frac{C_0(C-1)}{\phi(\rho)} \sum_{k=1}^\infty c_1^{-1} 2^{-\beta_1 (k-1)} \qquad\mbox{(by \eqref{regularGrowth})}.
\end{align*}
\end{proof}

Lemma \ref{JumpsAndAtomsPreliminaryFactsLemma}(d), the lower bound for $\mathcal{J}(x, \rho)$, is proved in a similar same way.

\begin{proof}[Proof of Lemma \ref{JumpsAndAtomsPreliminaryFactsLemma}(d)]
Assume $\QRVD$ and $\Jgeq$. Let $c$ and $\ell$ be the constants from \eqref{QrvdFormula}. By $\Jgeq$, there exists a $c_0>0$ such that
\begin{equation*}
    J(x, y) \geq \frac{c_0}{V(x, d(x, y)) \phi(d(x, y))}
\end{equation*}
for all distinct $x, y$.

Fix $x \in M$ and a positive $\rho \geq D_x$. For all $k \in \mathbb{N}$, let $E_k = \{y \in M : \ell^{k-1} \rho \leq d(x, y) < \ell^k \rho\}$. The $E_k$'s form a partition of $M \setminus B(x, \rho)$. Therefore,
\begin{equation} \label{bigJumpsPartition2}
    \mathcal{J}(x, \rho) = \sum_{k=1}^\infty \int_{E_k} J(x, y) \mu(dy).
\end{equation}
For $\mu$-almost all $y \in E_k$,
\begin{equation}\label{BigJumpsyinEk2}
    J(x, y) \geq \frac{c_0}{V(x, d(x, y)) \phi(d(x, y))} \geq \frac{c_0}{V(x, \ell^k \rho) \phi(\ell^k \rho)}.
\end{equation}
By \eqref{bigJumpsPartition2} and \eqref{BigJumpsyinEk2},
\begin{align*}
    \mathcal{J}(x, \rho) &\geq \sum_{k=1}^\infty \frac{c_0 \mu(E_k)}{V(x, \ell^k \rho) \phi(\ell^k \rho)} = c_0 \sum_{k=1}^\infty \frac{V(x, \ell^k \rho)-V(x, \ell^{k-1} \rho)}{V(x, \ell^k \rho)} \cdot \frac{1}{\phi(\ell^k \rho)}\\
    &\geq c_0 (1-c^{-1}) \sum_{k=1}^\infty \frac{1}{\phi(\ell^k \rho)} \qquad\mbox{(by $\QRVD$)}\\
    &\geq \frac{c_0(1-c^{-1})}{\phi(\rho)} \sum_{k=1}^\infty c_2^{-1} \ell^{-\beta_2 k} \qquad\mbox{(by \eqref{regularGrowth})}
\end{align*}
\end{proof}

\subsection{\texorpdfstring{$\Eleq$} i is not possible for spaces with atoms} \label{EphiNotPossibleAppendix}

Recall that an exponentially-distributed random variable $\xi$ is said to be Exponential($\lambda$) (or have \textit{rate} $\lambda$) if its mean is $1/\lambda$. Recall that $M_A$ is the set of isolated atoms in $(M, d, \mu)$.

\begin{prop}\label{EphiNotPossible}
Suppose $(M, d, \mu, \mathcal{E}, \mathcal{F})$ satisfies Assumption \ref{BasicAssumptions}, $\phi$ is of regular growth, and $(\mathcal{E}, \mathcal{F})$ admits a jump kernel.
If $M_A$ is non-empty, then $\Eleq$ does not to hold.
\end{prop}

\begin{proof}

Recall how we defined the exit times $\tau_U$ in \eqref{tauFormula}.
Let $x_0$ be an element of $M_A$. If the process $X$ starts at $x_0$, it leaves $x_0$ with rate equal to $\int_{M \setminus \{x_0\}} J(x_0, y) \mu(dy)=v(x_0)$ (where $v(x_0)$ is as defined in \eqref{v(x)Formula}). Therefore, $\tau_{\{x_0\}}$ is Exponential($v(x_0)$)-distributed. By Lemma \ref{JumpsAndAtomsPreliminaryFactsLemma}(b),
\begin{equation*}
    \mathbb{E}^x \tau_{\{x_0\}} = \frac{1}{v(x_0)} \in (0, \infty].
\end{equation*}
Assume for the sake of contradiction that there exists a $C>0$ such that
\begin{equation} \label{assumeEphiForContradiction}
    \mathbb{E}^x \tau_{B(x, r)} \leq C \phi(r) \qquad\mbox{for all $x$}.
\end{equation}
If $0 < r \leq D_{x_0}$, then $B(x_0, r) = \{x_0\}$, so
\begin{equation}\label{smallREscape}
    \mathbb{E}^{x_0} \tau_{B(x_0, r)} = \frac{1}{v(x_0)}.
\end{equation}
By \eqref{assumeEphiForContradiction} and \eqref{smallREscape}, for all $r \in (0, D_{x_0}]$,
\begin{equation*}
    \frac{1}{v(x_0)} \leq C \phi(r).
\end{equation*}
This is a contradiction because $C\phi(r)$ approaches $0$ for small $r$, while $1/v(x_0)$ is constant and positive.
\end{proof}

\subsection{Auxiliary metric measure space satisfies Assumption \ref{BasicAssumptions}(a)} \label{AppendixAssumption1}

In order to prove Proposition \ref{volumeGrowthProp}(a), there are three things to check, namely

\begin{itemize}
    \item $(\Mhat, \dhat)$ is a locally compact separable metric space.
    \item  $\muhat$ is a positive Radon measure with full support.
    \item $\muhat(\Mhat) = \infty$.
\end{itemize}

We establish in part (b) of the following lemma that $(\Mhat, \dhat)$ is locally compact and separable. Part (a) is an intermediary result about compactness of a set $K \subseteq M$ carrying over to its preimage under the projection $\pi : \Mhat \to M$.

\begin{lemma} \label{MhatLocallyCompactSeparable}
Suppose $(M, d, \mu)$ satisfies Assumption \ref{BasicAssumptions}(a), $\phi$ is of regular growth, and $M=M_A \cup M_C$.

(a) If $K$ is a compact subset of $M$, then $\pi^{-1}(K)$ is a compact subset of $\Mhat$.

(b) The metric space $(\Mhat, \dhat)$ is locally compact and separable.
\end{lemma}

\begin{proof}
(a) We will use sequential compactness. Let $\{z_n\}$ be a sequence in $\pi^{-1}(K)$. Let $x_n := \pi(z_n)$ for all $n$. Since $\{x_n\}$ is a sequence in $K$, and $K$ is compact, there exists a subsequence $\{x_{n_k}\}$ that converges to some $x_0 \in K$. If $x_0 \in M_C$, then $\{z_{n_k}\}$ also converges to $x_0$, since $\dhat(z_{n_k}, x_0) = d(x_{n_k}, x_0) \to 0$. If $x_0 \in M_A$, then all but finitely many terms in $\{x_{n_k}\}$ are equal to $x_0$. Therefore, we can assume without loss of generality that $x_{n_k} = x_0$ for all $k$ (by replacing $\{n_k\}$ with a further subsequence if necessary). Let $w_k$ be the element of $W$ such that $z_{n_k} = (x_0, w_k)$. Since $W$ is compact, there exists a subsequence $\{w_{k_j}\}$ that converges to some $w_0 \in W$. Then $\{z_{n_{k_j}}\}$ converges to $(x_0, w_0) \in \widehat{K}$.

(b) Fix $z \in \Mhat$. Let $x=\pi(z)$. Since $(M,d)$ is locally compact, there exist an open $U$ and a compact $K$ such that $x \in U \subseteq K$.
Clearly, $\pi^{-1}(U)$ is open. By part (a), $\pi^{-1}(K)$ is compact. Therefore, $z$ belongs to an open subset of a compact set. Thus, $\Mhat$ is locally compact.

Since $(M, d)$ is separable, there exists a dense countable $E \subseteq M$. Then
\begin{equation*}
    (E \cap M_C) \cup \left\{ (x, w) : \mbox{$x \in D \cap M_A$, and $w(i)=0$ for all but finitely $i$} \right\}
\end{equation*}
is a dense countable subset of $\Mhat$. Thus, $\Mhat$ is separable.
\end{proof}

The following lemma establishes that $\muhat$ is a Radon measure with full support. Recall that a Radon measure is a measure that is finite on compact sets, and regular. Just like in Lemma \ref{MhatLocallyCompactSeparable}, part (a) is an intermediary result about compactness in the auxiliary space.

\begin{lemma}\label{MuhatRadonFullSupport}
Suppose $(M, d, \mu)$ satisfies Assumption \ref{BasicAssumptions}(a), $\phi$ is of regular growth, and $M=M_A \cup M_C$.

(a) If $S$ is a compact subset of $\Mhat$, then $\pi(S) = \{x \in M : \mbox{$W_x$ intersects $S$}\}$ is a compact subset of $M$.

(b) The measure $\muhat$ is finite on compact sets.

(c) The measure $\muhat$ is regular.

(d) The measure $\muhat$ has full support.

(By parts (b) and (c), $\muhat$ is a Radon measure.)
\end{lemma}

\begin{proof}

(a) Let $K = \pi(S)$, and let $\{U_\alpha\}_{\alpha\in A}$ be an open cover of $K$. For each $\alpha$, $\pi^{-1}(U_\alpha)$ is open, so $\{\pi^{-1}(U_\alpha)\}_{\alpha \in A}$ is an open cover of $S$. By the compactness of $S$, there exists a finite $B \subseteq A$ such that
\begin{equation}\label{RadonCover}
    \bigcup_{\alpha \in B} \pi^{-1}(U_\alpha) \supseteq S.
\end{equation}
Fix $x \in K$. By the definition of $K$, there exists a $z \in W_x \cap S$. By \eqref{RadonCover}, there exists an $\alpha \in B$ such that $z \in \pi^{-1}(U_\alpha)$, so $x \in U_\alpha$. Therefore, $\{U_\alpha\}_{\alpha \in B}$ covers $K$.
Since every open cover of $K$ has a finite subcover, $K$ is compact.

(b) Let $S$ be a compact subset of $\Mhat$. Let $K = \{ x \in M : W_x \cap S \neq \emptyset \}$. By part (a), $K$ is compact.
Note that $S \subseteq \pi^{-1}(K)$.
Since $\mu$ is finite on compact sets, $\muhat(S) \leq \muhat(\pi^{-1}(K)) = \mu(K) < \infty$.

(c) Let $S$ be a Borel subset of $\Mhat$. We will construct a sequence $\{K_n\}$ of compact subsets of $S$, and a sequence $\{U_n\}$ of open supersets of $S$, such that $\lim_{n\to\infty} \muhat(K_n) = \muhat(S) = \lim_{n\to\infty} \muhat(U_n)$.
Note that there can only be countably many elements of $M_A$ (otherwise $(M, d)$ could not be separable). Since $M_A$ is uncountable, there exists an $h : M_A \to (0, \infty)$ such that
\begin{equation} \label{RadonHSmall}
    \sum_{x \in M_A} h(x) \leq 1.
\end{equation}
There also exists an increasing sequence $\{I_n\}$ such that each $I_n$ is a finite subset of $M_A$, and $\bigcup_{n \in \mathbb{N}} I_n = M_A$. (This is sometimes called an \textit{exhaustion} of $M_A$.) Since $\muhat$ is a measure, it is continuous from below, so
\begin{equation}\label{RadonInContBelow}
    \lim_{n\to\infty} \muhat\left(S \cap \pi^{-1}(I_n)\right) = \muhat\left(S \cap \pi^{-1}(M_A)\right).
\end{equation}
For all $x \in M_A$, let $E_x := \{w \in W : (x, w) \in S\}$. Recall that $\nu$ is regular (since it is a Radon measure). Thus, we can choose a sequence $\{K^x_n\}$ of compact subsets of $W$, and a sequence $\{U^x_n\}$ of open subsets of $W$, such that $K^x_n \subseteq E_x \subseteq U^x_n$ for all $n$, and
\begin{equation} \label{RadonKxnUxn}
    \nu(E_x) - 2^{-n} \cdot \frac{h(x)}{\mu(x)} \leq \nu(K^x_n) \leq \nu(U^x_n) \leq \nu(E_x) + 2^{-n}\cdot \frac{h(x)}{\mu(x)} \qquad\mbox{for all $n$}.
\end{equation}
Let
\begin{equation*}
    K^A_n := \bigcup_{x \in I_n} \left\{ (x, w) : w \in K^x_n \right\} \qquad \mbox{and} \qquad U^A_n := K^A_n := \bigcup_{x \in M_A} \left\{ (x, w) : w \in U^x_n \right\}.
\end{equation*}
For all $n$,
\begin{align*}
    0 \leq \muhat\left(S \cap \pi^{-1}(I_n)\right) - \muhat(K^A_n) &= \sum_{x \in I_n} \Big[ \muhat \left( \{x\} \times E_x \right) - \left( \{x\} \times K^x_n \right) \Big]\\
    &= \sum_{x \in I_n} \mu(x) \left[ \nu(E_x) - \nu(K^x_n) \right]\\
    &\leq \sum_{x \in I_n} 2^{-n} h(x) &\mbox{(by \eqref{RadonKxnUxn})}\\
    &\leq 2^{-n} &\mbox{(by \eqref{RadonHSmall})}
\end{align*}
so
\begin{equation}\label{RadonLimitSameAsIn}
    \lim_{n\to\infty} \muhat(K^A_n) = \lim_{n\to\infty} \muhat\left(S \cap \pi^{-1}(I_n)\right).
\end{equation}
By \eqref{RadonInContBelow} and \eqref{RadonLimitSameAsIn},
\begin{equation}\label{RadonKAnLimit}
    \lim_{n\to\infty} \muhat(K^A_n) = \muhat\left(S \cap \pi^{-1}(M_A)\right).
\end{equation}
Similarly,
\begin{align*}
    0 \leq \muhat\left(U^A_n\right) - \muhat\left(S \cap \pi^{-1}(M_A)\right) &= \sum_{x \in M_A} \Big[ \muhat \left( \{x\} \times U^x_n \right) - \left( \{x\} \times E_x \right) \Big]\\
    &= \sum_{x \in M_A} \mu(x) \left[ \nu(U^x_n) - \nu(E_x) \right]\\
    &\leq \sum_{x \in M_A} 2^{-n} h(x) &\mbox{(by \eqref{RadonKxnUxn})}\\
    &\leq 2^{-n} &\mbox{(by \eqref{RadonHSmall})}
\end{align*}
so
\begin{equation}\label{RadonUAnLimit}
    \lim_{n\to\infty} \muhat(U^A_n) = \muhat\left(S \cap \pi^{-1}(M_A)\right).
\end{equation}
By the regularity of $\mu$, we can choose a sequence $\{K^C_n\}$ of compact subsets of $S \cap M_c$, and a sequence $\{U^C_n\}$ of open supersets of $S \cap M_C$, such that
\begin{equation} \label{RadonKCnUCn}
    \lim_{n\to\infty} \mu(K^C_n) = \mu(S \cap M_C) = \lim_{n\to\infty} \mu(U^C_n).
\end{equation}
For all $n$, let
\begin{equation*}
    K_n:= K^A_n \cup K^C_n \qquad\mbox{and}\qquad U_n := U^A_n \cup U^C_n.
\end{equation*}
By \eqref{RadonKAnLimit} and \eqref{RadonKCnUCn},
\begin{equation*}
    \lim_{n\to\infty} \muhat(K_n) = \lim_{n\to\infty} \muhat(K^A_n) + \lim_{n\to\infty} \muhat(K^C_n)
    = \muhat\left( S \cap \pi^{-1}(M_A) \right) + \muhat(S \cap M(C) = \muhat(S).
\end{equation*}
By \eqref{RadonUAnLimit} and \eqref{RadonKCnUCn},
\begin{equation*}
    \lim_{n\to\infty} \muhat(U_n) = \lim_{n\to\infty} \muhat(U^A_n) + \lim_{n\to\infty} \muhat(U^C_n)
    = \muhat\left( S \cap \pi^{-1}(M_A) \right) + \muhat(S \cap M(C) = \muhat(S).
\end{equation*}

(d) In order to show that $\muhat$ has full support, we will show that any non-empty open subset of $\Mhat$ has positive measure. Let $S$ be a non-empty open subset of $\Mhat$. Let $z$ be an element of $S$. If $z =(x, w)$ for some $x \in M_A$ and $w \in W$, then there must be an open $U \subseteq W$ such that $(x, w') \in S$ for all $w' \in U$. Therefore, $\muhat(S) \geq \mu(x) \nu(U) > 0$. On the other hand, if $z \in M_C$, choose an $r>0$ such that $\Bhat(z, r) \subseteq S$. since $r \geq 0 = D_z$, $\Bhat(z, r) = \pi^{-1}(B(z, r))$. Since $\mu$ has full support, $\muhat(S) \geq \Vhat(z, r) = V(z, r) >0$.
\end{proof}

Let us check that we have proven everything we needed to in this appendix. By Lemma \ref{MhatLocallyCompactSeparable}(b), $(\Mhat, \dhat)$ is a locally compact, separable metric space. By Lemma \ref{MuhatRadonFullSupport} parts (b) and (c), $\muhat$ is a positive Radon measure
on $\Mhat$. By Lemma \ref{MuhatRadonFullSupport} part (d), $\muhat$ has full support. Since $\mu(M)$ is infinite, so is $\muhat(\Mhat)$ by \eqref{MeasuresAgree}.

\subsection{\texorpdfstring{$\Dhat$}{TEXT} satisfies the conditions of Lemma \ref{constructRegularDirichletForms}} \label{DhatAppendix}

In this appendix, we prove Lemma \ref{DhatSatisfiesLemma}. In other words, we show that the set $\Dhat$ constructed in Definition \ref{DhatDef} has the following properties
\begin{itemize}
    \item $\Dhat$ is a subspace of $\Fhatmax \cap C_c(\Mhat)$.
    \item $\Dhat$ is Markovian (in the sense of \eqref{Markovian})
    \item $\Dhat$ is dense in $\Fhatmax \cap C_c(\Mhat)$ under the uniform norm.
\end{itemize}
(assuming all of the conditions that are necessary for the construction of $(\Mhat, \dhat, \muhat, \Ehat, \Fhat)$ to be well-defined). We prove each property separately.

\subsubsection*{$\Dhat$ is a subspace of $\Fhatmax \cap C_c(\Mhat)$}

\begin{proof}[Proof of Lemma \ref{DhatSatisfiesLemma} part (a)]
Recall that $\Dhat$ is defined as the linear span of $\left\{ g \circ \pi : g \in \mathcal{F} \cap C_c(M) \right\}$ and $\left\{ H_{x, h} : x \in M_A, h \in \Dtilde \right\}$, where $\pi : \Mhat\to M$ is the projection that maps every point in $W_x$ to $x$ for all $x$, and $H_{x, h}$ is as defined in \eqref{HxhFormula}. Thus, it is enough to show that $g \circ \pi \in \Fhat_\max \cap C_c(\Mhat)$ for all $g \in \mathcal{F} \cap C_c(M)$, and $H_{x, h} \in \Fhat_\max\cap C_c(\Mhat)$ for all $x \in M_A$, $h \in \Dtilde$.

Fix $g \in \mathcal{F}\cap C_c(M)$. It is clear that $g \circ \pi \in \Fmax$, since $\norm{g \circ \pi}_2 = \norm{g}_2 < \infty$ and $\Ehat(g\circ \pi) = \mathcal{E}(g) < \infty$. Furthermore, $g \circ \pi$ is compactly supported, since it is supported on $\pi^{-1}(\supp(g))$, which is compact by Lemma \ref{MhatLocallyCompactSeparable}(a). It remains to show $g \circ \pi$ is continuous. Note that $g \circ \pi$ is locally constant (and therefore continuous) on $\pi^{-1}(M_A)$. Fix $x \in M_C$ and $\varepsilon>0$. By the continuity of $g$, there exists a $\delta>0$ such that $|g(y)-g(x)| < \varepsilon$ for all $y \in B(x, \delta)$. By \eqref{BhatFormula}, $\Bhat(x, \delta) = \pi^{-1}(B(x, \delta))$, so
\begin{equation*}
    |(g \circ \pi)(z) - (g \circ \pi)(x)| = |g(\pi(z))-g(x)| < \varepsilon \qquad\mbox{for all $z \in \Bhat(x, \delta)$}.
\end{equation*}
Therefore, $g \circ \pi$ is continuous on $M_C$.

Fix $x \in M_A$ and $h \in \Dtilde$. We will show $H_{x,h} \in \Fmax\cap C_c(\Mhat)$. Firstly, $H_{x, h} \in \Fmax$, since $\norm{H_{x, h}}_2^2 = \mu(x) \norm{h}_2^2 < \infty$ and 
\begin{dmath*}
    \Ehat(H_{x, h}) = \int_{\Mhat \times \Mhat \setminus \diag_{\Mhat}} (H_{x, h}(z)-H_{x,h}(z'))^2 \Jhat(dz, dz')
    = 2 \int_{z \in W_x} \int_{\Mhat \setminus W_x} (H_{x,h}(z))^2 \Jhat(z, z') \muhat(dz') \muhat(dz)\\ + \int_{W_x \times W_x \setminus \diag_{W_x}} (H_{x, h}(z)-H_{x, h}(z'))^2 \Jhat(z, z') \muhat(dz') \muhat(dz)
    = 2 \mu(x) v(x) \int_{W} h^2 d\nu + \mu(x)^2 \int_{W\times W \setminus \diag_W} (h(w)-h(w'))^2 \cdot \frac{\Jtilde^{D_x}(w, w')}{\mu(x)} \nu(dw') \nu(dw)
    = {2\mu(x) v(x) \norm{h}_2^2 + \mu(x) \Etilde^{D_x}(h) < \infty}
\end{dmath*}
(where $v(x)$ is as defined in \eqref{v(x)Formula} and $\diag_{W_x} := \{ (z, z) : z \in W_x\}$).
Next, $H_{x, h}$ is compactly supported, because it is supported on $W_x$. Lastly, $H_{x, h}$ is continuous because $W_x$ is its own connected component of $\Mhat$ and $h$ is continuous.
\end{proof}

\subsubsection*{$\Dhat$ is Markovian}

\begin{proof}[Proof of Lemma \ref{DhatSatisfiesLemma}(b)]

Let us use the notation $u^* := 0 \vee (u \wedge 1)$, just like in \eqref{Markovian}).

Fix $f = g \circ \pi + \sum_{j=1}^N H_{x_j, h_j} \in \Dhat$. We will show that $f^* \in \Dhat$. 
Let $g' : M \to \mathbb{R}$ be the function
\begin{equation*}
    g'(x) = \left\{ \begin{matrix}
        0 &:& \mbox{if $x \in \{x_1, \dots, x_N\}$}\\
        g(x) &:& \mbox{if $x \notin \{x_1, \dots, x_N\}$}.
    \end{matrix}\right.
\end{equation*}
We can also write $g'$ as
\begin{equation}\label{MarkovProp4Gprime}
    g' = g - \sum_{j=1}^N g(x_j) \delta_{x_j}.
\end{equation}
For all $1 \leq j \leq N$, let $h'_J : W\to\mathbb{R}$ be the function
\begin{equation*}
    h'_j(w) = h_j(w) + g(x_j).
\end{equation*}
If $z = (x_j, w) \in W_{x_j}$ for some $1 \leq j \leq N$, then $f(z) = g(x_j) + h_j(w) = h'_j(w)$. If $z \in \Mhat \setminus \bigcup_{j=1}^N W_{x_j}$, then $f(z) = g(\pi(z)) = g'(\pi(z))$. Therefore, another way to express $f$ is
\begin{equation}\label{AlternateFinMarkovianProp}
    f = (g' \circ \pi) + \sum_{j=1}^N H_{x_j, h'_j}.
\end{equation}
If $1 \leq j \leq N$, for all $z = (x_j, w) \in W_{x_j}$, by \eqref{AlternateFinMarkovianProp},
\begin{equation}\label{MarkovProp4h}
    f^*(z) = 0 \vee (f(z) \wedge 1) = 0 \vee (h'_j(w) \wedge 1) = (h'_j)^*(w).
\end{equation}
For all $z \in \Mhat \setminus \bigcup_{j=1}^N W_{x_j}$, by \eqref{AlternateFinMarkovianProp},
\begin{equation}\label{MarkovProp4g}
    f^*(z) = 0 \vee (f(z) \wedge 1) = 0 \vee (g'(\pi(z)) \wedge 1) = (g')^*(\pi(z)).
\end{equation}
By \eqref{MarkovProp4h} and \eqref{MarkovProp4g},
\begin{equation}\label{MarjovProp4Fstar}
    f^* = ((g')^* \circ \pi) + \sum_{j=1}^N H_{x_j, (h'_j)^*}.
\end{equation}

We know by Lemma \ref{JumpsAndAtomsPreliminaryFactsLemma}(a) that $\delta_{x_j} \in \mathcal{F}$ for all $j$. Therefore, by \eqref{MarkovProp4Gprime}, $g' \in \mathcal{F}$. Since $\supp(g') \subseteq \supp(g)$, $g'$ is compactly supported. Since $\mathcal{F} \cap C_c(M)$ is Markovian, $(g')^* \in \mathcal{F}\cap C_c(M)$.
For all $1 \leq j \leq N$, since $h_j \in \Dtilde$ and $h'_j$ differs from $h_j$ by a constant, $h'_j \in \Dtilde$. Since $\Dtilde$ is Markovian, $(h'_j)^* \in \Dtilde$.
By \eqref{MarjovProp4Fstar}, $f^* \in \Dhat$.

Since $|f^*(z_1)-f^*(z_2)| \leq |f(z_1)-f(z_2)|$ for all $z_1$ and $z_2$, we also have $\Ehat(f^*) \leq \Ehat(f)$.
\end{proof}

\subsubsection*{$\Dhat$ is dense in $\Fhatmax \cap C_c(\Mhat)$ under the uniform norm}

\begin{proof}[Proof of Lemma \ref{DhatSatisfiesLemma}(c)]

Fix $f \in C_c(\Mhat)$ and $\varepsilon>0$. We will construct an $f' \in \Dhat$ such that $\norm{f'-f}_\infty < \varepsilon$.

Let $S = \supp(f)$, which is compact.
For all $y \in M_C$, by the continuity of $f$, there exists a $\delta_y > 0$ such that
\begin{equation}\label{DhatUnifDenseBydy}
    |f(z)-f(y)| < \frac{\varepsilon}{4} \qquad\mbox{for all $z \in B(y, \delta_y)$}.
\end{equation}
The collection
\begin{equation*}
    \{W_x : x \in M_A \} \cup \{B(y, \delta_y) : y \in M_C \}
\end{equation*}
is an open cover of $S$. By compactness, it has a finite subcover. In other words, there exist finite $\{x_j\}_{j=1}^n \subseteq M_A$ and $\{y_k\}_{k=1}^m \subseteq M_C$ such that
\begin{equation}\label{DhatUnifDenseCover}
    S \subseteq \bigcup_{j=1}^n W_{x_j} \cup \bigcup_{k=1}^m B(y_k, \delta_{y_k}).
\end{equation}
The main idea that drives this proof is that for all $x \in M \setminus \{x_1, \dots, x_n\}$, $f$ is approximately constant on $W_x$. Fix $x \in M \setminus \{x_1, \dots, x_n\}$ and $z_1, z_2 \in W_x$. By \eqref{DhatUnifDenseCover}, since $z_1$ does not belong to $W_{x_j}$ for any $1 \leq j \leq n$, $z_1$ must belong to $B(y, \delta_y)$ for some $y \in M_C$. Since $\dhat(z_2, y) = d(x, y) = \dhat(z_1, y)$, $z_2$ also belongs to $B(y, \delta_y)$. By the triangle inequality,
\begin{align}\label{DhatUnidDenseFapproxConst}
\begin{split}
    |f(z_1) - f(z_2)| &\leq |f(z_1) - f(y)| + |f(y) - f(z_2)|\\
    &\leq \frac{\varepsilon}{4} + \frac{\varepsilon}{4} \qquad\mbox{(by \eqref{DhatUnifDenseBydy})}\\
    &= \frac{\varepsilon}{2}.
\end{split}
\end{align}

For all $x \in M$, let $z_x$ be a representative of $W_x$. (If $x \in M_C$, then $z_x$ must of course be $x$.) Let $g: M \to \mathbb{R}$ be the function
\begin{equation}\label{DhatUnifDenseG}
    g(x) = f(z_x).
\end{equation}
If $\{x_n\}$ is a sequence in $M$ converging to $x \in M$, then $z_{x_n} \to z_x$ and $g(x_n) = f(z_{x_n}) \to f(z_x) = g(x)$ by the continuity of $f$. Thus, $g$ is continuous. Moreover, $g$ is supported on $\pi(S)$, which is compact by Lemma \ref{MuhatRadonFullSupport}(a), so $g \in C_c(M)$. Recall that $\mathcal{F} \cap C_c(M)$ is a core of $(\mathcal{E}, \mathcal{F})$, so $\mathcal{F} \cap C_c(M)$ is dense in $C_c(M)$ under the uniform norm. Thus, there exists a $g' \in \mathcal{F} \cap C_c(M)$ such that
\begin{equation}\label{DhatUnifDenseG'gClose}
    \norm{g'-g}_\infty < \frac{\varepsilon}{2}.
\end{equation}
For all $x \in M_A$, let $h_x :W \to \mathbb{R}$ be the function
\begin{equation}\label{DhatUnifDenseHx}
    h_x(w) = f(x, w) - f(z_x).
\end{equation}
Since $f$ is continuous on $W_x$, $h_j$ is continuous. Since $W$ is compact, $h_j \in C_c(W)$. Thus, there exists an $h'_j \in \Dtilde$ such that
\begin{equation}\label{DhatUnifDenseH'hClose}
    \norm{h'_j - h_j}_\infty < \frac{\varepsilon}{2}.
\end{equation}
By \eqref{DhatUnifDenseG} and \eqref{DhatUnifDenseHx},
\begin{equation*}
    f = (g \circ \pi) + \sum_{x \in M_A} H_{x_j, h_{x_j}}.
\end{equation*}

Let
\begin{equation*}
    f' := (g' \circ \pi) + \sum_{j=1}^n H_{x_j, h'_{x_j}}.
\end{equation*}
If $z = (x_j, w) \in W_{x_j}$ for some $1 \leq j \leq n$,
\begin{align}\label{DhatUnifDenseWxj}
\begin{split}
    |f'(x)-f(x)| &= |g'(x_j) + h'_{x_j}(w) - g(x_j) - h_{x_j}(w)|\\
    &\leq |g'(x_j)-g(x_j)| + |h'_{x_j}(w) - h_{x_j}(w)|\\
    &< \frac{\varepsilon}{2} + \frac{\varepsilon}{2} \qquad\mbox{(by \eqref{DhatUnifDenseG'gClose} and \eqref{DhatUnifDenseH'hClose})}\\
    &= \varepsilon.
\end{split}
\end{align}
If $z = (x, w) \in W_x$ for some $x \in M_A \setminus \{x_1, \dots, x_n\}$, then
\begin{align}\label{DhatUnifDenseWx}
\begin{split}
    |f'(z)-f(z)| &= |g'(x) - g(x) - h_x(w)|\\
    &= |g'(x) - g(x) - f(z) + f(z_x)| \qquad\mbox{(by \eqref{DhatUnifDenseHx})}\\
    &\leq |g'(x)-g(x)| + |f(z_x)-f(z)|\\
    &< \frac{\varepsilon}{2} + \frac{\varepsilon}{2} \qquad\mbox{(by \eqref{DhatUnifDenseG'gClose} and \eqref{DhatUnidDenseFapproxConst})}\\
    &= \varepsilon.
\end{split}
\end{align}
If $z \in M_C$, then by \eqref{DhatUnifDenseG'gClose},
\begin{equation}\label{DhatUnifDenseMc}
    |f'(z)-f(z)| = |g'(z)-g(z)| < \frac{\varepsilon}{2}.
\end{equation}
By \eqref{DhatUnifDenseWxj}, \eqref{DhatUnifDenseWx}, and \eqref{DhatUnifDenseMc}, $\norm{f'-f}_\infty < \varepsilon$.
\end{proof}

\subsection{\texorpdfstring{Jump kernels: proof of \eqref{JleqBothways}-\eqref{UjsBothways}}{TEXT}} \label{JumpKernelAppendix}

The fact that $\Jleq$ holds for the original space iff $\Jleq$ holds for the auxiliary space (and similarly, $\Jgeq$ holds for the original space iff $\Jgeq$ holds for the auxiliary space) is a direct consequence of \eqref{JhatFormula}, the formula that defined the jump kernel $\Jhat$ of $(\Mhat, \dhat, \muhat, \Ehat, \Fhat)$.

\begin{proof}[Proof of Proposition \ref{MasterProp}, implications \eqref{JleqBothways} and \eqref{JgeqBothways}]
Let
\begin{align*}
    c &= \essinf \left\{ J(x, y) V(x, d(x, y)) \phi(d(x, y)) : \mbox{$x, y \in M$ and $x\neq y$} \right\},\\
    C &= \esssup \left\{ J(x, y) V(x, d(x, y)) \phi(d(x, y)) : \mbox{$x, y \in M$ and $x\neq y$} \right\},\\
    \widehat{c} &= \essinf \left\{ \Jhat(z, z') \Vhat(z, \dhat(z, z')) \phi(\dhat(z, z')) : \mbox{$z, z' \in \Mhat$ and $z \neq z'$} \right\},\\
    \mbox{and} \qquad \widehat{C} &= \esssup \left\{ \Jhat(z, z') \Vhat(z, \dhat(z, z')) \phi(\dhat(z, z')) : \mbox{$z, z' \in \Mhat$ and $z \neq z'$} \right\}.
\end{align*}
By \eqref{JhatFormula}, for all distinct $z, z' \in \Mhat$,
\begin{equation*}
    \Jhat(z, z') \Vhat(z, \dhat(z, z')) \phi(\dhat(z, z'))  = \left\{ \begin{matrix}
        J(x, y) V(x, d(x, y)) \phi(d(x, y)) &:& \mbox{if $x = \pi(z) \neq \pi(z') = y$}\\
        \\
        1 &:& \mbox{if $\pi(z) = \pi(z')$}.
    \end{matrix}\right.
\end{equation*}
Thus,
\begin{equation*}
    \widehat{c}= c \wedge 1 \qquad\mbox{and}\qquad \widehat{C}= C \vee 1
\end{equation*}
so
\begin{equation*}
    \boxed{\mbox{$\Jleq$ for $(M, d, \mu, \mathcal{E}, \mathcal{F})$}} \quad\Longleftrightarrow\quad C<\infty \quad\Longleftrightarrow\quad \widehat{C}=C \vee 1 < \infty \quad\Longleftrightarrow\quad \boxed{\mbox{$\Jleq$ for $(\Mhat, \dhat, \muhat, \Ehat, \Fhat)$}}
\end{equation*}
and

\begin{equation*}
    \boxed{\mbox{$\Jgeq$ for $(M, d, \mu, \mathcal{E}, \mathcal{F})$}} \quad\Longleftrightarrow\quad c>0 \quad\Longleftrightarrow\quad \widehat{c}=c \wedge 1 > 0 \quad\Longleftrightarrow\quad \boxed{\mbox{$\Jgeq$ for $(\Mhat, \dhat, \muhat, \Ehat, \Fhat)$}}.
\end{equation*}
\end{proof}

We divide the proof of \eqref{UjsBothways} into two (one part for each direction). To show that $\UJS$ for the original space implies $\UJS$ for the auxiliary space, we must show that for almost all $z_0, z_1 \in \Mhat$, the value of $\Jhat(z_0, z_1)$ is at most some constant times $\int_{\Bhat(z_0, r)} \Jhat(z_0, z_1) \, \muhat(dz)$, for all $r \in \left(0, \frac{\dhat(z, z')}{2} \right]$. If $z_0$ and $z_1$ belong to $W_{x_0}$ for the same $x_0$, we show this directly. Otherwise, we use some of the equations we have established (\eqref{integralsAgree}, \eqref{BhatFormula}, \eqref{WxUltrametric}, \eqref{JhatFormula}) to express the quantities $\Jhat(z_0, z_1)$ and $\int_{\Bhat(z_0, r)} \Jhat(z_0, z_1) \, \muhat(dz)$ in terms of the original space, and use the hypothesis that $\UJS$ holds for $(M, d, \mu, \mathcal{E}, \mathcal{F})$. 

\begin{proof}[Proof of Proposition \ref{MasterProp}, implication \eqref{UjsBothways} ($\Longrightarrow$ direction)]

Suppose $\UJS$ holds for $(M, d, \mu, \mathcal{E}, \mathcal{F})$. There exists a $C>0$ and a set $E \subseteq M \times M \setminus \diag_M$ of full measure such that for all $(x_0, x_1) \in E$,
\begin{equation} \label{assumeUjs}
    J(x_0, x_1) \leq \frac{C}{V(x_0, r)} \int_{B(x_0, r)} J(x, x_1) \mu(dx) \qquad\mbox{for all $r \in \left( 0, \frac{d(x_0, x_1)}{2} \right]$}.
\end{equation}
By $\UJS$, $(\mathcal{E}, \mathcal{F})$ admits a jump kernel, so the auxiliary space $(\Mhat, \dhat, \muhat, \Ehat, \Fhat)$ is well-defined.
Let
\begin{equation*}
    \widehat{E} = \{(z_0, z_1) \in \Mhat \times \Mhat : (\pi(z_0), \pi(z_1)) \in E\}.
\end{equation*}

Fix $(z_0, z_1) \in \widehat{E}$. Let $x_0=\pi(z_0)$ and $x_1 = \pi(z_1)$. Note that $x_0 \neq x_1$, since $(x_0, x_1) \in E \subseteq M\times M\setminus\diag_M$. Thus,
\begin{equation} \label{JsEqualUjs}
    \Jhat(z_0, z_1) = J(x_0, x_1).
\end{equation}
Fix $r$ such that $0 < r \leq \dhat(z_0, z_1)/2 = d(x_0, x_1)/2$. Suppose $r \geq D_{x_0}$. Note that if $z \in \Bhat(z_0, r)$, then $z \notin W_{x_1}$ (because $\dhat(z, z_0)$ is less than $r<d(x_1, x_0) = d(z_1, z_0)$), and thus $\Jhat(z, z_1) = J(\pi(z), x_1)$. This gives us
\begin{dmath} \label{UjsForwardBigr}
    J(x_0, x_1) \leq \frac{C}{V(x_0, r)} \int_{B(x_0, r)} J(x, x_1) \mu(dx)\qquad\mbox{(by \eqref{assumeUjs})}
    = \frac{C}{\Vhat(z_0, r)} \int_{\Bhat(z_0, r)} \Jhat(z, z_1) \muhat(dz) \qquad\mbox{(by \eqref{BhatFormula}, \eqref{JhatFormula}, and \eqref{integralsAgree})}.
\end{dmath}
On the other hand, if $r < D_{x_0}$, then $\Bhat(z_0, r) \subseteq W_{x_0}$, so $\Jhat(z, z_1) = J(x_0, x_1)$ for all $z \in \Bhat(z_0, r)$. This means that
\begin{equation} \label{UjsForwardSmallR}
    \frac{1}{\Vhat(z_0, r)} \int_{\Bhat(z_0, r)} \Jhat(z, z_1) \muhat(dz) = J(x_0, x_1).
\end{equation}
By \eqref{JsEqualUjs}, \eqref{UjsForwardBigr}, and \eqref{UjsForwardSmallR},
\begin{equation} \label{UjsOnEhat}
    \Jhat(z_0, z_1) = J(x_0, x_1) \leq \frac{C \vee 1}{\Vhat(z_0, r)} \int_{\Bhat(z_0, r)} \Jhat(z, z_1) \muhat(dz) \qquad\mbox{for all $(z_0, z_1) \in \widehat{E}$ and $0 < r \leq \frac{\dhat(z_0, z_1)}{2}$}.
\end{equation}
Now let $z_0$ and $z_1$ be distinct elements of $W_{x}$ for some $x \in M_A$, and fix $r$ such that $0 < r \leq \dhat(z, z')$. By \eqref{WxUltrametric}, the restriction of $\dhat$ to $W_{x_0}$ is ultrametric, so
\begin{equation} \label{UjsDhatzDhatz0}
    \dhat(z, z_1) \leq \max \left\{ \dhat(z, z_0), \dhat(z_0, z_1) \right\} = \dhat(z_0, z_1)\qquad\mbox{for all $z \in \Bhat(z_0, r)$}.
\end{equation}
By \eqref{JhatFormula}, the restriction of $z \mapsto \Jhat(z, z_1)$ to $W_x$ is a decreasing function of $\dhat(z, z_1)$, so by \eqref{UjsDhatzDhatz0},
\begin{equation*}
    \Jhat(z_0, z_1) \leq \Jhat(z, z_1) \qquad\mbox{for all $z \in \Bhat(z_0, r)$}.
\end{equation*}
Thus,
\begin{equation} \label{UjsOnDiaghat}
    \Jhat(z_0, z_1) \leq \frac{1}{\Vhat(z_0, r)} \int_{\Bhat(z_0, r)} \Jhat(z, z_1) \muhat(dz) \qquad\mbox{whenever $x \in M_A$ and $z_0, z_1 \in W_{x_0}$ are distinct}.
\end{equation}
Note that
\begin{equation} \label{setForUjsForward}
    \widehat{E} \cup \bigcup_{x \in M_A} \left\{ (z_0, z_1) : \mbox{$z_0, z_1 \in W_x$ and $z_0 \neq z_1$} \right\}
\end{equation}
is a subset of $\Mhat \times \Mhat \setminus \{ (z, z) : z \in \Mhat \}$ with full measure. By \eqref{UjsOnEhat} and \eqref{UjsOnDiaghat},
\begin{equation*}
    \Jhat(z_0, z_1)  \leq \frac{C \vee 1}{\Vhat(z_0, r)} \int_{\Bhat(z_0, r)} \Jhat(z, z_1) \muhat(dz) \qquad\mbox{for all $(z_0, z_1)$ in the set described in \eqref{setForUjsForward}}.
\end{equation*}
Thus, $\UJS$ holds for $(\Mhat, \dhat, \muhat, \Ehat, \Fhat)$.
\end{proof}

Conversely, in order to show that $\UJS$ for the auxiliary space implies $\UJS$ for the original space, we employ the same strategy: to show that $J(x_0, x_1)$ is at most some constant times $\int_{B(x_0, r)} J(x, x_1) \, \mu(dx)$, we use \eqref{VhatFormula} and \eqref{JhatFormula} to express $J(x_0, x_1)$ and $\int_{B(x_0, r)} J(x, x_1) \, \mu(dx)$ in terms of the auxiliary space, and then use the hypothesis that $\UJS$ holds for $(\Mhat, \dhat, \muhat, \Ehat, \Fhat)$.

\begin{proof}[Proof of Proposition \ref{MasterProp}, implication \eqref{UjsBothways} ($\Longleftarrow$ direction)]

If $\UJS$ holds for the auxiliary space, there exists an $S \subseteq \Mhat\times\Mhat\setminus\{(z, z) : z \in \Mhat\}$ of full measure such that
\begin{equation} \label{assumeUjsBackwards}
    \Jhat(z_0, z_1) \leq \frac{C}{\Vhat(z_0, r)} \Jhat(z, z_1) \muhat(dz) \qquad\mbox{for all $(z_0, z_1) \in S$}.
\end{equation}
Let
\begin{equation*}
    E = \left\{ (x_0, x_1) \in M \times M \setminus\diag_M: (W_{x_0} \times W_{x_1}) \cap S \neq \emptyset \right\}.
\end{equation*}
Note that $E$ has full measure.

Fix $(x_0, x_1) \in E$ and $r$ such that $0 < r \leq d(x_0, x_1)/2$. Let $(z_0, z_1)$ be a representative of $(W_{x_0} \times W_{x_1}) \cap S$.
If $r \geq D_{x_0}$,
\begin{align*}
    J(x_0, x_1) &= \Jhat(z_0, z_1)\\
    & \leq \frac{C}{\Vhat(z_0, r)} \int_{\Bhat(z_0, r)} \Jhat(z, z_1) \muhat(dz) &\mbox{(by \eqref{assumeUjsBackwards})}\\
    &= \frac{C}{V(x_0, r)} \int_{B(x_0, r)} J(x, x_1) \mu(dx). &\quad\mbox{(by \eqref{VhatFormula} and \eqref{JhatFormula})}.
\end{align*}
If $r \leq D_{x_0}$, then $B(x_0, r) = \{x_0\}$, so
\begin{equation*}
    \frac{1}{V(x_0, r)} \int_{B(x_0, r)} J(x, x_1) \mu(dx) = \frac{1}{\mu(x_0)} \cdot \mu(x_0) J(x_0, x_1)=J(x_0, x_1).
\end{equation*}
In either case,
\begin{equation*}
    J(x_0, x_1) \leq \frac{C \vee 1}{V(x_0, r)} \int_{B(x_0, r)} J(x, x_1) \mu(dx) \qquad\mbox{for all $(x_0, x_1)\in E$}
\end{equation*}
so $\UJS$ holds for the original space.
\end{proof}

\subsection{Parabolic H\"{o}lder regularity: proof of \eqref{PhrBackward}} \label{PhrAppendix}

Here we show that if $\PHR$ holds for the auxiliary space, then $\PHR$ holds for the original space.

\begin{proof}[Proof of Proposition \ref{MasterProp}, implication \eqref{PhrBackward}]
Suppose $\PHR$ holds for $(\Mhat, \dhat, \muhat, \Ehat, \Fhat)$, and let $c$, $\theta$, and $\varepsilon$ be the constants witnessing this.
Fix $x_0 \in M$, $t_0 \geq 0$, and $r>0$. Let $u(t, x)$ be a bounded measurable function that is caloric on $Q(t_0, x_0, \phi(r), r) = (t_0, t_0+\phi(r)) \times B(x_0, r)$.
We would like to show that there is a properly exceptional set $\mathcal{N}_u \supseteq \mathcal{N}$ such that \eqref{PhrFormula} holds for all $s, t \in (t_0, t_0+\phi(r))$ and $x, y \in B(x_0, \varepsilon r) \setminus \mathcal{N}_u$.

Let $\widehat{u}(t, z) := u(t, \pi(z))$ for all $t \geq 0$ and $z \in \Mhat$. Let $z_0$ be a representative of $W_{x_0}$. It is clear from \eqref{couple} and the definition of caloric that $\widehat{u}$ is caloric on
\begin{equation*}
    \widehat{Q}(t_0, z_0, \phi(r), r) := (t_0, t_0+\phi(r)) \times \Bhat(z_0, r).
\end{equation*}
(In fact, $\widehat{u}$ is caloric on $(t_0, t_0+\phi(r)) \times \Bhat(z_0, \max\{r, D_x\}) \supseteq (t_0, t_0+\phi(r)) \times \Bhat(z_0, r)$.)
Since $\PHR$ holds on the auxiliary space, there exists a properly exception $\widehat{\mathcal{N}}_{\widehat{u}} \supseteq \widehat{\mathcal{N}}$ such that
\begin{equation} \label{PhrOnAuxiliary}
    |\widehat{u}(s, z) - \widehat{u}(t, z')| \leq c \left( \frac{\phi^{-1}(|s-t|) + \dhat(z, z')}{r} \right)^\theta \esssup_{[t_0, t_0+\phi(r)] \times \Mhat} |\widehat{u}|
\end{equation}
for all $s, t \in (t_0, t_0+\phi(r))$ and $z, z' \in \Bhat(z_0, \varepsilon r) \setminus \widehat{\mathcal{N}}_{\widehat{u}}$.

Let $\mathcal{N}_u := \{ x \in M : W_x \subseteq \widehat{N}_{\widehat{u}} \}$. Now fix $s, t \in (t_0, t_0+\phi(r))$ and $x, y \in B(x_0, \varepsilon r) \setminus \mathcal{N}_u$. Let $z_1$ be a representative of $W_x \cap \Bhat(z_0, \varepsilon r) \setminus \widehat{\mathcal{N}}_{\widehat{u}}$. (One can check that such a representative exists by considering the cases:
\begin{itemize}
    \item $x = x_0 \in M_A$.
    \item $x = x_0 \in M_C\setminus \mathcal{N}_u$.
    \item $x \neq x_0$.
\end{itemize}
In the first case, $W_x \cap \Bhat(z_0, r)$ has positive measure, so it must have an element that does not belong to $\widehat{\mathcal{N}}_{\widehat{u}}$.
In the second and third cases, we can let $z_1$ be any element of $W_x \setminus \widehat{\mathcal{N}}_{\widehat{u}}$, which is non-empty since $x \in \mathcal{N}_u$, and then we will have $\dhat(z_1, z_0) = d(x, y) < \varepsilon r$.)

Similarly, let $z_2$ be a representative of $W_y \cap \Bhat(z_0, \varepsilon r) \setminus \widehat{\mathcal{N}}_{\widehat{u}}$.
These representatives $z_1$ and $z_2$ can be chosen such that $z_1=z_2$ if $x=y$. This way, $\dhat(z_1, z_2) = d(x, y)$.
Then
\begin{align*}
    |u(s, x) - u(t, y)| &= |\widehat{u}(s, z_1) - \widehat{u}(t, z_2)|\\
    &\leq c \left( \frac{\phi^{-1}(|s-t|) + \dhat(z_1, z_2)}{r} \right)^\theta \esssup_{[t_0, t_0+\phi(r)] \times \Mhat} |\widehat{u}| \qquad\mbox{(by \eqref{PhrOnAuxiliary})}\\
    &= c \left( \frac{\phi^{-1}(|s-t|) + d(x, y)}{r} \right)^\theta \esssup_{[t_0, t_0+\phi(r)] \times M} |u|.
\end{align*}
\end{proof}

\subsection{Cut-off Sobolev for auxiliary space implies cut-off Sobolev for original space} \label{CsjBackwardsAppendix}

Here we prove Proposition \ref{CsjBackwardsProp}: if $(M, d, \mu, \mathcal{E}, \mathcal{F})$ satisfies $\VD$ and $\Jleq$, then $\CSJ$ for the auxiliary space implies $\CSJ$ for the original space.
Let $C_D$, $C_J$, and $C_{\mathcal{J}}$ be as in Section \ref{cutoffSobolevSection} (see equations \eqref{FirstVdJphi}-\eqref{CmathcaljforCSJ}).

Assume $\CSJ$ holds on the auxiliary space, and let $C_0$, $C_1$, and $C_2$ be the constants witnessing this.

Fix $x_0 \in M$, $R \geq r>0$, and $f \in \mathcal{F}$. Let $B_1$, $B_2$, $B_3$, $U$, and $U^*$ be as in \eqref{B1B2B3}. We will like to show, using the assumption that $\CSJ$ holds for the auxiliary space, that there exists a $\varphi \in \cutoff(B_1, B_2)$ such that
\begin{equation} \label{CsjOriginalGoalAppendix}
    \int_{B_3} f^2 \, d\Gamma(\varphi) \leq C_1 \int_{U \times U^*} (f(x)-f(y))^2 \, J(dx, dy) + \frac{C_3}{\phi(r)} \int_{B_3} f^2 \, d\mu
\end{equation}
where $C_3$ is a constant that depends on $C_D$, $C_J$, $C_{\mathcal{J}}$, $C_0$, $C_1$, and $C_2$ (but not on $z_0$, $R$, $r$, or $f$) and $C_1$ is the familiar constant from $\CSJ$ on the auxiliary space.

Our proof is very similar to what we did in Section \ref{cutoffSobolevSection}. We let $z_0$ be a representative from $W_{x_0}$, and let $S_1$, $S_2$, $S_3$, $V$, and $V^*$ be as in \eqref{S1S2S3}.
For $R \leq D_{x_0}$, we prove in a simple lemma (Lemma \ref{CsjBackwardsSmallRLemma}) 
that $\int_{S_3} f^2 \, d\Gamma(\delta_{x_0})$ is at most some constant times $\frac{1}{\phi(r)}\int_{B_3} f^2 \, d\mu$, so we can take $\varphi$ to be the indicator $\delta_{x_0}$.
For $R \geq D_{x_0}$, we prove a lemma (Lemma \ref{CsjBackwardsBigRLemma}) that compares every term in \eqref{CsjOriginalGoalAppendix} to an analogous term in the auxiliary space. We find that every term compares in the direction we need it to, so putting these lemmas together (and using the hypothesis that $\CSJ$ holds for the auxiliary space), we obtain \eqref{CsjOriginalGoalAppendix}.

\begin{lemma} \label{CsjBackwardsSmallRLemma}

Suppose $(M, d, \mu, \mathcal{E}, \mathcal{F})$ satisfies Assumption \ref{BasicAssumptions}, $\phi$ is of regular growth, $(\mathcal{E}, \mathcal{F})$ admits a jump kernel, and $M=M_A \cup M_C$.
Fix $x_0 \in M$, $R \geq r>0$, and $f \in \mathcal{F}$. Let $B_1$, $B_2$, $B_3$, $U$, and $U^*$ be as in \eqref{B1B2B3}.
If $\VD$ and $\Jleq$ hold on the original space, and $R \leq D_{x_0}$, then
\begin{equation*}
    \int_{B_3} f^2 \, d\Gamma(\delta_{x_0}) \leq \frac{C}{\phi(R)} \int_{B_3} f^2 \, d\mu
\end{equation*}
where $C$ is a constant that depends only on $C_D$, $C_J$, and $C_{\mathcal{J}}$.

\end{lemma}

\begin{proof}
By \eqref{energyIntegralIndicator} and the definition of $v(x_0)$ (see \eqref{v(x)Formula}),
\begin{equation}\label{741}
    \int_{B_3} f^2 \, d\Gamma(\delta_{x_0}) = \mu(x_0) f^2(x_0) v(x_0) + \mu(x_0) \int_{B_3 \setminus \{x_0\}} f^2(y) J(x_0, y) \mu(dy).
\end{equation}
By \eqref{CmathcaljforCSJ},
\begin{equation}\label{742}
    v(x_0) \leq \frac{C_{\mathcal{J}}}{\phi(D_{x_0})}.
\end{equation}
For all $y \in B_3 \setminus \{x_0\}$, by $\Jleq$,
\begin{align}\label{743}
\begin{split}
    J(x_0, y) &\leq \frac{C_J}{V(x_0, d(x_0, y))\phi(d(x_0, y))}\\
    &\leq \frac{C_J}{V(x_0, D_{x_0}) \phi(D_{x_0})} \qquad\mbox{(since $d(x_0, y) \geq D_{x_0}$)}\\
    &= \frac{C_J}{\mu(x_0) D_{x_0}}.
\end{split}
\end{align}
By plugging \eqref{742} and \eqref{743} into \eqref{741},
\begin{align*}
    \int_{B_3} f^2 \, d\Gamma(\delta_{x_0}) &\leq \frac{C_{\mathcal{J}}}{\phi(D_{x_0}} \mu(x_0) f^2(x_0) + \frac{C_J}{\phi(D_{x_0})} \int_{B_3 \setminus \{x_0\}} f^2(y) \mu(dy)\\
    &= \frac{C_{\mathcal{J}}}{\phi(D_{x_0}} \int_{\{x_0\}} f^2 \, d\mu + \frac{C_J}{\phi(D_{x_0})} \int_{B_3 \setminus \{x_0\}} f^2 \, d\mu\\
    &\leq \frac{\max\{C_{\mathcal{J}}, C_J \}}{\phi(D_{x_0})} \int_{B_3} f^2 \, d\mu.
\end{align*}
\end{proof}

\begin{lemma} \label{CsjBackwardsBigRLemma}
Suppose $(M, d, \mu, \mathcal{E}, \mathcal{F})$ satisfies Assumption \ref{BasicAssumptions}, $\phi$ is of regular growth, $(\mathcal{E}, \mathcal{F})$ admits a jump kernel, and $M=M_A \cup M_C$.
Fix $x_0 \in M$, $R \geq r>0$, and $f \in \mathcal{F}$. Let $B_1$, $B_2$, $B_3$, $U$, and $U^*$ be as in \eqref{B1B2B3}, let $z_0$ be a representative of $W_{x_0}$, and let $S_1$, $S_2$, $S_3$, $V$, and $V^*$ be as in \eqref{S1S2S3}. Let $\psi$ be a function in $\cutoff(S_1, S_2)$.
If $\VD$ and $\Jleq$ hold on the original space, and $R \leq D_{x_0}$, then

(a)
\begin{equation*}
    \int_{S_3}  (f \circ \pi)^2 \, d\Gamma(\psi) \geq \int_{B_3} f^2 \, d\Gamma(\psi_\mean).
\end{equation*}

(b) There exists a $C>0$ (which depends only on $C_D$, $C_J$, and $C_{\mathcal{J}}$) such that
\begin{equation*}
    \int_{V \times V^*} \left((f \circ \pi)(z) - (f \circ \pi)(z') \right)^2 \, \Jhat(dz, dz') \leq \int_{U \times U^*} (f(x)-f(y))^2 J(dx, dy) + \frac{C}{\phi(r)} \int_{B_3} f^2 \, d\mu.
\end{equation*}

(c)
\begin{equation*}
    \int_{S_3} (f \circ \pi)^2 \, d\muhat = \int_{B_3} f^2 \, d\mu.
\end{equation*}

\end{lemma}

\begin{proof}
(a) Write out $\int_{B_3} f^2 \, d\Gamma(\psi_\mean)$ using the definition of the carr\'e du Champ operator and apply Lemma \ref{differencesSquaredLem}.

(b)
Since $R+(1+C_0)r \geq R+r \geq R \geq D_{x_0}$,
\begin{equation} \label{V*formula}
    V = \pi^{-1}(U) \qquad\mbox{and} \qquad V^* = \pi^{-1}(U^*) \cup \left\{ z \in W_{x_0} : \dhat(z_0, z) \geq R-C_0 r\right\}.
\end{equation}
Let $$\alpha := \int_{U \times U^*} (f(x)-f(y))^2 J(dx, dy)$$ and $$\beta:= \int_{V \times V^*} \left( (f\circ \pi)(z) - (f\circ \pi)(z') \right)^2 \Jhat(dz, dz').$$
By \eqref{V*formula},
\begin{align}\label{762}
\begin{split}
    \beta &= \left( \int_{V \times \pi^{-1}(U^*)} + \int_{V \times \left(W_{x_0} \setminus \Bhat(z_0, R - C_0 r)\right)} \right) \Big( (f\circ \pi)(z) - (f\circ \pi)(z') \Big)^2 \Jhat(dz, dz')\\
    &\leq \left( \int_{V \times \pi^{-1}(U^*)} + \int_{V \times W_{x_0} } \right) \Big( (f\circ \pi)(z) - (f\circ \pi)(z') \Big)^2 \Jhat(dz, dz')\\
    &= \int_{U \times U^*} (f(x)-f(y))^2 J(dx, dy) + \int_{U \times \{x_0\}} (f(x)-f(y))^2 J(dx, dy)\\
    &= \alpha + \mu(x_0) \int_U (f(x)-f(x_0))^2 J(dx, dy).
\end{split}
\end{align}
For all $x \in U$, by the inequality $(a-b)^2 \leq 2a^2 + 2b^2$, we have
\begin{equation}\label{763}
    (f(x)-f(x_0))^2 \leq 2f^2(x) + 2f^2(x_0).
\end{equation}
For all $x \in U$, $d(x_0, x) \geq R$, so by $\Jleq$,
\begin{equation}\label{764}
    J(x_0, x) \leq \frac{C_J}{V(x_0, R) \phi(R)}.
\end{equation}
By \eqref{762}, \eqref{763}, and \eqref{764},
\begin{align}\label{765}
\begin{split}
    \beta &\leq \alpha + \frac{2C_J \mu(x_0)}{V(x_0, R) \phi(R)} \left( \int_U f^2 \, d\mu + \mu(U) f^2(x_0) \right) \\
    &\leq \alpha + \frac{2C_J}{\phi(R)} \left( \frac{\mu(x_0)}{V(x_0, R)} \int_U f^2 \, d\mu + \frac{\mu(U)}{V(x_0, R)} \int_{\{x_0\}} f^2 \, d\mu \right).
\end{split}
\end{align}
Obviously, $\mu(x_0) / V(x_0, R) \leq 1$. By $\VD$,
\begin{equation*}
    \frac{\mu(U)}{V(x_0, R)} = \frac{V(x_0, R+r)-V(x_0, R)}{V(x_0, R)} \leq \frac{V(x_0, 2R)-V(x_0, R)}{V(x_0, R)} \leq C_D - 1.
\end{equation*}
Thus, \eqref{765} becomes
\begin{align}\label{766}
\begin{split}
    \beta &\leq \alpha + \frac{2C_J}{\phi(R)} \left( \int_U f^2 \, d\mu + (C_D-1) \int_{\{x_0\}} f^2 \, d\mu \right) \\
    &\leq \alpha + \frac{2C_J}{\phi(R)} \left( \int_{B_3} f^2 \, d\mu + (C_D-1) \int_{B_3} f^2 \, d\mu \right)\\
    &= \alpha + \frac{2C_J C_D}{\phi(R)} \int_{B_3} f^2 \, d\mu\\
    &\leq \alpha + \frac{2C_J C_D}{\phi(r)} \int_{B_3} f^2 \, d\mu \qquad\mbox{(since $R \geq r$)}.
\end{split}
\end{align}

(c) follows from \eqref{integralsAgree}.
\end{proof}

Proving Proposition \ref{CsjBackwardsProp} is a simple matter of putting Lemmas \ref{CsjBackwardsSmallRLemma} and \ref{CsjBackwardsBigRLemma} together, and using the hypothesis that $\CSJ$ holds on the auxiliary space.

\begin{proof}[Proof of Proposition \ref{CsjBackwardsProp}]
If $\CSJ$ holds on the auxiliary space, there exists an $S \subseteq \Mhat$ of full measure, and constants $C_0, C_1, C_2$ such that for all $z_0 \in S$, for all $R \geq r > 0$, and for all $g \in \Fhat$, there exists a $\psi \in \cutoff\left(\Bhat(z_0, R), \Bhat(z_0, R+r)\right)$ such that
\begin{equation*}
    \int_{\Bhat(z_0, R+(1+C_0)r)} g^2 \, d\Gamma(\psi) \leq C_1 \int_{V \times V^*} (g(z)-g(z'))^2 \Jhat(dz, dz') + \frac{C_2}{\phi(r)} \int_{\Bhat(z_0, R+(1+C_0)r)} g^2 \, d\muhat
\end{equation*}
where 
\begin{equation*}
    V := \Bhat(z_0, R+r) \setminus \Bhat(z_0, R)
\end{equation*}
and
\begin{equation*}
    V^* := \Bhat(z_0, R+(1+C_0)r) \setminus \Bhat(z_0, R-C_0 r).
\end{equation*}
Let $E$ be the set of $x \in M$ such that $W_x$ intersects $S$.
Clearly, $E$ has full measure.

Fix $x_0 \in E$, $R \geq r>0$, and $f \in \mathcal{F}$. Let $B_1$, $B_2$, $B_3$, $U$, and $U^*$ be as in \eqref{B1B2B3}. Let $z_0$ be an element of $W_{x_0} \cap S$ (which we know is non-empty since $x_0 \in E$). Let $S_1$, $S_2$, and $S_3$ be as in \eqref{S1S2S3}.

If $R \leq D_{x_0}$, let $\varphi = \delta_{x_0}$. By Lemma \ref{JumpsAndAtomsPreliminaryFactsLemma}(a), $\delta_{x_0} \in \mathcal{F}$. Since $R$ is less than or equal to $D_{x_0}$, $\delta_{x_0} = 1$ in $B_1$. Clearly, $\delta_{x_0}=0$ on $B_2^c$. Therefore, $\varphi = \delta_{x_0} \in \cutoff(B_1, B_2)$. By Lemma \ref{CsjBackwardsSmallRLemma},
\begin{equation*}
    \int_{B_3} f^2 \, d\Gamma(\varphi) \leq \frac{C}{\phi(R)} \int_{B_3} f^2 \, d\mu \leq \frac{C}{\phi(r)} \int_{B_3} f^2 \, d\mu
\end{equation*}
(where $C$ is the constant from Lemma \ref{CsjBackwardsSmallRLemma}).

Suppose $R \geq D_{x_0}$. By Proposition \ref{DomainProp}(a), the function $f \circ \pi$ belongs to $\Fhat$. Therefore, since $\CSJ$ holds for the auxiliary space, there exists a $\psi \in \cutoff(S_1, S_2)$ such that
\begin{equation} \label{CsjHypothesisAppendix}
    \int_{S_3} (f \circ \pi)^2 \, d\Gamma(\psi) \leq C_1 \int_{V \times V^*} \left( (f \circ \pi)(z) - (f \circ \pi)(z') \right)^2 \Jhat(dz, dz') + \frac{C_2}{\phi(r)} \int_{S_3} (f \circ \pi)^2 \, d\muhat.
\end{equation}
Let $\varphi = \psi_\mean$. By Proposition \ref{DomainProp}(b), $\psi_\mean \in \mathcal{F}$. Since $R \geq D_{x_0}$ and $\psi \in \cutoff(S_1, S_2)$, we have $\psi_\mean=1$ on $B_1$ and $\psi_\mean=0$ on $B_2^c$. We also have $0 \leq \psi_\mean \leq 1$ everywhere, since $\psi_\mean(x)$ is an average of values between $0$ and $1$ for all $x$. Therefore, $\varphi = \psi_\mean \in \cutoff(B_1, B_2)$. By \eqref{CsjHypothesisAppendix} and all three parts of Lemma \ref{CsjBackwardsBigRLemma},
\begin{align*}
    \int_{B_3} f^2 \, d\Gamma(\varphi) &\leq \int_{S_3} (f \circ \pi)^2 \, d\Gamma(\psi)\\
    &\leq C_1 \int_{V \times V^*} \left( (f \circ \pi)(z) - (f \circ \pi)(z') \right)^2 \Jhat(dz, dz') + \frac{C_2}{\phi(r)} \int_{S_3} (f \circ \pi)^2 \, d\muhat\\
    &\leq C_1 \left( \int_{U \times U^*} (f(x)-f(y))^2 J(dx, dy) + \frac{C}{\phi(r)} \int_{B_3} f^2 \, d\mu \right) + \frac{C_2}{\phi(r)} \int_{B_3} f^2 \, d\mu\\
    &= C_1 \int_{U \times U^*} (f(x)-f(y))^2 J(dx, dy) + \frac{C_3}{\phi(r)} \int_{B_3} f^2 \, d\mu
\end{align*}

\end{proof}

As with \eqref{CsjForwards}, an almost identical argument can be used to show that if $(M, d, \mu, \mathcal{E}, \mathcal{F})$ satisfies Assumption \ref{BasicAssumptions}, $\phi$ is of regular growth, $(\mathcal{E}, \mathcal{F})$ admits a jump kernel, $M=M_A \cup M_C$, and $\VD$ and $\Jleq$ hold for $(M, d, \mu, \mathcal{E}, \mathcal{F})$, then $$ \boxed{\mbox{$\SCSJ$ for $(\Mhat, \dhat, \muhat, \Ehat, \Fhat)$}} \Longrightarrow \boxed{\mbox{$\SCSJ$ for $(M, d, \mu, \mathcal{E}, \mathcal{F})$}}.$$
}{}

\end{document}